\documentclass[11pt]{amsart}
\usepackage{amsmath,amssymb,amscd,amsthm,stmaryrd}

\usepackage{amssymb}
\usepackage{graphicx}
\usepackage{amsfonts}
\usepackage{bbm}
\usepackage{mathrsfs}
\usepackage[colorlinks=true, pdfstartview=FitV, linkcolor=blue, citecolor=blue, urlcolor=blue]{hyperref}
\usepackage{color}
\usepackage{graphicx,epstopdf,verbatim}
\usepackage[cal=boondoxo]{mathalfa}

\usepackage[OT2,T1]{fontenc}
\DeclareSymbolFont{cyrletters}{OT2}{wncyr}{m}{n}
\DeclareMathSymbol{\Sha}{\mathalpha}{cyrletters}{"58}

\allowdisplaybreaks

\newtheorem{theorem}{Theorem}[section]
\newtheorem{lemma}[theorem]{Lemma}

\newtheorem{example}[theorem]{Example}
\newtheorem{definition}[theorem]{Definition}
\newtheorem{corollary}[theorem]{Corollary}

\theoremstyle{remark}

\newcommand{\Z}{\mathbb{Z}}
\newcommand{\R}{\mathbb{R}}
\newcommand{\T}{\mathbb{T}}
\newcommand{\C}{\mathbb{C}}
\newcommand{\N}{\mathbb{N}}

\numberwithin{equation}{section}

\begin{document}
\title[Dyadic harmonic analysis and Weighted inequalities]{\textbf{Dyadic harmonic analysis and weighted inequalities: the sparse revolution}}

\author[M.C. Pereyra]{Mar\'{i}a Cristina Pereyra}

\address{Mar\'{i}a Cristina Pereyra\\
Department of Mathematics and Statistics\\
 1 University of New Mexico\\ 311 Terrace St. NE, MSC01 1115\\ 
Albuquerque, NM 87131-0001} \email{crisp@math.unm.edu}

\begin{abstract}
 We will introduce the basics of dyadic harmonic analysis and how it can be used to obtain weighted estimates for classical Calder\'on-Zygmund singular integral operators and their commutators.
Harmonic analysts have used dyadic models for many years as a first step towards the understanding of more complex continuous operators. In 2000 Stefanie Petermichl discovered a representation formula for the venerable Hilbert transform as an average (over grids) of dyadic shift operators, allowing her to reduce arguments to finding  estimates for these simpler dyadic models. For the next decade the technique used to get sharp weighted inequalities was the Bellman function method introduced by Nazarov, Treil, and Volberg, paired with sharp extrapolation by Dragi\v{c}evi\'c et al. Other methods where introduced by Hyt\"onen, Lerner, Cruz-Uribe, Martell, P\'erez, Lacey, Reguera, Sawyer, Uriarte-Tuero, involving stopping time and median oscillation arguments, precursors of the very successful domination by positive sparse operators methodology. The culmination of this work was Tuomas Hyt\"onen's 2012 proof of the $A_2$ conjecture based on a representation formula for any Calder\'on-Zygmund operator as an average of appropriate dyadic operators. Since then domination by sparse dyadic operators has taken central stage and has found applications well beyond Hyt\"onen's $A_p$ theorem.  We will survey this remarkable progression and more in these lecture notes.

\end{abstract}


\maketitle

\tableofcontents

\section{Introduction}

These notes are based on lectures delivered by the author on August 7-9,  2017 at the  CIMPA 2017 \emph{Research School -- IX Escuela Santal\'o: Harmonic Analysis, Geometric Measure Theory and Applications}, held in  Buenos Aires, Argentina. The course was titled  "Dyadic Harmonic Analysis and Weighted Inequalities".

The main question of interest  in these notes is to decide for a given operator or class of operators and 
a pair of weights $(u,v)$, if there is  a positive constant, $C_p(u,v,T)$,  such that 
$$
\|Tf\|_{L^p(v)}\leq C_p(u,v,T) \, \|f\|_{L^p(u)} \;\;\mbox{for all functions} \;\; f\in L^p(u).
$$

The main goals in these lectures are two-fold. 
First, given an operator  $T$ (or family of operators), identify and classify pairs of weights $(u,v)$ 
for which the operator(s)  $T$ is(are) bounded  on weighted Lebesgue spaces, more specifically from
$L^p(u)$ to $L^p(v)$ --\emph{qualitative bounds}--.
Second, understand the nature of the constant $C_p(u,v,T)$  --\emph{quantitative bounds}--.

We concentrate on one-weight $L^p$ inequalities for $1<p<\infty$,  that is the case when $u=v=w$, for the prototypical operators, dyadic models, and their commutators, although we will state some of the known two-weight results.
The operators we will focus on are
the Hardy-Littlewood maximal function;
Calder\'on-Zygmund operators  $T$, such as the Hilbert transform $H$; and 
 {their dyadic analogues}, specifically the dyadic maximal function, the martingale transform, the dyadic square function, the Haar shift multipliers, the dyadic paraproducts, and  the {sparse dyadic operators}.

The question now reduces to:
Given weight $w$ and $1< p <\infty$, is there  a constant  $C_p(w,T)>0$ such that for all functions $f\in L^p(w)$
\[
\|Tf\|_{L^p(w)}\leq C_p(w,T) \, \|f\|_{L^p(w)} \;? 
\]

We have known since the 70's that the maximal function is bounded on $L^p(w)$ if an only if the weight $w$ is in the Muckenhoupt $A_p$ class \cite{Mu}, similar result holds for the Hilbert transform \cite{HMW}.  General  Calder\'on-Zygmund operators and dyadic analogues are bounded on $L^p(w)$  \cite{CoFe} when
the weight $w\in A_p$ and the same holds for  their commutators with functions in the space of bounded mean oscillation $({\rm BMO}$)  \cite{Bl, ABKPz}. 
The quantitative versions of these results were obtained several decades later, in 1993  for the maximal function  \cite{Bu}, in 2007 for the Hilbert transform \cite{Pet2}, in 2012  for Calder\'on-Zygmund singular integral operators \cite{Hyt2} and for their commutators \cite{ChPPz}. We will say more about $A_p$ weights and the quantitative versions of these classical results in the following pages. 

We will show or at least describe, for the model operators $T$, the validity of a weighted $L^2$ inequality  that is linear on  $[w]_{A_2}$, the $A_2$ characteristic of the weight, namely
 there is a constant $C>0$ such that for all weights $w\in A_2$ and for all functions $f\in L^2(w)$ 
\[ \|Tf\|_{L^2(w)}\leq C [w]_{A_2} \|f\|_{L^2(w)}.\]
That this holds for all Calder\'on-Zygmund singular integrals operators was the $A_2$ conjecture. We will also describe several approaches for the corresponding quadratic estimate for the commutator $[b,T]=bT-Tb$ where
$b$ is a function in ${\rm BMO}$, namely
\[ \|[b,T]f\|_{L^2(w)}\leq C [w]^2_{A_2}\|b\|_{{\rm BMO}} \|f\|_{L^2(w)}.\]

Dyadic models have been used in harmonic analysis  and  other areas of mathematics for a long time,
Terry Tao has an interesting post in his blog\footnote{https://terrytao.wordpress.com/2007/07/27/dyadic-models/}
regarding the ubiquitous "dyadic model". For a presentation suitable for beginners, see  the lecture notes by the author \cite{P1}, which describe the status quo of dyadic harmonic analysis  and weighted inequalities as of 2000.
This millennium has seen new dyadic techniques evolve, become mainstream, and help settle old problems, these lecture notes try to illustrate some of this progress. In particular averaging and  sparse domination techniques with and by dyadic operators have allowed researchers to transfer results from the dyadic world to the continuous world. No longer the dyadic models are just toy models in harmonic analysis, they can truly inform the continuous models.
Here are some  examples where this dyadic paradigm has been  useful. 

 The {dyadic maximal function} controls the maximal function  (the converse is immediate) by means of the one-third trick.  Estimates for the dyadic maximal function are easier to obtain and  transfer to the maximal function painlessly. 

  The Walsh model is the dyadic counterpart to Fourier analysis. The first real progress towards proving boundedness of the {bilinear Hilbert transform}  \cite{LTh}, result that earned Christoph Thiele and Michael Lacey  the 1996 Salem Prize\footnote{The Salem Prize, founded by the widow of Raphael Salem, is awarded every year to a young mathematician judged to have done outstanding work in Salem's field of interest, primarily the theory of Fourier series. The prize is considered highly prestigious and many Fields Medalists previously received Salem prize (Wikipedia)}, was made by Thiele in his 1995 PhD thesis proving the Walsh model version of such result \cite{Th}.

  {Stefanie Petermichl }  showed in 2000
that one can write the Hilbert transform  as an {"average of dyadic shift operators"} over {random dyadic grids} \cite{Pet1}. 
She achieved this using the well-known symmetry properties that  characterize the Hilbert transform. Namely, the Hilbert transform commutes with translations and dilations, and anticommutes
 with reflections. A linear and bounded operator 
on $L^2(\R )$ with those properties {must be} a constant multiple of the Hilbert transform. 
Similarly, the  Riesz  transforms  \cite{Pet3} can be written as  averages of suitable dyadic operators. Petermichl proved the $A_2$ conjecture for these dyadic operators using  Bellman function techniques \cite{Pet2,Pet3}.
These results added a very precise new dyadic  perspective to   such  classic and well-studied operators  in harmonic analysis and earned Petermichl the 2006 Salem Prize, first time this prize was awarded to a female mathematician.

The Martingale transform was considered the dyadic toy model "par excellence" for Calde-r\'on-Zygmund singular integral operators. For many years  one would test the martingale transform first and, if successful,  then worry about the continuous versions. In 2000, Janine Wittwer proved the $A_2$ conjecture for the martingale transform using Bellman functions \cite{W1}.
 The Beurling transform can be written as an average of martingale transforms in the complex plane, and this allowed Stefanie Petermichl and  Sasha Volberg \cite{PetV} to prove in 2002 linear weighted inequalities on $L^p(w)$ for $p\geq 2$, and as a consequence deduce an important end-point result in the theory of quasiconformal mappings that had been conjectured by  Kari Astala,  Tadeusz Iwaniec, and  Eero Saksman \cite{AIS}.
 
 Surprisingly, all Calder\'on-Zygmund singular integral  operators, can be written as averages of Haar shift dyadic operators of arbitrary complexity and dyadic paraproducts  as proven by Tuomas Hyt\"onen  \cite{Hyt2}. 
In 2008, Oleksandra Beznosova proved the $A_2$ conjecture for the dyadic paraproduct  \cite{Be} and, together with  Hyt\"onen's dyadic representation theorem, this lead to Hyt\"onen's  proof of the full $A_2$ conjecture \cite{Hyt2}.

 Leading towards Hyt\"onen's result there were a number of breakthroughs  that have recently coalesced under the umbrella of {"domination by finitely many sparse  positive dyadic operators"}.
 Andrei Lerner's early results \cite{Le6} played a central role in this development.   It is usually straightforward  to verify that these sparse operators have desired  (quantitative) estimates, it is harder to prove appropriate domination results for each particular operator and function it acts on. This methodology has seen an explosion of applications well-beyond the original $A_2$ conjecture where it originated.
Identifying the sparse collections associated to a given operator and function is the most difficult part of the argument and it  involves using weak-type inequalities, {stopping time techniques}, and {adjacent dyadic grids}. 

We will explore some of these examples in the lecture notes with emphasis on quantitative weighted estimates. We will illustrate in a few case studies different techniques  that have evolved as a result of these investigations such as Bellman functions, quantitative extrapolation and transference theorems, and reduction to studying dyadic operators either by averaging or by sparse domination.

The structure of the lecture notes remains faithful to the lectures delivered by the author in Buenos Aires except for some minor reorganization. Some themes are touched at the beginning, to wet the appetite of the audience, and are expanded on later sections. Most objects are defined as they make their first appearance in the story.  Naturally  more details are provided than in the actual lectures, some details were in the original slides, but had to be skipped or fast forwarded, those topics are included in these lecture notes.
The sections are peppered with historical remarks and references, but inevitably some will be missing or could be inaccurate despite the time and effort spent by  the author on them. Thus, the author apologizes in advance for any inaccuracy  or omission, and gratefully would like to hear about any corrections for future reference.    

In Section~\ref{sec:weighted-norm-inequalities}, we introduce the basic model operators: the Hilbert transform and the maximal function and we discuss their classical $L^p$  and weighted $L^p$ boundedness properties. 
We show that  $A_p$ is a necessary condition for the boundedness of the maximal function on weighted Lebesgue spaces $L^p$.
We describe  why are we interested on weighted estimates, and more recently on quantitative weighted estimates. In particular we describe the linear weighted $L^2$ estimates saga leading towards the resolution of the $A_2$ conjecture and how to derive quantitative weighted $L^p$ estimates using  sharp extrapolation. We finalize the section with a brief summary of the two-weight results known for the Hilbert transform and the maximal function.

In Section~\ref{sec:Dyadic-harmonic-analysis}, we introduce the elements of dyadic harmonic analysis and  the basic dyadic maximal function. More precisely we discuss dyadic grids (regular, random, adjacent) and Haar functions (on the line, on $\R^d$,  on spaces of homogeneous type).  As a first example,  illustrating the power of the dyadic techniques, we present Lerner's proof of Buckley's quantitative $L^p$ estimates for the  maximal function, which reduces, using the one-third trick, to estimates for the dyadic maximal function. We also describe, given dyadic cubes on spaces of homogeneous type, how to construct corresponding Haar bases, and briefly describe the Auscher-Hyt\"onen "wavelets" in this setting.

In Section~\ref{sec:Dyadic-operators}, we discuss the basic dyadic operators:  the martingale transform, the dyadic square function, the Haar shifts multipliers (Petermichl's and those of arbitrary complexity), and the dyadic paraproducts. These are the ingredients needed to state Petermichl's and Hyt\"onen's representation theorems for the Hilbert transform and Calder\'on-Zygmund operators respectively.  For each of these dyadic model operators we describe the known $L^p$  and weighted $L^p$ theory and we state both Petermichl's and Hyt\"onen's representation theorems.

In Section~\ref{sec:A2paraproduct}, we sketch Beznosova's proof of the $A_2$ conjecture for the dyadic paraproduct, this is a Bellman function argument. As a first approach we get a 3/2 estimate, and with a refinement the linear estimate for the dyadic paraproduct is obtained. Along the way we introduce weighted Carleson sequences, a weighted Carleson embedding lemma, some Bellman function lemmas: the Little Lemma and the $\alpha$-Lemma,  and weighted Haar functions needed in the argument, we also sketch the proofs of these auxiliary results.

 In Section~\ref{sec:Commutator},  we discuss weighted inequalities in a case study:  the  commutator of the Hilbert transform $H$ with a function $b$ in ${\rm BMO}$. We summarize chronologically the weighted norm inequalities known for the commutator. We sketch the  dyadic proof of the quantitative weighted $L^2$ estimate for the commutator $[b,H]$  due to Daewon Chung, yielding the optimal quadratic dependence on the $A_2$ characteristic of the weight. We  discuss a very useful transference theorem of Daewon Chung, Carlos P\'erez and the author, and present its proof based on the celebrated Coifman--Rochberg--Weiss argument. The transference theorem allows to deduce quantitative weighted $L^p$ estimates for the commutator of a linear operator  with a ${\rm BMO}$ function, from given weighted $L^p$ estimates for the operator. 
 
In Section~\ref{sec:Sparse}, we introduce the sparse domination by positive dyadic operators paradigm that has 
emerged and proven to be very powerful with applications in many areas not only weighted inequalities.
We discuss a characterization of sparse families of cubes via Carleson families of dyadic cubes due to  Andrei Lerner and Fedja Nazarov. We  present the beautiful proof of the $A_2$ conjecture for sparse operators due to David Cruz-Uribe,  Chema Martell, and Carlos P\'erez. We illustrate with one toy model example, the martingale transform, how to achieve the pointwise domination by sparse operators following an argument by Michael Lacey. Finally we briefly discuss a sparse domination theorem for commutators valid for (rough) Calderón-Zygmund singular integral operators due to Andrei Lerner, Sheldon Ombrosi, and Israel Rivera-Ríos that yields a new quantitative two weight estimates of Bloom type, and recovers all known  weighted results for the commutators.

Finally, in Section~\ref{sec:Recent-progress} we present a summary and briefly discuss some very recent progress.

Throughout the lecture notes a constant $C>0$ might change from line to line. The notation $A:=B$ or $B=:A$ means that $A$ is defined to be $B$. The notation $A\lesssim B$  means that there is a constant $C>0$ such that $A\leq CB$.  The notation $A\sim B$ means that  $A\lesssim B$ and $B\lesssim A$.  The notation $A\lesssim_{r,s} B$  means that the constant $C>0$ in the implied  inequality depends only on the parameters $r,s$.

\vskip .1in
{\bf Acknowledgements:} 
I would like to thank Ursula Molter, Carlos Cabrelli, and all the organizers  of the CIMPA 2017 \emph{Research School -- IX Escuela Santal\'o: Harmonic Analysis, Geometric Measure Theory and Applications}, held in  Buenos Aires, Argentina from July 31 to  August 11, 2017,    for the invitation to give the course on which  these lecture notes are based.  It meant a lot to me to teach in  the "Pabell\'on 1 de la Facultad de Ciencias Exactas", 
 having grown up hearing stories about the mythical Universidad de Buenos Aires (UBA) from  my 
  parents, {Concepci\'on Ballester and Victor Pereyra}, and dear friends\footnote{Dear friends such as Juli\'an and Amanda Araoz, Manolo Bemporad, Mischa and Yanny Cotlar, Rebeca Guber, Mauricio and Gloria Milchberg, Cora Ratto and Manuel Sadosky, Cora Sadosky and Daniel Goldstein, and Cristina Zoltan, some saddly no longer with us.} who, like us,  were welcomed in Venezuela in the late 60's and 70's, and to whom I would like to dedicate these lecture notes. Unfortunately the flow is now being reversed as many Venezuelans of all walks of life are fleeing their country and many, among them mathematicians and scientists, are finding a home in other South American countries, in particular in Argentina.  
  I would also like to thank the enthusiastic students and other attendants, as always, there is no course without an audience, you are always an inspiration for us.  I thank the kind referee,  who made many comments that greatly improved the presentation, and my former PhD student David Weirich, who kindly provided the figures.
Last but not least, I would like to thank my husband, who looked after our boys while I was traveling, and my family in Buenos Aires who lodged and fed me.

\section{Weighted Norm Inequalities}\label{sec:weighted-norm-inequalities}

In this section, we introduce some  basic  notation and the model operators: the Hilbert transform and the maximal function and we discuss their classical $L^p$  and weighted $L^p$ boundedness properties. 
We show that  $A_p$ is a necessary condition for the boundedness of the maximal function on weighted $L^p$.
We describe  why are we interested in weighted estimates, and more recently on quantitative weighted estimates. In particular we describe the linear weighted $L^2$ estimates saga leading towards the resolution of the $A_2$ conjecture and how to derive quantitative weighted $L^p$ estimates using  sharp extrapolation. We finalize the section with a brief summary of the two-weight results known for the Hilbert transform and the maximal function.

\subsection{Some basic notation and prototypical operators}

We  introduce some basic notation used throughout the lecture notes. We remind the reader the basic spaces (weighted $L^p$ and bounded mean oscillation, ${\rm BMO}$), and  the prototypical continuous operators to be studied, namely the maximal function, the Hilbert transform  and  its commutator with functions in ${\rm BMO}$. We briefly recall some of the settings where these operators appear. 

 The \emph{weights} $u$ and $v$ are locally integrable functions  on $\R^d$, namely $u,v \in L^1_{loc}(\R^d)$,  that are almost everywhere  positive functions. 
  
 Given a weight $u$,  a measurable function $f$ is in $L^p(u) $ if and only if  
\[\|f\|_{L^p(u)}:=\left (\int_{\R^d} |f(x)|^p \, u(x)\,dx \right )^{1/p} < \infty.\]
 When $u\equiv 1$ we denote $L^p(\R^d )=L^p(u)$ and  $\|f\|_{L^p}:=\|f\|_{L^p(\R^d )}$.
 
  Given $f,g\in L^1(\R^d )$ their \emph{convolution} is given by
 \begin{equation}\label{convolution}
  f*g(x)=\int_{\R^d} f(x-y)\, g(y)\,dy.
  \end{equation}
 
 A locally integrable function $b$ is in the space of \emph{bounded mean oscillation},  namely $b\in {\rm BMO}$,  if and only if
 \begin{equation}\label{BMO}
  \|b\|_{{\rm BMO}} :=\sup_{Q} \frac{1}{|Q|} \int_Q |b(x) -\langle b\rangle_Q|\,dx  <\infty,  \;\mbox{where} \; \langle b \rangle_Q= \frac{1}{|Q|} \int_Q b(t)\;dt,
  \end{equation}
here  $Q\subset \R^d$ are cubes with sides parallel to the axes, $|Q|$ denotes the volume of the cube $Q$, and more generally, $|E|$ denotes the Lebesgue measure of a measurable set  $E$ in $\R^d$. 
Note that  $L^{\infty}(\R^d)$,  the space of essentially bounded functions on $\R^d$, is a proper subset of ${\rm BMO}$ (e.g. $\log |x|$ is a function in ${\rm BMO}$ but  not in $L^{\infty}(\R )$).
 
We will consider linear or sublinear operators  $T: L^p(u)\to L^p(v)$.  Among the linear operators the Calder\'on-Zygmund singular integral operators  and their dyadic analogues will be most important for us.

  The prototypical \emph{Calder\'on-Zygmund singular integral operator}  is the  \emph{Hilbert transform} on $\R$,  given by convolution with the distributional  \emph{Hilbert kernel}  $k_H(x):={\rm p.v.}\big ({1}/{(\pi x)}\big )$
\begin{equation}\label{def:Hpv} Hf(x):= k_H * f(x)=\mbox{p.v.} \frac{1}{\pi}\int \frac{f(y)}{x-y}\,dy
 :=\lim_{\epsilon\to 0} \frac{1}{\pi} \int_{|x-y|>\epsilon  } \frac{f(y)}{x-y} \,dy.
 \end{equation}
The Hilbert transform and its periodic analogue naturally appear in complex analysis and  in the study of convergence on $L^p$ of partial Fourier sums/integrals. The Hilbert transform  siblings, the Riesz transforms on $\R^d$ and  the Beurling transform on $\C$, are intimately connected to partial differential equations and to quasiconformal theory, respectively.  Its cousin, the Cauchy integral on curves and  higher dimensional analogues, is  connected to rectifiability  and  geometric measure theory.

A prototypical sublinear operator is the \emph{Hardy-Littlewood  maximal function} 
\begin{equation}\label{def:Maximal}
Mf(x):= \sup_{Q: x\in Q} \frac{1}{|Q|}\int_{Q} |f(y)|\,dy,
\end{equation}
here the supremum is taken over all cubes $Q\subset \R^d$ containing $x$ and  with sides parallel to the axes. The maximal function naturally controls many singular integral operators and approximations of the identity, its weak-boundedness properties on $L^1(\R^d)$ imply the Lebesgue differentiation theorem. Another sublinear operator that we will encounter in these lectures is the dyadic square function, see Section~\ref{sec:dyadic-square-function}.
 
Given $T$ a linear or sublinear operator, its  \emph{commutator} with a function $b$ is given by 
$$[b,T](f):= b\,T(f)-T(bf).$$
  The commutators are important in the study of factorization for Hardy spaces and to  characterize the space of bounded mean oscillation (BMO). They also play a central role in the theory of partial differential equations (PDEs).

  We refer the reader to \cite{Gr1,Gr2, St} for encyclopedic presentations of classical harmonic analysis,
  \cite{Duo1} for a more succinct yet deep presentation, and \cite{PW} for an elementary presentation emphasizing the dyadic point of view.

\subsection{Hilbert transform}
We now recall familiar facts about the Hilbert transform, including its $L^p$ and  one-weight (quantitative) $L^p$ boundedness properties.

The Hilbert transform  is defined  by \eqref{def:Hpv} on the underlying space and on frequency space the following representation as a \emph{Fourier multiplier} with \emph{Fourier symbol} $m_H$, holds,
\begin{equation}\label{def:Hfouriermultiplier}
\widehat{Hf}(\xi ) = m_H(\xi)\, \widehat{f}(\xi), 
\;\;\mbox{where}\;\; m_H(\xi) := -i \,\mbox{sgn} (\xi ).
\end{equation}
To connect the two representations for the Hilbert transform, on the underlying space and on the frequency space, remember that multiplication on the Fourier side corresponds to convolution on the underlying space. Therefore, $k_H$, the Hilbert kernel, is  given by the inverse Fourier transform of the Fourier symbol $m_H$, 
$$
{ Hf(x)=  k_H * f(x), \quad\mbox{where} \;\; k_H(x) := (m_H)^{\vee}(x) =  \mbox{p.v.} \frac{1}{\pi x}},
$$
which  is precisely the content of \eqref{def:Hpv}.
Here the {\em Fourier transform}  and \emph{inverse Fourier transform} of  a Schwartz function $f$ on $\R$ are defined by
  $$\widehat{f}(\xi ) :=\int_{\R} f(x)\, e^{-2\pi i \xi x} \,dx, \quad\quad  (f)^{\vee}(x):= \int_{\R} f(\xi) \, e^{2\pi i \xi x} \, d\xi. $$
  The  Fourier transform is a bijection and an $L^2$ isometry on the Schwartz class that can be extended to be an  isometry on $L^2(\R )$, that is $\|\widehat{f}\|_{L^2(\R )} = \|f\|_{L^2(\R )}$ (\emph{Plancherel's identity}), and it can also be extended to be a bijection on the space of tempered distributions. The convolution $f*g$ is a well-defined function on $L^r(\R )$ when $f\in L^p(\R) $ and $g\in L^p(\R )$, provided $\frac1p + \frac1q=\frac1r +1$ and $p,q,r\in [1,\infty]$. Moreover, on the same range,  \emph{Young's inequality} holds,
 \begin{equation}\label{Youngs}
 \|g*f\|_{L^r} \leq \|g\|_{L^q}\|f\|_{L^p}. 
 \end{equation}
 In these lecture notes we will explore, in Section~\ref{sec:Petermichl'sSha}, a third representation for the Hilbert transform in terms of dyadic shift operators  discovered by Stefanie Petermichl \cite{Pet1} in 2000.

\subsubsection{$L^p$ boundedness properties of $H$}

 Fourier theory ensures boundedness on $L^2(\R)$  for the Hilbert transform $H$. In fact, applying Plancherel's identity twice and using the fact that $|m_H(\xi)|=1$ a.e., one immediately verifies that  $H$ is an isometry on $L^2(\R )$, namely
$$\|Hf\|_{L^2}=\|\widehat{Hf}\|_{L^2}=\|\widehat{f}\|_{L^2}=\|f\|_{L^2}.$$
Young's inequality \eqref{Youngs} for ${p\geq 1}$, $q=1$ (hence $r=p$), imply that if
$g\in L^1(\R)$ and  $f\in L^p(\R)$ then $g*f\in L^p(\R)$, moreover
\[
\|g*f\|_{L^p}\leq \|g\|_{L^1}\|f\|_{L^p}.
\]
This would imply boundedness on $L^p(\R )$ for the Hilbert transform {\sc if} 
the Hilbert kernel, {$k_H$,  were integrable, but is not. Despite this fact,
the following boundedness properties for the Hilbert transform hold (shared by all Calder\'on-Zygmund  singular integral operators).

The Hilbert transform is not bounded on $L^1(\R )$, it is of \emph{weak-type} (1,1) ({Kolmogorov 
$1927$}), that is there is a constant $C>0$ such that for all $\lambda >0$ and for all $f\in L^1(\R )$
\[ 
|\{x\in\R: |Hf(x)| > \lambda\}| \leq \frac{C}{\lambda} \|f\|_{L^1}.
\]

The Hilbert transform is bounded on 
${L^p(\R )}$ for all ${1<p<\infty}$  ({M. Riesz $1927$}), namely there is a constant $C_p>0$ such that for all $f\in L^p(\R)$
\[
 \|Hf\|_{L^p}\leq C_p \|f\|_{L^p} \quad\mbox{(best constant was found by Pichorides in $1972$)}.
 \]
 Note that for $1<p<2$ the $L^p$ boundedness can be obtained by Marcinkiewicz interpolation theorem, from the  weak-type (1,1) and the $L^2$ boundedness.  Then, for $2<p<\infty$, the boundedness on $L^p(\R )$ can be obtained by a duality argument,  suffices to observe that  the adjoint of $H$  is $-H$, that is the Hilbert transform is almost self-adjoint. However the Marcinkiewicz interpolation did not exist in 1927,  Riesz proved instead that boundedness on $L^p(\R )$ implied boundedness on $L^{2p}(\R )$, hence boundedness on $L^2(\R )$ implied boundedness on $L^4(\R )$, then on $L^8(\R )$ and by induction on $L^{2^n}(\R )$. Strong interpolation, which already existed, then gave boundedness on $L^p(\R)$ for  $2^n\leq p \leq 2^{n+1}$ and for all $n\geq 1$, that is for all $2\leq p<\infty$. Finally a duality argument took care of $1<p<2$. In Section~\ref{sec:Petermichl'sSha} we will  deduce the $L^p$ boundedness of the Hilbert transform from the $L^p$ boundedness of dyadic shift operators, see Section~\ref{sec:extrapolation}
 
  Interpolation is an extremely powerful tool in analysis that allows  to deduce intermediate  norm inequalities given two end-point  (weak)norm inequalities. We will not discuss interpolation further in these notes, instead we will focus on extrapolation, that allows us to deduce weighted $L^p$ norm inequalities  for all $1<p<\infty$ given weighted $L^r$ norm inequalities for one index $r>1$. 
 
Finally it is important to note that the  Hilbert transform   is not bounded on $L^{\infty}(\R )$, however  it is bounded on the larger space ${\rm BMO}$ of functions of bounded mean oscillation ({C. Fefferman $1971$}).

To illustrate the lack of boundedness on $L^{\infty}(\R )$ and on $L^1(\R )$  it is helpful to calculate  the Hilbert transform for some simple functions, showing in fact that the Hilbert transform does not map neither $L^1(\R )$ nor $L^{\infty}(\R )$ into themselves.  This immediately eliminates the possibility for the Hilbert transform  being bounded on either space.

\begin{example}[Hilbert transform of  an indicator function]
\[H\mathbbm{1}_{[a,b]}(x)= (1/\pi) \,\log \big (|x-a|/|x-b|\big ),\]
 where the indicator $\mathbbm{1}_{[a,b]}(x):=1$ when $x\in [a,b]$ and zero otherwise, a bounded and integrable function, that is $\mathbbm{1}_{[a,b]}\in L^1(\R )\cap L^{\infty}(\R)$. 
However  $\log |x|$  is neither in $L^{\infty}(\R )$ nor in $L^1(\R)$, but it is a function of bounded mean
 oscillation. The functions $f$ in $L^1(\R)$ whose Hilbert transforms  $Hf$ are also in $L^1(\R)$ constitute the Hardy space $H^1(\R )$, such functions need to have some cancellation $(\int_{\R} f(x)\,dx =0)$, clearly not shared by the indicator function $\mathbbm{1}_{[a,b]}$.
\end{example}

\subsubsection{One-weight inequalities for $H$}\label{sec:one-weight-inequalities-for-H}

The one-weight theory \`a la Muckenhoupt for the Hilbert transform is well understood, the qualitative theory has been known since 1973 \cite{HMW}, the quantitive estimates were settled by Stefanie Petermichl in  2007 \cite{Pet2}. The two-weight problem on the other hand, was studied for a long time but the necessary and sufficient conditions  \`a la Muckenhoupt for pairs of weights  $(u,v)$ that ensure boundedness of the Hilbert transform from $L^p(u)$ into $L^p(v)$ were only settled in 2014 by Michael Lacey, Chun-Yen Shen, Eric Sawyer, and Ignacio Uriarte-Tuero \cite{L2, LSSU}.
\begin{theorem}[Hunt, Muckenhoupt, Wheeden 1973]\label{thm:HMW}
The Hilbert transform is bounded on $L^p(w)$ for $1<p<\infty$ if and only if the weight $w\in A_p$. In either case there is a constant $C_p(w)>0$ depending on $p$ and  on the weight $w$ such that
\[
 \|Hf\|_{L^p(w)}\leq C_p(w) \|f\|_{L^p(w)}\quad \mbox{ for all $f\in L^p(w)$}.
\]
\end{theorem}
At this point we remind the reader that a weight $w$ is in the \emph{Muckenhoupt $A_p$ class} if and only if  $[w]_{A_p}<\infty$, where the \emph{$A_p$ characteristic} of the weight $w$ is defined to be
$$[w]_{A_p}:=\sup_Q\bigg(\frac{1}{|Q|}\int_Q w(x)\,dx\bigg)
\bigg(\frac{1}{|Q|}\int_Qw^{\frac{-1}{p-1}}(x)\, dx\bigg)^{p-1}\quad \mbox{for}\;\;1<p<\infty\,,
$$
 the supremum is taken over all cubes $Q$  in $\R^d$ with sides parallel to the axes. We will denote integral averages with respect to Lebesgue measure on cubes or on measurable sets $E$ by $\langle f\rangle_E := \frac{1}{|E|}\int_E f(x)\, dx$. Also given $w$,  a weight, $w(E)$ will denote the $w$-mass of the measurable set $E$, that is, $w(E)=\int_E w(x)\, dx$. With this notation 
 $$[w]_{A_2}:=\sup_Q \langle w\rangle_Q \langle w^{-1} \rangle_Q.$$ 
 Note that $w\in A_2$ if and only if $w^{-1}\in A_2$.
 \begin{example} Power weights offer examples of $A_p$ weights on $\R^d$, $w(x)=|x|^{\alpha}$ is in $A_p$
if and only if $-d \leq \alpha\leq d(p-1)$ for $1<p<\infty$. 
\end{example}

In Theorem~\ref{thm:HMW}, the optimal dependence of the constant  $C_p(w)$  on  the $A_p$ characteristic $[w]_{A_p}$  of the weight $w$,    was  
found more than 30 years later.
\begin{theorem}[{Petermichl 2007}] Given $1<p<\infty$,
for all  $w\in A_p$ and for all $f\in L^p(w)$ we have that
$$
 \|Hf\|_{L^p(w)}\lesssim_p {[w]_{A_p}^{\max{\{1,\frac{1}{p-1}\}}}} \|f\|_{L^p(w)}.
$$
\end{theorem}

Note that the estimate is \emph{linear}  on $[w]_{A_p}$ for $p\geq 2$, and of power ${\frac{1}{p-1}}$ for $1<p <2$.

\begin{proof}[Cartoon of the proof] The following is a very brief sketch of Petermichl's argument.
 First, write $H$ as an {average over dyadic grids} of {dyadic shift  operators} \cite{Pet1}.
Second, find  {linear estimates}, uniform  (on the dyadic grids), for  the  dyadic shift operators on $L^2(w)$ \cite{Pet2}. Deduce from the first two steps linear estimates on $L^2(w)$ for the Hilbert transform, 
namely estimates  valid for all $w\in A_2$ and for all $f\in L^2(w)$ of the form
$$ \|Hf\|_{L^2(w)}\lesssim[w]_{A_2}\|f\|_{L^2(w)}.$$
 Third, use a {sharp  extrapolation theorem} \cite{DGPPet} to get estimates for $p\neq 2$ from  the linear $L^2(w)$ estimate.
\end{proof}
Same estimates hold for {\sc all} Calder\'on-Zygmund singular integral operators, solving the famous $A_2$ conjecture,  which was proven by Tuomas Hyt\"onen in 2012, see \cite{Hyt2}. We will say more about Petermichl's and Hyt\"onen's  landmark results  as well as about  sharp extrapolation later in Section~\ref{sec:extrapolation} and in Section~\ref{sec:Dyadic-operators}.

\subsection{Maximal function}
We summarize  the $L^p$  and  one-weight (quantitative)  $L^p$ boundedness properties for the maximal function. We also show that the $A_p$ condition on the weight $w$  is a necessary condition for boundedness of the maximal function on $L^p(w)$.

\subsubsection{$L^p$ boundedness  properties of $M$}

From its definition \eqref{def:Maximal}, it is clear that the  maximal function is bounded on $L^{\infty}(\R^d )$ with norm one.  The maximal function is not bounded on $L^1(\R^d)$, however it  is  of weak-type (1,1)   ({Hardy, Littlewood  1930}). The next example shows that the maximal function does not map $L^1(\R )$ onto itself.

\begin{example}  The characteristic function $\mathbbm{1}_{[0,1]}$ is integrable however its image, under the maximal function,  $M\mathbbm{1}_{[0,1]}$, is not.      The diligent reader can verify that  $M\mathbbm{1}_{[0,1]} (x) = {1}/(1-x)$ if $x<0$, $M\mathbbm{1}_{[0,1]} (x) = 1$ if $ 0\leq x\leq 1$, and $M\mathbbm{1}_{[0,1]} (x) = 1/x$ if $x>1$.
\end{example}

 Marcinkiewicz interpolation  gives boundedness of the maximal function on $L^p(\R^d )$ for $1<p<\infty$ from the strong $L^{\infty}$ and the weak-type $(1,1)$ boundedness results. We will present an alternate argument in Section~\ref{sec:Lerners-proof} that will cover the weighted $L^p$ estimates as well without reference to neither interpolation nor extrapolation.

\subsubsection{One-weight  $L^p$ inequalities for $M$}
 The maximal function is of weak $L^p(w)$ type if and only if  $w\in A_p$, moreover the following quantitative result was proven in 1972 by Benjamin Muckenhoupt  \cite{Mu}, for $p\geq 1$ and for all $w\in A_p$, 
\begin{equation}\label{Muckenhoupt-weak-Lp-bound-for-M}
\|M\|_{L^p(w)\to L^{p,\infty}(w)}  \lesssim_p { [w]_{A_p}^{1/p}}, 
 \end{equation}
where the quantity  on the left-hand-side, $\|M\|_{L^p(w)\to L^{p,\infty}(w)} $,  denotes the smallest constant $C>0$ such that for all $\lambda>0$ and for all $f\in L^p(w)$
$$ w \big (  \{x\in\R^d: Mf(x)>\lambda\}  \big) \leq  \left (\frac{C}{\lambda} \|f\|_{L^p(w)}\right )^p.$$
We say a weight $w$ is in the \emph{Muckenhoupt $A_1$ class} if and only if there is a constant $C>0$ such that 
\[ Mw(x) \leq C w(x) \quad \mbox{for a.e. $x\in \R^d$}.\]
The infimum over all possible such constants $C$ is denoted $[w]_{A_1}$. The $A_1$ class of weights  is contained in all $A_p$ classes  of weights for $p>1$.

 The maximal function is bounded on $L^p(w)$, moreover the following quantitative result was proven in 1993 by Stephen Buckley \cite{Bu} valid for $p>1$ and for all $ w\in A_p$ and   $f\in L^p(w)$,
 \begin{equation}\label{Buckleys}
  \|Mf\|_{L^p(w)} \lesssim_p [w]_{A_p}^{1/(p-1)}\|f\|_{L^p(w)}.  
  \end{equation} 
 Buckley  deduced these estimates from quantitative self-improvement integrability results known for $A_p$ weights, the weak $L^{p\pm\epsilon}(w)$  boundedness of the maximal function, and Marcinkie-wicz interpolation. More precisely,  $w\in A_p$ implies  $w\in A_{p-\epsilon}$ with $\epsilon\sim {[w]^{1-p'}_{A_p}}$ and $[w]_{A_{p-\epsilon}}\leq 2[w]_{A_p}$, on the other hand H\"older's inequality  implies $A_p\subset A_{p+\epsilon}$ and $[w]_{A_{p+\epsilon}}\leq [w]_{A_p}$. Interpolating between weak $L^{p-\epsilon}(w)$   and weak $L^{p+\epsilon}(w)$ estimates and keeping track of the constants one gets  Buckley's quantitative estimate  \eqref{Buckleys}.
 
In particular when $p=2$ the maximal function obeys a linear estimate on $L^2(w)$ with respect to the $A_2$ characteristic of the weight, namely for all $w\in A_2$ and $f\in L^2(w)$
 $$\|Mf\|_{L^2(w)} \lesssim {[w]_{A_2}}\|f\|_{L^2(w)}.$$

A beautiful proof  of Buckley's quantitative estimate for the maximal function was presented in 2008 by Andrei Lerner \cite{Le1}, mixed $A_p$-$A_{\infty}$ estimates in 2011 by Tuomas Hyt\"onen and Carlos P\'erez \cite{HPz}, and  extensions to spaces of homogeneous type in 2012 by Tuomas Hyt\"onen and  Anna Kairema  \cite{HK}. We will present Lerner's proof of Buckley's inequality~\eqref{Buckleys}
in Section~\ref{sec:Lerners-proof}.

\subsubsection{$A_p$ is a necessary condition for $L^p(w)$ boundedness of $M$} 

We would like to demystify the appearance of the $A_p$ weights in the theory by showing that $w\in A_p$ is a necessary condition for the maximal function to be bounded on $L^p(w)$ when $p>1$.

We will show that {\sc If} the maximal function is bounded on $L^p(w)$
  {\sc then}  the weight $w$ must be in the Muckenhoupt  $A_p$ class.

\begin{proof}
By hypothesis,  there is a constant $C>0$ such that for all $f\in L^p(w)$,
\[
\|Mf\|_{L^p(w)}\leq C \|f\|_{L^p(w)}.
\]
For all $\lambda>0$, let $E^{Mf}_{\lambda}$ be the $\lambda$-level set for the maximal function $Mf$, that is
\[E_{\lambda}^{Mf}:= \{ x\in \R^d: Mf(x)\geq \lambda \},\]
 then, by Chebychev's inequality\footnote{Namely, for $g\in L^1(\mu )$ it holds that $\mu \{x\in\R: |g(x)|>\lambda\}|\leq \frac{1}{\lambda}\|g\|_{L^1(\mu )}$ for all $\lambda>0$, in other words if $g\in L^1(\mu )$ then
 $g\in L^{1,\infty}(\mu )$, where $g\in L^{p,\infty}(\mu )$  means $\|g\|_{L^{p,\infty}(\mu )} := \sup_{\lambda>0} \lambda \mu^{1/p}\{x\in\R^d: |g(x)| >\lambda\}<\infty$.}, and using the hypothesis we conclude that
$$ w (E_{\lambda}^{Mf})= \int_{E_{\lambda}^{Mf}} w(x)\,dx \leq \frac{1}{\lambda^p}\int_{\R^d} |Mf(x)|^p w(x)\,dx
\leq  \frac{C^p}{\lambda^p} \|f\|_{L^p(w)}^p.$$
Fix a cube $Q\subset \R^d$, for any integrable function $f\geq 0$, supported on the cube $Q$, let $\lambda := \frac{1}{|Q|}\int_Q f(y)\,dy$. Then $Mf(x) \geq \lambda$ for all $x\in Q$ hence $Q\subset E_{\lambda}^{Mf}$, moreover
 \begin{equation}\label{eqn:Ap-necessary-M}
\Big (\frac{1}{|Q|}\int_Q f(x)\,dx \Big )^p w(Q) \leq \lambda^p \, w(E_{\lambda}^{Mf})
 \leq C^p \int_Q f^p(x)\,w(x)\,dx.
 \end{equation}
Consider the specific function $f=w^{\frac{-1}{p-1}}\, \mathbbm{1}_Q$ supported on $Q$ and chosen  so that   both integrands coincide, namely $f=f^pw$. Substitute this specific function $f$ into \eqref{eqn:Ap-necessary-M} to obtain the following inequality only pertaining the weight $w$ and the cube $Q$,
\[
 \frac{1}{|Q|^p}\Big ( \int_Q w^{\frac{-1}{p-1}}(x)\,dx\Big )^{p-1}\, w(Q) \leq C^p.
 \]
 Distribute $|Q|$ and take the supremum over all cubes $Q$ to conclude that  $[w]_{A_p}\leq C^p$, and hence $w\in A_p$.  There is one technicality, the chosen function may not be integrable, choose instead $f_{\epsilon}=\mathbbm{1}_Q(w+\epsilon)^{\frac{-1}{p-1}}$, run the argument for each $\epsilon>0$ then let $\epsilon$ go to zero.
\end{proof}

We  just showed that { if} the maximal function $M$ is bounded on $L^p(w)$ then it is of  weak~$L^p(w)$ type. Moreover
$ [w]_{A_p}^{1/p}\leq \|M\|_{L^p(w)\to L^{p,\infty}(w)}$,  therefore Muckenhoupt's  weak  $L^p(w)$ bound \eqref{Muckenhoupt-weak-Lp-bound-for-M} is optimal.

\subsection{Why are we interested in these estimates?}

We record a few instances where $L^p$~and weighted $L^p$ estimates are of importance in analysis.
\begin{itemize}
\item[-] {\sc Fourier Analysis}: Boundedness of the periodic Hilbert transform on $L^p(\T)$ implies convergence on $L^p(\T )$
of the partial Fourier sums. 
\item[-] {\sc Complex Analysis}: $Hf$ is the boundary value of the harmonic conjugate of the Poisson extension
to the upper-half-plane of a function $f\in L^p(\R )$.  
\item[-] {\sc Factorization}: Theory of (holomorphic) Hardy spaces $H^p$. Elements of $H^p$ can be defined as those distributions whose image under properly defined maximal functions (or other suitable singular operators or square functions) are in $L^p$.
\item[-] {\sc Approximation Theory}: Boundedness properties of the martingale transform (a dyadic analogue of the Hilbert transform) show that Haar functions and other wavelet families   are unconditional bases of several functional spaces. 
\item[-] {\sc PDEs}: Boundedness of the  \emph{Riesz transforms} (analogues of the Hilbert transform on $\R^d$) and their commutators 
have deep connections to partial differential equations. 
\item[-] {\sc Quasiconformal Theory}: Boundedness of the
 \emph{Beurling transform}  (singular integral operator on $\C$) on $L^p(w)$
for $p>2$ and with linear estimates on $[w]_{A_p}$ implies borderline regularity result. \item[-] {\sc Operator Theory:} Weighted inequalities appear naturally in the theory of \emph{Hankel and Toeplitz operators}, perturbation theory, etc.
\end{itemize}

We expand on the  weighted estimate needed in quasiconformal theory which propelled the interest in quantitative weighted estimates. This was work by Kari Astala, Tadeusz Iwaniec, and Eero Saksman in 2001, we  refer to their paper \cite{AIS} for appropriate definitions. They showed that for $1<K<\infty$
every weakly $K$-quasi-regular mapping, contained in a
Sobolev space $W^{1,q}_{{\rm loc}}(\Omega)$ with {$2K/(K+1)< q\leq 2$}, is quasi-regular on $\Omega$,
that is to say,  it belongs to  $W^{1,2}_{{\rm loc}}(\Omega)$. 
 For each {$q< 2K/(K + 1)$} there are weakly $K$-quasi-regular mappings $f\in W^{1,q}_{{\rm loc}}(\C)$
which are not quasi-regular. The only value of $q$ that remained unresolved was the endpoint, 
they conjectured that all weakly $K$-quasi-regular mappings $f\in W^{1,q}_{{\rm loc}}$ with {$q =
2K/(K + 1)$} are in fact quasi-regular. 
 They reduced the conjecture to showing {\cite[Proposition 22]{AIS}}  that the Beurling transform  $T$ satisfies {linear bounds} in $L^p(w)$ for "$p>1$", namely
\[ \|Tg\|_{L^p(w)}\lesssim_p {[w]_{A_p}}\|g\|_{L^p(w)}, \quad \mbox{for all} \, w\in A_p \; \mbox{and}\; g\in L^p(w). \] 
Fortunately the values of interest for $q$ are $1<q<2$ and $p=q' >2$. 
Linear bounds for the Beurling transform and {$p\geq 2$} were proven in 2002 by  Stefanie Petermichl and Sasha Volberg \cite{PetV}. As a consequence the regularity at the borderline case $q=2K/(K+1)$ was established.
For $1<p<2$ the correct estimate for the Beurling transform is of the form
\[  \|Tg\|_{L^p(w)}\lesssim_p[w]_{A_p}^{1/(p-1)} \|g\|_{L^p(w)}, \quad \mbox{for all } \,w\in A_p \; \mbox{and}\;  g\in L^p(w),\]
as shown in \cite{DGPPet}.

\subsection{First Linear Estimates}

Interest in quantitative weighted estimates exploded in this millennium.  A chronology of the early linear estimates on $L^2(w)$ for the weight $w$ in the Muckenhoupt  $A_2$ class, namely 
$ \|Tf\|_{L^2(w)}\leq C[w]_{A_2}\|f\|_{L^2(w)}$, is as follows.

\begin{itemize}
\item[-] \emph{Maximal function} ({Buckley `93} \cite{Bu}). 
\item[-]  \emph{Martingale transform} ({Wittwer `00} \cite{W1}). 
\item[-] \emph{(Dyadic) square function} ({Hukovic, Treil, Volberg `00} \cite{HTV} ; {Wittwer `02} \cite{W2}). 
\item[-] \emph{Beurling transform} ({Petermichl, Volberg  `02} \cite{PetV}).
\item[-] \emph{Hilbert transform}  ({Petermichl `07} \cite{Pet2}).
\item[-] \emph{Riesz transforms} ({Petermichl  `08} \cite{Pet3}). 
\item[-] \emph{Dyadic paraproduct} ({Beznosova  `08} \cite{Be}).  
\end{itemize} 

Except for the maximal function, all these  linear estimates were obtained using {Bellman functions} and {(bilinear) Carleson estimates} for certain dyadic operators (Petermichl dyadic shift operators, martingale transform, dyadic paraproducts, dyadic square function), and then either the operator under study was one of them or had enough symmetries that it could be represented as a suitable average of dyadic operators (Beurling, Hilbert, and Riesz transforms). 
The Bellman function method  was  introduced in the 90's to harmonic analysis by  Fedja Nazarov, Sergei Treil, and Sasha Volberg \cite{NT, NTV2}, although they credit Donald Burkholder in his celebrated work  finding the exact $L^p$ norm for the martingale transform \cite{Bur}. With  their students and collaborators they  have been able to use the Bellman function method 
to obtain a number of astonishing results not only in this area,    see {Volberg}'s INRIA lecture notes \cite{V} and references. In Volberg's own words\footnote{http://www-sop.inria.fr/apics/ahpi/summerschool11/bellman$_-$lectures$_-$volberg-1.pdf} "the Bellman function method makes apparent the hidden multiscale properties of Harmonic Analysis problems".

A flurry of work ensued and other techniques were brought into play including stopping time techniques (corona decompositions) and median oscillation techniques. These techniques  became the precursors of what is now known as the method of domination by dyadic sparse operators, with important contributions from David Cruz-Uribe, Chema Martell, Carlos P\'erez, Andrei Lerner, Tuomas  Hyt\"onen, Michael Lacey, Mari Carmen Reguera, Stefanie Petermichl, Fedja Nazarov, Sergei Treil,  Sasha Volberg and others.  We will say more about sparse domination in Section~\ref{sec:Sparse}.

The culmination of this work was the celebrated resolution of the $A_2$ conjecture by Tuomas Hyt\"onen \cite{Hyt2} in 2012 where he showed  first  every Calder\'on-Zygmund operator  could be written as an average of dyadic shift operators of arbitrary complexity, dyadic paraproducts and their adjoints,  second  the weighted $L^2$ norm of the dyadic shifts depended linearly on the $A_2$ characteristic of the weight and polynomially on the complexity,  and third these ingredients implied that the Calder\'on-Zygmund operator obeyed  linear bounds on $L^2(w)$. How about weighted $L^p$ estimates for $1<p<\infty$?

\subsection{Extrapolation and Hyt\"onen's $A_p$ theorem}\label{sec:extrapolation}
There is a, by now, classical technique to obtain weighted $L^p$ estimates from weighted $L^2$ estimates 
or more generally from weighted $L^r$ estimates, called extrapolation. 
In this section, we recall the classical Rubio de Francia extrapolation theorem, a quantitative version, due to Oliver Dragi\v{c}evi\'c et al, called "sharp extrapolation",  and deduce from the later Hyt\"onen's  $A_p$ theorem.

\subsubsection{Rubio de Francia Extrapolation Theorem }
Jos\'e Luis Rubio de Francia introduced in the 80's  his celebrated extrapolation result, a theorem that allowed to transfer estimates from  weighted $L^r$ (provided it held for all $A_r$ weights) to weighted $L^p$ for all $1<p<\infty$ and all $A_p$ weights.

\begin{theorem}[{Rubio de Francia 1981}]   
Given $T$ a  sublinear operator and  $r\in \R$ with $1<r<\infty$. If {for all ${ w\in A_r}$} there is a constant  $C_{T,r,d,w}>0$ such that
 $$
 \| Tf\|_{L^r(w)} \leq  C_{T, r,d, w}\|f\|_{L^r(w)} \;\;\mbox{for all} \;\;f\in L^r(w).
$$
Then {for each $1<p<\infty$} and {for all ${ w\in A_p}$}, there is a constant $C_{T,p,r,d,w}>0$ such that
$$
 \| Tf\|_{L^p(w)} \leq  
C_{T,p,r,d,w}\|f\|_{L^p(w)} \;\;\mbox{for all} \;\; f\in L^p(w).
$$
\end{theorem}

If we choose $r=2$,   paraphrasing {Antonio C\'ordoba}\footnote{See page 8 in Jos\'e Garc\'ia-Cuerva's eulogy for  {Jos\'e Luis Rubio de Francia} (1949-1988) \cite{Ga}.} we will conclude that 

\centerline{\emph{There is no $L^p$ just weighted $L^2$.}}

\noindent (Since $w\equiv 1\in A_p$ for all $p$.)

There are books dedicated to the subject that cover this and many useful variants of this theorem,  the classical reference  is  the out-of-print  1985 book by {Garc\'ia-Cuerva and  Rubio de Francia} \cite{GR}. A  modern presentation,  including quantitative versions of this theorem,  is the 2011 book by David Cruz-Uribe, Chema Martell, and Carlos P\'erez \cite{CrMPz1}.

\subsubsection{Sharp extrapolation}\label{sec:sharp-extrapolation}
In the 80's and 90's the interest was on qualitative weighted estimates. Once the interest on quantitative weighted estimates was sparked it was natural to consider quantitative extrapolation theorems, what we call "sharp extrapolation theorems".  This is precisely what Stefanie Petermichl and Sasha Volberg did \cite{PetV}  to obtain linear estimates for the Beurling transform and $p\geq 2$, they missed the range $1<p<2$ because it was of no interest, and their calculation was very specific to the martingale transforms that properly averaged yielded the Beurling transform. It was soon realized that a general principle was at work  \cite{DGPPet}.  We state a simplified version  of what a quantitative extrapolation theorem says, useful for the purposes of this survey.
\begin{theorem}[Dragi\v{c}evi\'{c} et al 2005]   
Let $T$ be a  sublinear operator, $r\in\R$ with  $1<r<\infty$. If {for all ${ w\in A_r}$}  there are constants $\alpha, C_{T,r,d}>0$ such that
 $$
 \| Tf\|_{L^r(w)} \leq  C_{T, r,d}{[w]^{\alpha}_{A_r} }\|f\|_{L^r(w)} \;\;\mbox{for all} \;\;f\in L^r(w).
$$
Then for each $1<p<\infty$ and  {for all ${ w\in A_p}$}, there is a constant  $C_{T,p,r,d}>0$ such that
$$
 \| Tf\|_{L^p(w)} \leq  
C_{T,p,r,d}{ [w]^{\alpha\max{\{1,\frac{r-1}{p-1}\}}}_{A_p} }\|f\|_{L^p(w)} \;\;\mbox{for all} \;\; f\in L^p(w).
$$
\end{theorem}

The proof follows  by now standard arguments involving the celebrated Rubio de Francia algorithm, and  inserting  whenever possible  Buckley's quantitative bounds  \eqref{Buckleys} for the  maximal function  \cite{Bu}.

An alternative, streamlined proof of the sharp extrapolation theorem, was presented by  Javier Duoandikoetxea  in \cite{Duo2}, extending the result to more general settings including off-diagonal and partial range extrapolation.   It was observed \cite{CrMPz1} that one can replace the pair {$(Tf, f)$} by a pair of functions {$(g,f)$} in the extrapolation theorem, in particular one could consider the pair $(f,Tf)$ instead, as long as one has the corresponding initial weighted inequalities required to jump-start the theorem.

Sharp extrapolation is sharp in the sense that no better power for $[w]_{A_p}$ can appear in the conclusion that will work for all operators. For some operators it is known that the extrapolated $L^p(w)$ bounds from the known optimal $L^r(w)$ estimates are themselves  optimal for all $1<p<\infty$. However it is not necessarily optimal for a particular given operator. Here are some examples illustrating this phenomenon.
\begin{example}
Start with Buckley's sharp estimate on $L^r(w)$, $\alpha = \frac{1}{r-1}$, for the maximal function, extrapolation will give sharp bounds
{only for $1<p\leq r$}.
\end{example}

\begin{example}
Sharp extrapolation from $r=2$, $\alpha =1$, is sharp for the  Hilbert, 
Beurling, Riesz 
transforms for all $1<p<\infty$ {\rm (}for $p>2$ \cite{PetV}, \cite{Pet2}, \cite{Pet3}; $1< p <2$ \cite{DGPPet}{\rm )}. 
\end{example}

\begin{example}
Extrapolation from linear bound on $L^2(w)$  is sharp for the {dyadic square function} only when ${1<p\leq 2}$ 
$($"sharp" \cite{DGPPet}, "only" \cite{Le2}$)$. 
However, extrapolation from square root bound on {$L^3(w)$} is sharp for all $p>1$ \cite{CrMPz2}.
\end{example}

\subsubsection{Hyt\"onen's  $A_p$ Theorem}\label{sec:Hytonen-Ap-theorem}

Sharp extrapolation was used by Tuomas Hyt\"onen to prove the celebrated $A_p$ theorem, the quantitative weighted $L^p$ estimates for Calder\'on-Zygmund operators \cite{Hyt2}.
\begin{theorem}[{Hyt\"onen 2012}]
Let \, $1<p<\infty$\,
and let $T$ be any Calder\'on-Zygmund singular 
integral operator on $\R^d$, then  for all $w\in A_p$ and $f\in L^p(w)$
$$
\|Tf\|_{L^{p}(w)} \lesssim_{T,d,p} \,
{[w]_{A_p}^{\max\{1,\frac{1}{p-1} \} }}\|f\|_{L^p(w)}.
$$
\end{theorem}

\begin{proof}[Cartoon of the proof] 
Enough to prove the  $p=2$ case thanks to sharp extrapolation.  To prove the linear weighted  $L^2$ estimate
two important steps were required.

First, prove a representation theorem in terms
of  Haar shift operators
of arbitrary complexity, dyadic paraproducts, and their adjoints on random dyadic grids introduced in \cite{NTV3}. This representation hinges on certain reductions obtained in \cite{PzTV}.

Second,  prove linear estimates on $L^2(w)$ with respect to the $A_2$ characteristic 
for paraproducts \cite{Be} and Haar shift operators  \cite{LPetR}  but with {polynomial dependence on the complexity} (independent of the
dyadic grid) \cite{Hyt2}.
\end{proof}
 We will say more about random dyadic grids, Haar shift operators, and paraproducts, the ingredients in Hyt\"onen's theorem, in Sections~\ref{sec:Dyadic-harmonic-analysis} and~\ref{sec:Dyadic-operators}.
 It is now well understood that the  $L^2(w)$ bounds for the Haar shift operators not only depend linearly on the $A_2$ characteristic of $w$ but also depend linearly on the complexity
 \cite{T}.

Sharp extrapolation has also  been  used to obtain quantitative estimates in other settings. For example, Sandra Pott and Mari Carmen Reguera used sharp extrapolation when studying  the Bergman projection on weighted Bergman spaces in terms of the B\'ekoll\'e constant \cite{PoR}.  They proved the  base estimate on $L^2(w)$  for certain sparse dyadic operators and then showed the Bergman projection could be dominated with  these sparse  dyadic operators.

\subsection{Two-weight problem for the Hilbert transform and the maximal function}
We briefly state a necessarily incomplete chronological list of two-weight results for the Hilbert transform, the maximal function, and allied  dyadic operators.

\subsubsection{Two-weight problem for $H$ and its dyadic model the martingale transform}
 In the `80s, Mischa Cotlar, and Cora Sadosky found necessary and sufficient conditions \`a la Helson-Szeg\"o solving the two-weight problem for the Hilbert transform. The methods used involved complex analysis and had applications to operator theory \cite{CS1,CS2}.
 Afterwards various sets of sufficient  conditions  \`a la Muckenhoupt were found to be valid also in the matrix-valued context, one of the earliest such sets appeared  in 1997 in joint work with Nets Katz   \cite{KP}, see also  the 2005 unpublished manuscript \cite{NTV5}.
 Necessary and sufficient conditions for (uniform and individual) martingale transform and  well-localized dyadic operators were found  in 1999 and 2008 respectively by {Fedja Nazarov,  Sergei Treil, Sasha Volberg} \cite{NTV1,NTV4}, using Bellman function techniques, we will say more about this in Section~\ref{sec:martingale}.
 Long-time sought  necessary and sufficient conditions for two-weight boundedness of the Hilbert transform were found in 2014 by {Michael Lacey, Eric Sawyer,  Chun-Yen Shen, and Ignacio Uriarte-Tuero   \cite{L2,LSSU} for pairs of weights that do not share a common point-mass.
Corresponding quantitative estimates were obtained using   very delicate stopping time arguments. See also \cite{L3}.
 Improvements have since been obtained, relaxing the conditions on the weights, by the same authors and Tuomas Hyt\"onen \cite{Hyt4}.

\subsubsection{Two-weight estimates for the maximal function}
 In 1982 Eric Sawyer showed in \cite{S1} that  the maximal function  $M$ is bounded from $L^2(u)$ into $L^2(v)$  if and only if the following testing conditions\footnote{Nowadays called "Sawyer's testing conditions".} hold for the weights $u$ and $v$: there is a constant $C_{u,v}>0$ such that for all cubes $Q$
$$ 
\int_Q \big (M({ \mathbbm{1}_Q u^{-1}})(x)\big )^2v(x)\,dx \leq {C_{u,v} }u^{-1}(Q) \;\;\mbox{and} \;\;
\int_Q \big (M({ \mathbbm{1}_Q v})(x)\big )^2u^{-1}(x)\,dx \leq {C_{u,v} }v(Q).
$$
Sawyer also identified  necessary and sufficient conditions for  two-weight inequalities for certain positive operators, the fractional and Poisson integrals \cite{S2}, these results were of qualitative type.
In 2009, Kabe Moen  presented the first quantitative result \cite{Moe}, he proved that the two weight operator norm of $M$ is comparable to the constants  ${C_{u,v}}$ in Sawyer's result.
Note that Sawyer's testing conditions imply the following \emph{joint $\mathcal{A}_2$ condition}: 
$$[u,v]_{\mathcal{A}_2}:= \sup_{Q} \, \langle u^{-1}\rangle_Q \langle v \rangle_Q<\infty, \quad\mbox{{\small where}} \quad \langle v\rangle_Q:= v(Q) / |Q|$$
 In 2015, Carlos P\'erez and Ezequiel Rela  \cite{PzR} considered a particular case
when {$(u,v)\in \mathcal{A}_2$} and ${u^{-1}\in A_{\infty}}$ and showed the following  so-called mixed-type estimate
$$
\|M\|_{L^2(u)\to L^2(v)}\lesssim{[u,v]_{\mathcal{A}_2}^{\frac{1}{2}}[u^{-1}]_{A_{\infty}}^{\frac{1}{2}}}.
$$
In the one-weight setting, when $u=v=w$, one gets the following improved mixed-type estimate
 $$\|M\|_{L^2(w)\to L^2(w)} \lesssim {[w]_{A_2}^{\frac{1}{2}}[w^{-1}]_{A_{\infty}}^{\frac{1}{2}}}
  \leq   {[w]_{A_2}}.
  $$
The \emph{$A_{\infty}$ class} of weights is defined to be the union of all the $A_p$ classes of weights for $p>1$, the \emph{classical $A_{\infty}$ characteristic}  is given by
\[ [w]_{A_{\infty}^{{\rm cl}}} := \sup_{Q} \langle w\rangle_Q \exp ({-\langle \log w\rangle_Q}).\]
A weight $w$  is in $A_{\infty}$ if and only if $[w]_{A_{\infty}^{{\rm cl}}}<\infty$. An equivalent characterization is obtained using instead the Fujii-Wilson characteristic, defined by
\[ [w]_{A_{\infty}} :=\sup_Q\frac{1}{w(Q)}\int_Q M(w\chi_Q)(x) \,dx. \]
The Fujii-Wilson $A_{\infty}$ characteristic is smaller than the classical one \cite{BeRe}.
For mixed-type estimates  of similar nature for Calder\'on-Zygmund singular integral operators see \cite{HPz}.

For sharp weighted inequalities for fractional integral operators see \cite{LMPzTo}.

\section{Dyadic harmonic analysis}\label{sec:Dyadic-harmonic-analysis}

In this section, we introduce the elements of dyadic harmonic analysis and  the basic dyadic maximal function. More precisely we discuss dyadic grids (regular, random, adjacent) and Haar functions on the line, on $\R^d$, and  on spaces of homogeneous type.  As a first example,  illustrating the power of the dyadic techniques, we present Lerner's proof of Buckley's quantitative $L^p$ estimates for the  maximal function, which reduces, using the one-third trick, to estimates for the dyadic maximal function. We also describe, given dyadic cubes on spaces of homogeneous type, how to construct corresponding Haar bases, and briefly describe the Auscher-Hyt\"onen "wavelets" in this setting.

\subsection{Dyadic intervals, dyadic maximal functions} \label{sec:dyadics-maximal}
In this section we  recall the dyadic intervals and   the weighted dyadic maximal function on the line,  as well as
basic $L^p$ estimates for the dyadic maximal function.

\subsubsection{Dyadic intervals}

The \emph{standard dyadic  grid} $\mathcal{D}$  on $\R$ is the collection of  intervals
 of the form
$[k2^{-j},(k+1)2^{-j})$,  for all integers $k,j\in \Z$.
The dyadic intervals  are organized by generations: $\mathcal{D}=\cup_{j\in\Z}\mathcal{D}_j$, where
 ${I\in \mathcal{D}_j}$ if and only if ${|I|=2^{-j}}$.   Note that the larger $j$ is the smaller the intervals are. For each interval $J\in \mathcal{D}$ denote by $\mathcal{D}(J)$ the collection of dyadic intervals $I$ contained in $J$.
 
 The standard dyadic intervals  satisfy 
 the following properties,
 
\begin{itemize}
\item[-] (Partition Property) Each generation $\mathcal{D}_j$  is a partition of $\R$. 
\item[-] (Nested property)  If $I,J\in \mathcal{D}$ then
$ I\cap J =\emptyset, \quad  I\subseteq J$, or $ J\subset I.$ 
\item[-] (One parent property) If ${I\in \mathcal{D}_j}$ then there   is a unique interval ${\widetilde{I}\in \mathcal{D}_{j-1}}$, called the \emph{parent} of $I$, such that ${I\subset \widetilde{I}}$. The parent is twice as long as the child, that is $|\widetilde{I}|=2|I|$.
\item[-] (Two children property)
  Given $I\in \mathcal{D}_j$, there are two disjoint intervals 
${I_r, I_l \in \mathcal{D}_{j+1}}$ (the right and left children), such that $I=I_l\cup I_r$. 
\item[-] (Tower of dyadic intervals) Each point $x\in \R$ belongs to exactly one dyadic interval $I_j(x)\in \mathcal{D}_j$. The family $\{I_j(x)\}_{j\in\Z}$ forms a "tower" or "cone" over $x$. The union of the intervals in a "tower", $\cup_{j\in\Z}I_j(x)$,  is a "quadrant".
 \item[-] (Two-quadrant property) The origin, $0$, separates the positive and the negative dyadic interval, creating two  "quadrants". 
\end{itemize}

More generally, a \emph{dyadic grid} on $\R$ is a collection of intervals organized in generations with the  partition, nested, and two children properties. In this subsection we reserve the name $\mathcal{D}$ for the standard dyadic grid, however later on we will use $\mathcal{D}$ to denote a general dyadic grid.

The partition and nested properties are common to all dyadic grids, the one parent property  is a consequence of these properties.  The two children property is responsible for the name "dyadic", the equal length property is a consequence of choosing to subdivide in halves, and is in general not so important, one could  subdivide into two  children of different lengths, if the ratio is uniformly bounded we have a homogeneous or doubling dyadic grid. One can manufacture dyadic grids on the line where each interval has two equal-length children but there is no distinguished point and only one quadrant.  This is because given an interval in the grid, its descendants are completely determined, however we have two  choices for the parent, hence four choices for the grandparent, etc.  In \cite{LeN} dyadic grids are defined to  have one quadrant, such grids have the additional useful property that given any compact set there will be a dyadic interval containing it.

There are many variants, for example, we could subdivide each interval into a uniformly bounded number of children or into arbitrarily finitely many children.  In fact, there are regular dyadic structures on $\R^d$ where the  role of the intervals is played by cubes with sides parallel to the axes. In this case,   each cube in the dyadic grid is subdivided into $2^d$ congruent  children, see Section~\ref{sec:Haar-in-Rd}. We will also see that there are dyadic structures in spaces of homogeneous type, where each "cube" may have no more than a fixed number of children, but sometimes it will only have one child (itself) for several generations, see Section~\ref{sec:dyadic-cubes-SHT}.
In all cases the dyadic grids  provide a hierarchical  structure that allows for simplified arguments in this setting, the so-called "induction on scale arguments".

\subsubsection{Dyadic  Maximal Function}
Given a dyadic grid $\mathcal{D}$ on $\R^d$ and a weight $u$,
the  (weighted) dyadic maximal function $M^{\mathcal{D}}_u$ is defined like the maximal function $M$ except that instead of taking the supremum over all cubes in $\R^d$ with sides parallel to the axes we restrict to the dyadic cubes. This is often how one transitions from continuous to dyadic models.

More precisely, the \emph{weighted dyadic maximal function} with respect to a weight $u$ and a dyadic grid $\mathcal{D}$  on $\R^d$ is defined by 
$$
{ M^{\mathcal{D}}_uf(x):= \sup_{Q\in\mathcal{D}, Q\ni x}\frac{1}{u(Q)} \int_Q |f(y)| \,u(y)\, dy.}
$$
Here $u(Q):=\int_Q u(x)\, dx$.   When $u =1$ a.e. then $M^{\mathcal{D}}_1=:M^{\mathcal{D}}$.

The dyadic maximal function inherits boundedness properties from  the regular maximal function. This is clear once one notices that the dyadic  maximal function is trivially pointwise  dominated by the maximal function. However these properties are much easier to verify for the dyadic maximal function. We  now list three basic boundedness properties of the dyadic maximal function, with a word or two as how one can verify each one of them.

First, the dyadic maximal function, $M^{\mathcal{D}}_u$,  is of weak~$L^1(u)$ type, with constant one (independent of dimension).  This is an immediate corollary of the Calder\'on-Zygmund  lemma (a stopping time), no covering lemmas are required unlike the usual arguments for $M$.

Second, clearly $M^{\mathcal{D}}_u$ is bounded on $L^{\infty}(u)$ with constant one. 
Interpolation between the weak $L^1(u)$ and the $L^{\infty}(u)$ estimates  shows that  $M^{\mathcal{D}}_u$  is  bounded on $L^p(u)$ for all $p>1$. Moreover the following estimate holds with a constant independent of the weight $v$ and the dimension $d$,
\begin{equation}\label{Mv-on-Lp(v)}
\|M^{\mathcal{D}}_uf\|_{L^p(u)}\lesssim  p' \|f\|_{L^p(u)} \quad\mbox{where $\frac1p+\frac{1}{p'} =1$ and $p>1$.}
\end{equation}

Third, the dyadic maximal function is pointwise comparable to the maximal function. 
 We explain in Section~\ref{sec:one-third-trick} why this domination holds in the one-dimensional case ($d=1$).

\subsection{One-third trick and Lerner's proof of Buckley's result}
 We present the one-third trick on $\R$ and how it can be used to dominate the maximal function by  a sum of dyadic maximal functions.  The one-third trick appeared in print  in 1991 in Kate Okikiolu   characterization of subsets of rectifiable curves in $\R^d$ \cite[Lemma 1(b)]{Ok}, see \cite[Footnote p.32]{Cr} for fascinating historical remarks on the one-third trick. This was probably well-known among the John Garnett's school of thought see for example \cite{GJ}, and also by the Polish school specifically by Tadeusz Figiel \cite{Fi}.
 We illustrate how this principle can be used to recover Buckley's quantitative weighted $L^p$ estimate for the maximal function. 

\subsubsection{One-Third Trick}\label{sec:one-third-trick}
The families of intervals $\displaystyle{\mathcal{D}^i := \cup_{j\in\Z} \mathcal{D}^i_j}$, for $i=0,1,2$, where

$$ \mathcal{D}^i_j := \{ 2^{-j}\big ([0,1) + m + (-1)^j \frac{i}{3}\big ): m\in \Z\},$$
are dyadic grids satisfying partition, nested, and two equal children properties. We make four observations. First, 
when $i=0$ we recover the standard dyadic grid, $\mathcal{D}^0=\mathcal{D}$.  
Second, the grids  $\mathcal{D}^1$ and $\mathcal{D}^2$ are {nested}  but there is only one quadrant (the line $\R$).
Third, the grids,  $\mathcal{D}^i$,  for $i=0,1,2$  are as "far away" as possible from each other,  to be  made more precise in Example~\ref{eg:1/3grids}.
Fourth, given any finite interval $I\subset \R$, for \emph{at least two values} 
of $i=0,1,2$, there are $J^i\in \mathcal{D}^i$ such that {$I\subset J^i$, $ 3|I|\leq  |J^i| \leq 6|I|$}. In particular this implies that given $i\neq k$, $i,k=0,1,2$, there is at least one interval $J\in \mathcal{D}^i\cup \mathcal{D}^k$ such that $I\subset J$ and $3|I|\leq |J|\leq 6|I|$, and furthermore
\[ \frac{1}{|I|} \int_I |f(y)|\, dy \leq \frac{6}{|J|} \int_J |f(y)|\, dy.\]
This last observation allows us to dominate the maximal function $M$ by its dyadic counterpart.
In fact,  the following estimate holds,
\begin{equation}\label{Md-dominates-M}
M f(x) \leq 6 \big ( M^{\mathcal{D}}f(x) + M^{\mathcal{D}^1}f(x) \big ).
\end{equation}

 More precisely, for $i\neq k$ 
\begin{eqnarray*}
Mf(x) &=& \sup_{I\ni x} \frac{1}{|I|}\int_I|f(y)|\,dy \leq 6 \sup_{J\in \mathcal{D}^i\cup\mathcal{D}^k: J\ni x} \frac{1}{|J|} \int_J |f(y)|\, dy\\
& \leq & 6\max\{M^{\mathcal{D}^i}f(x), M^{\mathcal{D}^k}f(x)\} \leq 6\Big [M^{\mathcal{D}^i}f(x) + M^{\mathcal{D}^k}f(x)\Big ].
\end{eqnarray*}
In particular setting  $i=0$ and $k=1$, we obtain  \eqref{Md-dominates-M}.

 There is an analogue of the one-third trick in higher dimensions. In  $\R^d$ one can get by with  $3^d$  grids as is very well explained in \cite[Section 3]{LeN}, with $2^d$ grids \cite{HPz}, or, with $d+1$ grids and this is optimal, by  cleverly choosing the grids,  for $\R$ and for the $d$-torus see  \cite{Me}, for $\R^d$ and  $d>1$ see \cite{C}.

\subsubsection{Buckley's $A_p$  estimate  for the maximal function}\label{sec:Lerners-proof}
We illustrate the use of dyadic techniques paired with domination to recover Stephen Buckley's quantitative weighted $L^p$ estimate for the maximal function \cite{Bu}.  Namely for all $w\in A_p$  and $f\in L^p(w)$ 
\[ \|Mf\|_{L^p(w)}\lesssim [w]_{A_p}^{\frac{1}{p-1}} \|f\|_{L^p(w)}.\]
The beautiful  argument we present is due to Andrei Lerner
\cite{Le1}.

\begin{proof}[Lerner's Proof]
By the one-third trick suffices to check that  for $1<p<\infty$ there is a constant $C_p>0$ such that for all $w\in A_p$ and for all $f\in L^p(w)$ then
$$ \|M^{\mathcal{D}}f\|_{L^p(w)}\leq C_p{[w]_{A_p}^{\frac{1}{p-1}}} \|f\|_{L^p(w)},$$
independently of the dyadic grid  $\mathcal{D}$ chosen on $\R^d$.

For any dyadic cube $Q\in\mathcal{D}$, let {$A_p(Q)=w(Q)\big (\sigma(Q)\big )^{p-1}/|Q|^p$}, where  we denote  by {$\sigma:=w^{\frac{-1}{p-1}}$} the dual weight of $w$,  then
 \begin{eqnarray*} 
\frac{1}{|Q|} \int_Q |f(x)|\,dx
  & = & A_p(Q)^{\frac{1}{p-1}}\left [ \frac{|Q|}{w(Q)}
            \Big (\frac{1}{\sigma (Q)}\int_Q |f(x)|\, \sigma^{-1}(x)\, \sigma (x)\, dx\Big )^{p-1}\right ]^{\frac{1}{p-1}}\\
  & \leq &{ [w]_{A_p}^{\frac{1}{p-1}}}\left [ \frac{1}{w(Q)}
   \int_Q   \big (M^{\mathcal{D}}_{\sigma}  (f\sigma^{-1})(x) \, w^{-1}(x)\, w(x)\,dx \big )^{p-1}\right ]^{\frac{1}{p-1}}.
   \end{eqnarray*}
Taking the  supremum over $Q\in\mathcal{D}$ we obtain
 $$ M^{\mathcal{D}}f(x)\leq   { [w]_{A_p}^{\frac{1}{p-1}}} \Big [M_w^{\mathcal{D}}(M_{\sigma}^{\mathcal{D}}(f\sigma^{-1})^{p-1}w^{-1})(x)\Big ]^{\frac{1}{p-1}}.$$
   Computing the  $L^p(w)$ norm on both sides, recalling that $(p-1)p' = p$ where $\frac{1}{p} + \frac{1}{p'}=1$ and  carefully peeling off the maximal functions, we get
\begin{eqnarray*} 
 \|M^{\mathcal{D}}f\|_{L^p(w)}  & \leq &  [w]_{A_p}^{\frac{1}{p-1}}\,
  \|M^{\mathcal{D}}_w(M_{\sigma}^{\mathcal{D}}(f\sigma^{-1})^{p-1}w^{-1}) \|^{\frac{1}{p-1}}_{L^{p'}(w)}\\
 & \leq & [w]_{A_p}^{\frac{1}{p-1}}\,  \|M^{\mathcal{D}}_w\|^{\frac{1}{p-1}}_{L^{p'}(w)} \,
 \|M_{\sigma}^{\mathcal{D}}(f\sigma^{-1})\|_{L^p(\sigma )}\\
 & \leq &  [w]_{A_p}^{\frac{1}{p-1}} \, \|M^{\mathcal{D}}_w\|^{\frac{1}{p-1}}_{L^{p'}(w)}\,
     \|M_{\sigma}^{\mathcal{D}}\|_{L^p(\sigma)}\|f\sigma^{-1}\|_{L^p(\sigma)} \\
  & \leq & p^{\frac{1}{p-1}} p'\, {[w]_{A_p}^{\frac{1}{p-1}}}\|f\|_{L^p(w)},
\end{eqnarray*} 
where we used in the last line the uniform bounds \eqref{Mv-on-Lp(v)} of $M^{\mathcal{D}}_w$ on $L^{p'}(w)$ and $M^{\mathcal{D}}_{\sigma}$ on $L^p(\sigma )$.
 \end{proof} 
 
 Notice that in this argument  neither extrapolation nor interpolation are used.
 For extensions to two-weight inequalities  and to the  {fractional maximal function} see \cite{Moe}.

\subsection{Random dyadic grids on $\R$}

For the purpose of this section, a dyadic grid on $\R$ is a collection of intervals that are organized 
in generations,
 each generation provides  a partition of $\R$ and the family 
has the nested, one parent,  and two equal-length children per interval properties.
 Shifted and scaled  regular dyadic grid are dyadic grids. These are not the only ones, 
 there are other dyadic  grids, such as the ones defined for the one-third trick: $\mathcal{D}^1$ and $\mathcal{D}^2$.
The following parametrization will capture {\sc all} dyadic grids in $\R$ \cite{Hyt1}.

\begin{lemma}[Hyt\"onen 2008]
 For each
{scaling parameter   $r$} with {$1\leq r<2$},
and  shift  parameter $\beta\in \{0,1\}^{\Z}$, meaning  ${\beta=\{\beta_i\}_{i\in\Z}}$ with  
${\beta_i=0}$ or~1,  then  
  $\mathcal{D}^{r,\beta}:=\cup_{j\in\Z}\mathcal{D}_j^{r,\beta}$ is a dyadic grid. Where
$$
{\mathcal{D}_j^{r,\beta}:= r\mathcal{D}_j^{\beta}}, \;\; \mbox{and} \;\;
{\mathcal{D}_j^{\beta}:= x_j+ \mathcal{D}_j}, \;\;\mbox{with}\;\; x_j=\sum_{i>j}\beta_i2^{-i}.
$$
\end{lemma}

We shift by a different parameter $x_j$  at each level $j$, in a way that is consistent and preserves the nested property of the grid. Moreover the shift parameter $\beta_j=0,1$ for ${j\in\Z}$ encodes the information whether a base  interval at level $j$ will be the right or the left half of its parent.

\begin{example} Shifted and scaled  regular grids correspond to the shift parameter $\beta_i = 0$ for all $i<N$ $($or $\beta_i=1$ for all $i<N)$ for some integer $N$. These are the grids with two quadrants.
Comparatively speaking this set of dyadic grids is negligible, since it corresponds to a set of measure zero in parameter space described below. 
\end{example}

\begin{example}\label{eg:1/3grids}
The  $1/3$-shifted dyadic  grids introduced in the previous section correspond to  Hyt\"onen's   dyadic grids for $r=1$. More precisely, 
$$\mathcal{D}^i=\mathcal{D}^{1,\beta^i}\quad\mbox{for} \;\; i\in\{0,1,2\},$$
where  for all $j\in\Z$, $\beta^0_j\equiv 0$ (or $\equiv 1$), $\beta^1_j=\mathbbm{1}_{2\Z}(j)$, 
and $\beta_j^2=\mathbbm{1}_{2\Z+1}(j)$. 
\end{example}

We call these grids \emph{random dyadic grids} because we view the parameters $\beta_j$ and $r$ as independent  identically distributed random variables.
There is a very natural  probability space, say  $(\Omega, \mathbb{P})$ associated to the parameters, 
{$\Omega= [1,2) \times \{0,1\}^{\mathbb{Z}}.$}
Averaging in this context means calculating the expectation in this probability space, 
that is
{$$E_{\Omega} f = \int_{\Omega} f(\omega) \,d\mathbb{P}(\omega) = \int_1^2\int_{\{0,1\}^{\Z}} f(r,\beta) \, d\mu(\beta)\, \frac{dr}{r},$$
where $\mu$ stands for the canonical probability measure on $\{0,1\}^{\Z}$ which makes the coordinate functions $\beta_j$ independent with $\mu(\beta_j=0)=\mu(\beta_j=1) = 1/2$.

Random dyadic grids have been used for example in the study of 
 $T(b)$ theorems on  metric spaces with non-doubling measures \cite{NTV3, HMa} and 
of  ${\rm BMO}$ from dyadic ${\rm BMO}$ on the bidisc and product spaces of spaces of homogeneous type
\cite{PiW, CLW},
inspired   by celebrated work of John Garnett  and Peter Jones from the 80's \cite{GJ}.
 They have also been used in Hyt\"onen's representation theorem \cite{Hyt2} and  in the resolution  of the two-weight problem for the Hilbert transform \cite{LSSU, L2}.

\subsection{Haar bases}

Associated to dyadic intervals (or dyadic cubes) there is a very important collection of step functions, the Haar functions. In this section we recall the Haar bases on $\R$ and on $\R^d$, and some of their well-known properties.

\subsubsection{Haar basis on $\R$}

The  {\em Haar function} associated to an interval ${I}\subset \R$  is defined to be
$$
{h_I(x) := |I|^{-1/2}\big (\mathbbm{1}_{I_r}(x)-\mathbbm{1}_{I_l}(x)\big ),}
$$
where $I_r$ and $I_l$ are the right and left halves respectively of $I$, and the characteristic function $\mathbbm{1}_I(x)= 1$ if $x\in I$, zero otherwise. Haar functions have mean zero, that is,  $\int _{\R}h_I=0$, and they are normalized on $L^2(\R)$.

The Haar functions indexed on any dyadic grid $\mathcal{D}$, $\{h_I\}_{I\in\mathcal{D}} $, form a \emph{complete orthonormal system} of 
$L^2(\R)$ ({Haar 1910)}. In particular   for all $f\in L^2(\R )$, with $\langle f,g\rangle:=\int_\R f(x)\,\overline{g(x)}\,dx$,
$$ f  = \sum_{I\in\mathcal{D}} \langle f,h_I\rangle \, h_I .$$
You can find a complete proof of this statement in \cite[Chapter 9]{PW}.

 The Haar basis 
is an {unconditional basis} of $L^p(\R )$  and of {$L^p(w)$ if $w\in A_p$}
for $1<p<\infty$  \cite{TV}. 
This is deduced from the  boundedness properties of the  {martingale transform}, we will say more about this dyadic operator in Section~\ref{sec:martingale}.

The Haar basis constitutes the first example of a {wavelet basis}\footnote{An \emph{orthonormal wavelet basis} of $L^2(\R )$ is an orthonormal  basis  where all its elements are translations and dilations of a fixed function $\psi$, called the wavelet. More precisely, a function $\psi\in L^2(\R )$ is a \emph{wavelet} if and only if  the functions $\psi_{j,k}(x)=2^{j/2}\psi(2^jx-k)$ for $j,k\in\Z$ form an orthonormal basis of $L^2(\R )$.} and its corresponding {Haar multiresolution analysis} provides the canonical example of a multiresolution analysis \cite[Chapters 9-11]{PW}.

\subsubsection{Dyadic cubes and Haar basis on $\R^d$}\label{sec:Haar-in-Rd}

In $d$-dimensional Euclidean space  the regular {dyadic cubes} are  cartesian products of regular dyadic intervals of the same generation. More precisely, a cube $Q\in\mathcal{D}_j(\R^d)$ if and only if  $Q= I_1\times\dots\times I_d$, where  $I_n\in \mathcal{D}_j(\R )$ for $n=1,2,\dots, d$.
Each generation $\mathcal{D}_j(\R^d)$ is a partition of $\R^d$ and they form a nested grid, each cube has  one parent and  $2^d$ congruent children, and there are $2^d$ quadrants.  Had we used dyadic intervals with just one quadrant then the corresponding dyadic cubes in  $\R^d$ will also have only one quadrant. We denote $\mathcal{D}(\R^d)$ the collection of all dyadic cubes in all generations, that is, $\mathcal{D}(\R^d)=\cup_{j\in\Z}\mathcal{D}_j(\R^d)$. For $Q\in\mathcal{D}(\R^d)$ we denote $\mathcal{D}(Q)$ the set of dyadic cubes contained in $Q$.

For each dyadic cube $Q$ in $\R^d$ we can associate $2^d$ step functions, constant on each children of $Q$ by taking appropriate tensor products. More precisely,
for $Q\in\mathcal{D}(\R^d)$ and  $\vec{\epsilon}=(\epsilon_1,\dots,\epsilon_d)$, with $\epsilon_n=0$ or $1$, let  $$h^{\vec\epsilon}_Q(x_1,\dots, x_d) := h^{\epsilon_1}_{I_1}(x_1)\times \dots \times h^{\epsilon_d}_{I_d}(x_d), \quad $$
where for each dyadic interval $I$ we denote $h^0_I:=h_I$ and $h_I^1=|I|^{-1/2}\mathbbm{1}_I$.
Note that $h_Q^{\vec{1}}={{|Q|^{-1/2}}}{\mathbbm{1}_Q}$, where $\vec{1}=(1,1,\dots, 1)$. The remaining   ($2^d-1$) functions are the Haar functions associated to  the cube $Q$. The tensor product Haar functions
$h^{\vec\epsilon}_Q$, for $\vec{\epsilon}\neq \vec{1}$,  are supported on the corresponding dyadic cube $Q$, they have mean zero,  $L^2$ norm one, and they are constant on $Q$'s children. The collection  $\{h_Q^{\vec\epsilon}: \, \vec\epsilon\neq\vec1,\, \,Q\in\mathcal{D}(\R^d)\}$  is an orthonormal  basis of $L^2(\R^d)$, and an unconditional basis of $L^p(\R^d )$, $1<p<\infty$ (the Haar basis).  
Figure~\ref{fig:Haar-in-R2} and Figure~\ref{fig:haar-in-R3}
illustrate the Haar functions associated to a square in $\R^2$ and to a cube in $\R^3$ respectively.

\begin{figure}[ht]
\label{fig:Haar-in-R2}
\centering
\includegraphics[scale=0.5]{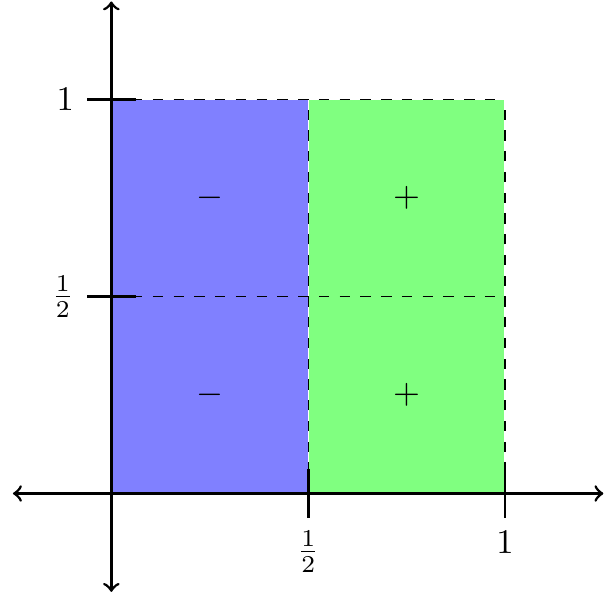}
\includegraphics[scale=0.5]{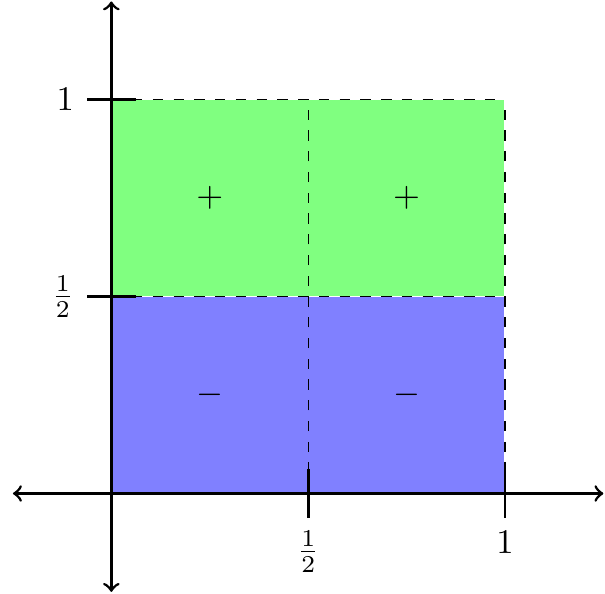}
\includegraphics[scale=0.5]{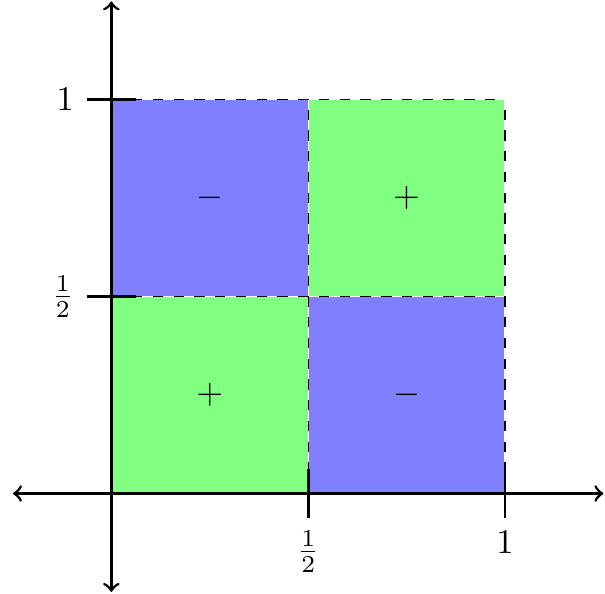}
\caption{The three Haar function associated to the unit square in $\R^2$.  \footnotesize{Figure kindly provided by David Weirich \cite{We}}.}
\end{figure}

\begin{figure}[ht]
\label{fig:haar-in-R3}
\centering
\includegraphics[scale=0.5]{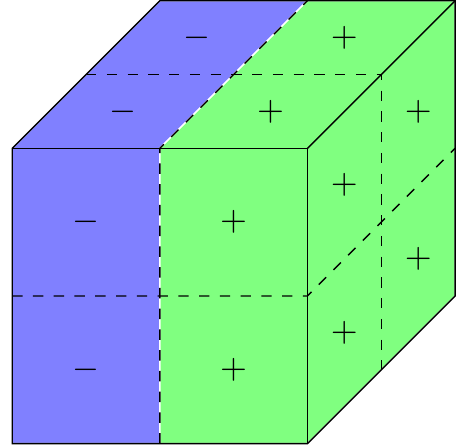}
\includegraphics[scale=0.5]{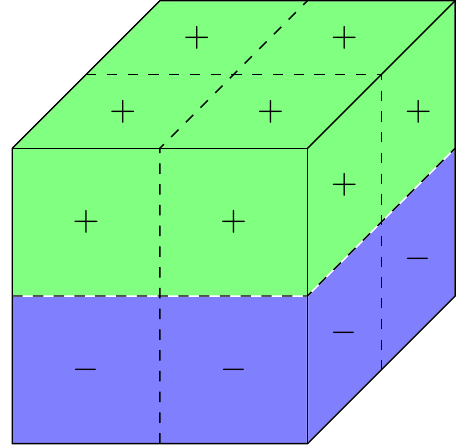}
\includegraphics[scale=0.5]{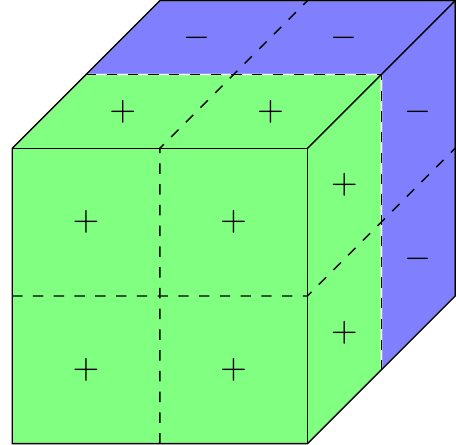}\\
\vspace{5mm}
\includegraphics[scale=0.5]{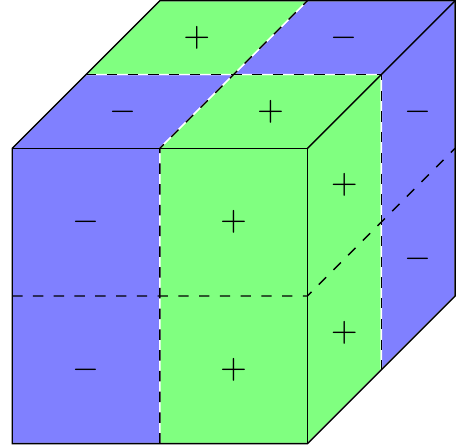}
\includegraphics[scale=0.5]{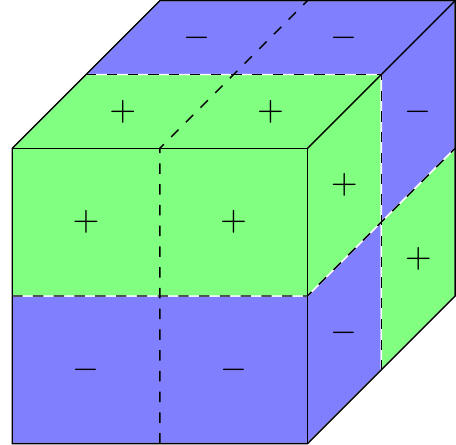}
\includegraphics[scale=0.5]{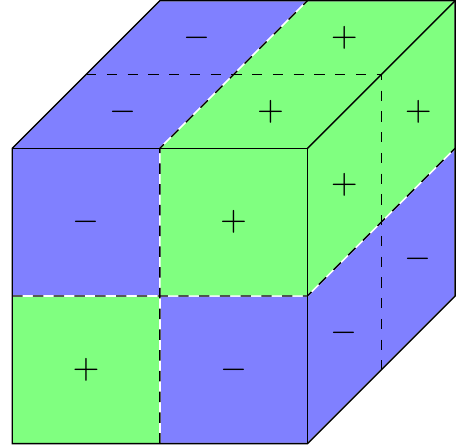}
\includegraphics[scale=0.5]{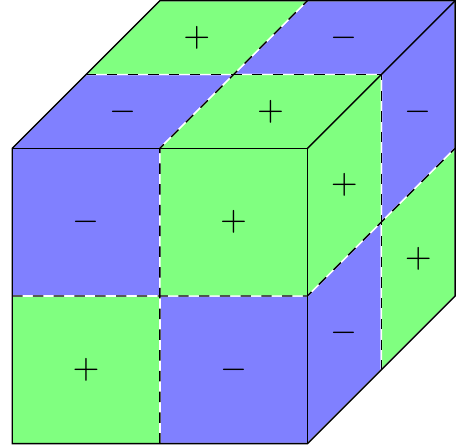}
\vspace{5mm}
\caption{The seven Haar functions associated to a cube in $\R^3$.  \footnotesize{Figure kindly provided by David Weirich \cite{We}}.}
\end{figure}

The tensor product construction just described  seems very rigid, it is very dependent on the geometry of the cubes and on the group structure of the Euclidean space $\R^d$.
{ Can we do dyadic analysis on other settings?}
The answer is a resounding {\sc yes}!!!!  One such setting  is on 
{spaces of homogeneous type}}
introduced by {Coifman and Weiss in the early 70s. 
 In Section~\ref{sec:dyadic-grids-Haar-functions-in-SHT} we will describe how to construct Haar basis on spaces of homogeneous type given suitable collections of "dyadic cubes" and arg\"ue why they constitute an orthonormal basis. This  argument  can be used to show that the Haar functions introduced in this section constitute an orthonormal basis of $L^2(\R^d)$.

\subsection{Dyadic analysis on spaces of homogeneous type}\label{sec:dyadic-grids-Haar-functions-in-SHT}
In this section we will define  spaces of homogeneous type.
We will present a generalization of the dyadic cubes  adapted to this setting. Given dyadic cubes we will show how to construct corresponding Haar functions, and briefly discuss the Auscher-Hyt\"onen wavelets on spaces of homogeneous type.

Before we start we would like to quote Yves Meyer.
 \begin{quote} \emph{One is amazed by the dramatic changes that occurred in analysis during the twentieth century. In the $1930$s complex methods and Fourier series played a seminal role. After many improvements, mostly achieved by the Calder\'on-Zygmund school, the action takes place today on spaces of homogeneous type. No group structure is available, the Fourier transform is missing, but a version of harmonic analysis is still present. Indeed the geometry is conducting the analysis.}
\end{quote}
\vskip -.1in
\hfill {Yves Meyer\footnote{Recipient of the 2017 Abel Prize. } in his preface to \cite{DH}.} \\

\subsubsection{Spaces of homogeneous type {\rm (SHT)}}
Let us first define what is a space of homogeneous type in the sense of Coifman and Weiss~\cite{CW}.

\begin{definition} [{Coifman, Weiss 1971}] For a set $X$, a triple 
{$(X,\rho ,\mu)$} is
a  {\rm space of homogeneous type  (SHT)}  in Coifman-Weiss's sense if 
\begin{enumerate}
\item  $\rho:X\times X\longrightarrow [0,\infty)$ is a {\rm quasi-metric} on~$X$, more precisely the following hold:
    \begin{itemize}
     \item[(a)] \emph{(positive definite)} $\;\rho (x,y) = 0$ if and only if $x = y$; 
     \item[(b)] \emph{(symmetry)} $\;\rho(x,y) = \rho (y,x) \geq 0$ for all $x$, $y\in X$;  
     \item[(c)]  \emph{(quasi-triangle inequality)} there exists constant $A_0\geq 1$ such that\[    \rho (x,y)   \leq{A_0 }\big (\rho (x,z) + \rho (z,y)\big ) \quad  \mbox{for all} \; \; x,y,z \in X.\]
     \end{itemize}
  \item  $\mu$ is a nonzero Borel regular\footnote{A measurable set $E$ of finite measure is \emph{Borel regular} if  there is a Borel set $B$ such that  $E\subset B$ and $\mu (E)=\mu (B)$.}  measure with respect to the topology induced by the quasimetric\footnote{The topology induced by a  quasi-metric is the largest topology $\mathcal{T}$ such that for each $x\in X$ the quasi-metric balls 
centered at $x$ form a fundamental system of neighborhoods  of $x$. Equivalently  a set $\Omega$ is \emph{open}, $\Omega \in \mathcal{T}$, if for each $x\in\Omega$ there exists $r>0$ such that the quasi-metric ball  $B(x,r)\subset \Omega$. A set in $X$ is \emph{closed} if it is the complement of an open set. }.
  \item Quasi-metric balls are $\mu$-measurable. 
  A {\rm quasi-metric ball} is the set $B(x,r):= \{y\in X: \rho (x,y)< r\}$, where $x\in X$ and  $r > 0$.
\item $\mu$ is a  {\rm doubling measure}, namely,  there  exists a constant $D_\mu\geq 1$ $($the doubling constant of the measure $\mu )$ such that for each quasimetric ball $B(x,r)$
 \[0< \mu(B(x,2r))\leq {D_\mu} \, \mu(B(x,r)) < \infty \quad  \mbox{for all} \;\; x\in X, r > 0.\]
\end{enumerate}
\end{definition}

Notice that  Condition~(4) implies that there are constants ${\omega} >0$
(known as an \emph{upper dimension} of $\mu$) and $C\geq 1$ such
that for all $x\in X$, $\lambda\geq 1$ and $r > 0$
   \[ \mu(B(x, \lambda r)) \leq C\lambda^{{\omega}} \mu(B(x,r)).\]
   In fact, we can choose $C=D_{\mu}\geq 1$ and $\omega = \log_2 D_{\mu}$.

The quasi-metric balls may {\sc not} be open in the topology induced by the quasi-metric, as Example~\ref{eg:non-open-balls} shows. Therefore the assumption that the quasimetric balls are $\mu$-measurable is not redundant.
The following example illustrate this phenomenon \cite{HK}

\begin{example}\label{eg:non-open-balls}
 Consider the set $X=\{-1\}\cup [0,\infty)$, the map $\rho: X\times X\to [0,\infty)$ given by
$\rho (-1,0)= \rho(0,-1)=1/2$ and $\rho(x,y)= |x-y|$ otherwise, and the measure $\mu(E) = \delta_{-1}(E) + m\big (E\cap [0,\infty)\big )$, where $m$ is the Lebesgue measure and $\delta_{-1}$ the point-mass at $x=-1$, that is $\delta_{-1} (E) = 0$ if $-1\notin E$ and $\delta_{-1}(E)=1$ if $-1\in E$. Then $\rho$ is not a metric since $\rho(1,-1)=2  > 3/2 = 1+1/2 = \rho(1,0)+\rho(0,-1)$, however $\rho$ is a quasi-metric and
the measure  $\mu$ is doubling. It is a good exercise to compute both the quasi-triangle constant of $\rho$ and  the doubling constant of $\mu$. Finally the ball $B(-1,1)=\{-1,0\}$ is not open because it does not contain any ball centered at 0 with positive radius $r$, since $[0,r)\subset B(0,r)$ and
the interval $[0,r)$ is not contained in $B(-1,1)$.
\end{example}

A couple of further remarks are in order.

  First, a given  quasi-metric $\rho$  may {\sc not} be  {H\"older regular}. 
Recall that $\rho$ is a \emph{H\"older regular quasi-metric} if  there are constants  $0<{\theta}<1$ and $C_0>0$ such that
 \[|\rho (x,y)-\rho (x',y)|\leq C_0 \rho (x,x')^{\theta}\big [ \rho (x,y)+\rho (x,y')\big ]^{{1-\theta}} \quad \forall x,x',y\in X.\]
Metrics are H\"older regular for any $0<\theta \leq 1$, $C_0=1$. The quasi-metric in Example~\ref{eg:non-open-balls} is not continuous let alone H\"older regular.  Quasi-metric balls for H\"older regular quasi-metrics  are always open.

Second, {Roberto Mac\'ias and Carlos Segovia  showed  in 1979 \cite{MS}  that given a space of homogeneous type  $(X,\rho , \mu)$ 
 there is an equivalent H\"older regular quasi-metric 
 $\rho '$ on $X$ and  some $\theta\in (0,1)$, and for which the measure $\mu$ is 1-\emph{Ahlfors regular}, more precisely,
 \[ \mu\big (B_{\rho '}(x,r)\big ) \sim r^{1}.\]
 
Here are some examples of spaces of homogeneous type.

\begin{itemize}
\item[-] $\R^n$, with the Euclidean metric  and the Lebesgue measure.
\item[-] $\R^n$ with  the Euclidean metric and an absolutely continuous measure with respect to the Lebesgue measure $d\mu = w\,dx$ where $w$ is a doubling weight 
(for example  $w$ could be an $A_{\infty}$ weight). 
\item[-] Quasi-metric spaces with $d$-\emph{Ahlfors regular measure}:
 $ \mu (B(x,r)) \sim r^{d}$ (e.g. Lipschitz surfaces, fractal sets, $n$-thick subsets of $\R^n$). More concretely, consider for example $X$ the four-corners Cantor set with the Euclidean metric and the one-dimensional Hausdorff measure, or consider $X$ the graph of a Lipschitz function $F:\R^n\to\R$ with the induced Euclidean metric and measure the volume of the set's "shadow",
  $\mu(E)=m\Big ( \{x\in\R^n\,: \, \big (x,F(x)\big )\subset E\} \Big )$ 
  where $m$ is the Lebesgue measure on $\R^n$.
\item[-] $C^{\infty}$ manifolds with doubling volume measure for geodesic balls.
\item[-] Nilpotent Lie groups $G$ with the left-invariant Riemannian metric and the induced measure
(e.g. Heisenberg group where $X$ is the boundary of the unit ball in $\mathbb{C}^n$, 
$\rho(z,w)=1-\overline{z}\cdot w$ and with surface measure).

\end{itemize}

The 2015  book  by {Ryan Alvarado and Marius Mitrea} \cite{AMi}  discusses in more detail many of these examples and  relies heavily on  the Mac\'ias-Segovia philosophy, meaning they consider equivalent classes of quasi-metrics knowing that among them they can choose a representative that is  H\"older regular  and for which the measure is Ahlfors regular.

\subsubsection{Dyadic cubes in {\rm SHT} }\label{sec:dyadic-cubes-SHT}
Systems of "dyadic cubes" were built by Hugo Aimar and Roberto Mac\'ias,
Eric Sawyer and Richard Wheeden,   and  Guy David in the 80's \cite{AiM, SW,Da}, and  by Michael Christ in the 90's 
\cite{Chr} on 
spaces of homogeneous type,  and by Tuomas Hyt\"onen and Anna Kairema in 2012  on geometrically doubling quasi-metric spaces \cite{HK} without reference to a measure. 

A \emph{geometrically doubling quasi-metric space } $(X, d)$ is one such that
 every quasi-metric ball of radius $r$ can be covered with at most $N$ quasi-metric balls of radius $r/2$ for some natural number $N$.

 \begin{example}
 Spaces of homogeneous type in the Coifman-Weiss sense  are  geometrically doubling \cite{CW}.
 \end{example}

Systems $\mathcal{D}$ of  dyadic cubes in spaces of homogeneous type or, more generally, on geometrically doubling spaces, are organized in disjoint generations $\mathcal{D}_k$, $k\in\Z$,  such that  $\mathcal{D}=\cup_{k\in\Z} \mathcal{D}_k$ and the following qualitative properties hold.
\begin{itemize} 
\item[(a)] Each generation  $\mathcal{D}_k$ is a  partition of $X$, so the cubes in a generation are pairwise disjoint and form a covering of $X$.
\item[(b)] The generations are nested, that is there is no partial overlap across generations.
\item[(c)] As a consequence, each cube has  unique ancestors in earlier generations.
\item[(d)] Dyadic cubes have  at most $M$ children  for some positive natural number $M$ (this is a consequence of the geometric doubling property).
\item[(e)] There exists a constant $\delta\in (0,1)$ such that for every dyadic cube in $\mathcal{D}_k$ there are
 inner and outer balls of radius roughly $\delta^k$ (the "sidelength" of the cube). \item[(f)] The outer ball  corresponding to a dyadic cube's child is inside its parent's outer ball.
\end{itemize}
Note that since $\delta\in (0,1)$ the larger $k$ is the smaller in diameter the cubes are. If $Q\in\mathcal{D}_k$ then its parent will be the unique cube $\widetilde{Q}\in \mathcal{D}_{k-1}$ such that  $Q\subset \widetilde{Q}$.

Furthermore, cubes can be constructed to have a "small boundary property" \cite{Chr,HK} which is very useful in applications. 

A quantitative and more precise statement  of the defining properties  for a dyadic system of cubes on geometric doubling metric spaces is  encapsulated in the following construction that appeared in \cite[Theorem 2.2]{HK}.

\begin{theorem}[Hyt\"onen, Kairema 2012]\label{thm:HK}
   Given $(X,d)$ a geometrically doubling quasi-metric space.  Suppose the  constants $C_0 \geq c_0 >1$ and
    $\delta\in (0,1)$ satisfy
      {$12 A_0^3 C_0\delta  \leq c_0$}.
    Given a set of points   $\{z_\alpha^k\, : \, \alpha \in \mathcal{A}_k\}$,  where $\mathcal{A}_k$ is a countable set of indexes, with the properties that
    \[d(z_{\alpha}^k,z_{\beta}^k)\geq c_0\delta_k \; (\alpha\neq \beta), \quad  \min_{\alpha\in\mathcal{A}_k}d(x,z_{\alpha}^k)< C_0\delta_k \;\mbox{for all} \; x\in X.\]
    For each $k\in \Z$ and $\alpha\in \mathcal{A}_k$ there exists  sets ${Q}_\alpha^{k,\circ}
    \subseteq Q_\alpha^k \subseteq   \overline{Q}_\alpha^k$ ---called open, half-open, and closed dyadic cubes---   such that:
   \begin{enumerate}
    \item $Q_\alpha^{k,\circ} \mbox{ and } \overline{Q}_\alpha^k
            \mbox{ are the interior and closure of } Q_\alpha^k, \mbox{ respectively}$;
    \item {\rm (nested)} $\mbox{if } \ell\geq k, \mbox{ then either } Q_\beta^\ell\subseteq
            Q_\alpha^k \mbox{ or } Q_\alpha^k
            \cap Q_\beta^\ell=\emptyset$; 
    \item {\rm (partition)} $ \quad X = \bigcup_{\alpha \in \mathcal{A}_k} Q_\alpha^k \;$
            for all $k\in\mathbb{Z}$;
    \item {\rm (inner/outer balls)} $\quad B(z_\alpha^k,c_1\delta^k)\subseteq Q_\alpha^k\subseteq
            B(z_\alpha^k,C_1\delta^k)\;$ where  $c_1 := (3 A_0^2)^{-1}c_0$
            and $C_1 := 2A_0C_0$;
    \item $\mbox{if } \ell \geq k \mbox{ and } Q_\beta^\ell\subseteq Q_\alpha^k,
            \mbox{ then } B(z_\beta^\ell,C_1\delta^\ell)\subseteq
            B(z_\alpha^k,C_1\delta^k)$.     
    \end{enumerate}
    
        The open and closed cubes $Q_\alpha^{k,\circ}$ and
    $\overline{Q}_\alpha^k$ depend only on the points
    $z_\beta^\ell$ for $\ell\geq k$. The half-open cubes
    $Q_\alpha^k$ depend on $z_\beta^\ell$ for $\ell\geq
    \min(k,k_0)$, where $k_0\in\mathbb{Z}$ is a preassigned
    number entering the construction.
\end{theorem}
The geometrically doubling condition implies  that sets of points $\{x^k_{\alpha}: k\in\Z, \alpha\in\mathcal{A}_k\}$
with the required separation properties exist and that the set  $\mathcal{A}_k$ is a countable set of indices for each $k\in\mathbb{Z}$.
The cubes in this construction are built as countable unions of quasi-metric balls, hence once a space of homogeneous type is given, the cubes will be measurable sets.

\begin{figure}[ht] \label{fig:HaarSHT}
\label{haar in SHT fig}
\centering
\includegraphics[scale=0.5]{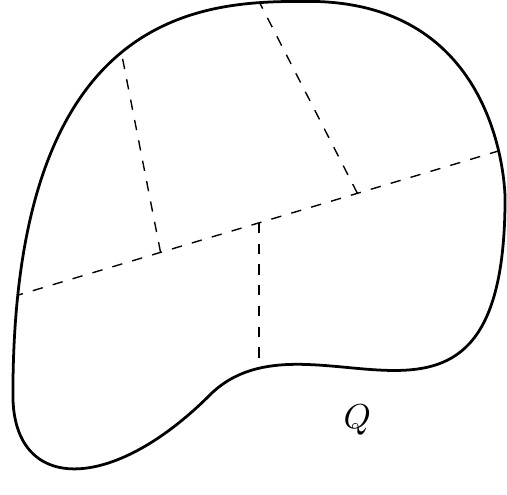}
\includegraphics[scale=0.5]{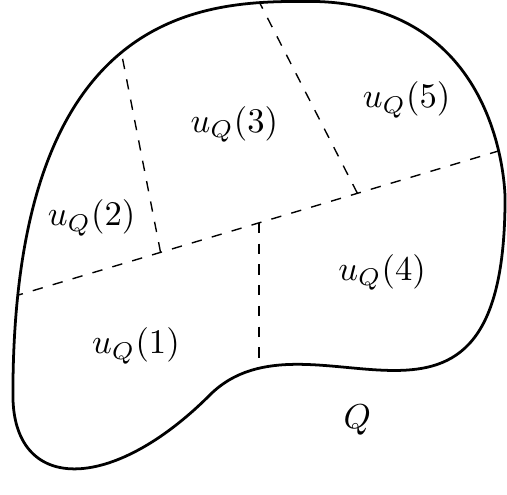}
\includegraphics[scale=0.5]{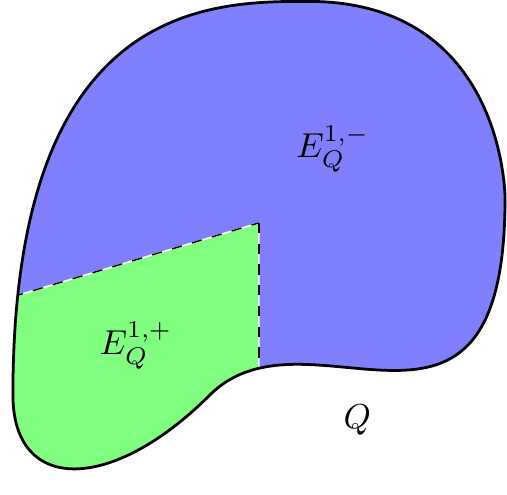}
\vspace{5mm}
\includegraphics[scale=0.5]{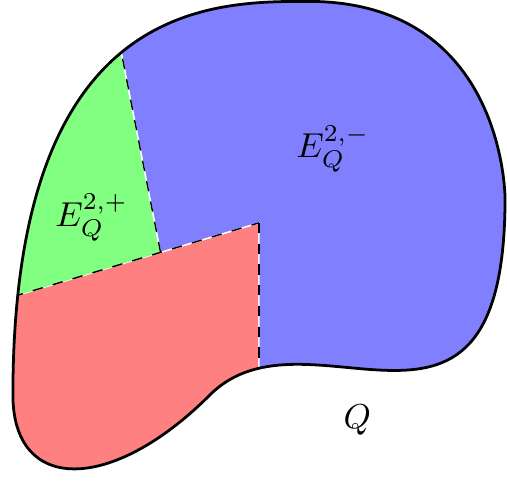}
\includegraphics[scale=0.5]{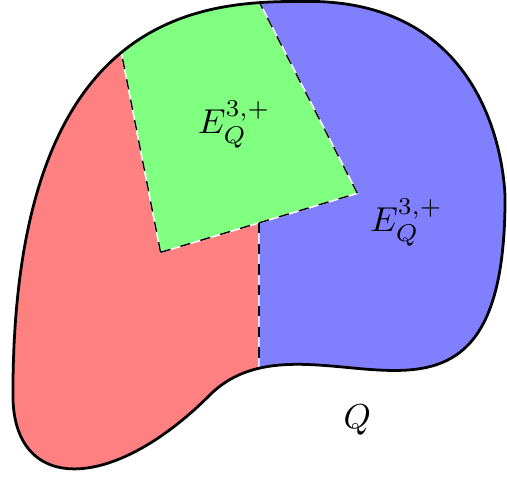}
\includegraphics[scale=0.5]{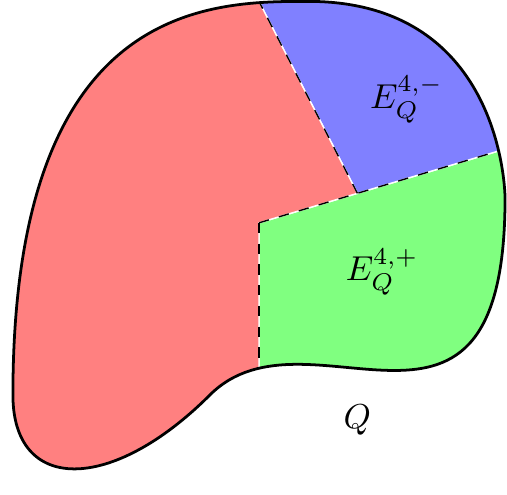}
\caption{The 4 Haar functions for a cube with 5 children in SHT. \footnotesize{Figures kindly provided by David Weirich \cite{We}}.}
\end{figure}

\subsubsection{Haar basis on {\rm SHT}}\label{sec:Haar-SHT} 
Given a space of homogeneous type  $(X, \rho, \mu)$ with a dyadic structure $\mathcal{D}$  given by  Theorem~\ref{thm:HK}, we can construct a system of Haar functions that will be an orthonormal basis of $L^2(X, \mu)$.

Given a  cube $Q\in \mathcal{D}$,  denote by ${\rm ch}(Q)$ the collection of dyadic children of $Q$, and by $N(Q)$ its cardinality, that is $Q$ has $N(Q)$ children.  
Let  $S_Q$ be the subspace  of $L^2(X,\mu)$  spanned by those  square integrable functions that are supported on $Q$ and  are constant on the children of $Q$. The subspace $S_Q$ has dimension $N(Q)$ as the characteristic functions of the children cubes normalized with respect to the $L^2$ norm, namely $\{\mathbbm{1}_{Q'}/\sqrt{\mu (Q')}\, : \, Q'\in {\rm ch}(Q)\}$,    form an orthonormal basis for $S_Q$.  The subspace $S_Q^0$ of $S_Q$ consisting of those functions that have mean zero, that is  $\int_Q f(x)\,dx=0$, will have one fewer dimension, namely  dim$(S^0_Q) = N(Q) -1$. 
 
 Given an enumeration of the children of $Q$, that is a  bijection $u_Q:\{1,2,\dots, N(Q)\}\to ch(Q)$, we will define recursively subsets of $Q$ that are unions of children of $Q$. More precisely at each stage we will remove one child according to the given enumeration,  let 
$E_Q^{1}:=Q$, given $E_Q^k\subset Q$, let  $E_Q^{k+1}=E_Q^k\setminus u_Q(k)$ for $k=1,2,\dots, N(Q)-1$.
We can split each of these sets into two  disjoint pieces, $E_Q^{i}:=E_Q^{i,+}\cup E_Q^{i,-}$ where $E_Q^{i,+}=u_Q(i)$, the child removed (green in Figure~\ref{fig:HaarSHT})  and $E_Q^{i,-}=E_Q^{i+1}$ (blue in Figure~\ref{fig:HaarSHT}). With this notation, 
 the Haar functions associated to the cube $Q$ and the enumeration $u_Q$ as illustrated in  Figure~\ref{fig:HaarSHT}, are supported on $Q$ and  are constant on the colored regions: positive on the green regions, negative on the blue  regions, and zero on the red regions, thus they are given  by
$$ h^i_Q(x) = {a} \mathbbm{1}_{E^{i,+}_Q}(x) - {b} \mathbbm{1}_{E^{i,-}_Q}(x), \quad 1\leq i \leq N(Q)-1,$$
where  the positive constants $a$ and $b$, dependent on the base cube $Q$ and the label $i$, are chosen to enforce $L^2$ normalization and mean zero. More precisely, the unknowns $a,b$ must satisfy  the system of two equations:
\begin{eqnarray*}
  \int_Q |h_Q^i(x)|^2 \,d\mu = a^2 \mu (E^{i,+}_Q) + b^2 \mu (E^{i,-}_Q) & = &1\\
  \int_Q h_Q^i(x)\, d\mu =a\, \mu (E^{i,+}_Q) - b\, \mu (E^{i,-}_Q) &=& 0.
 \end{eqnarray*}
 Solving the system of equations we get the positive solutions 
 $$  { a=\sqrt{ {\mu (E_Q^{i,-})}/ \big ({\mu (E_Q^{i})\, \mu (E_Q^{i,+})}}\big )}, \quad
{b=\sqrt{ {\mu (E_Q^{i,+})}/ \big ({\mu (E_Q^{i})\, \mu (E_Q^{i,-})}}\big )}.$$
Note that the doubling condition on the measure $\mu$ ensures $\mu(Q)>0$ for all $Q\in\mathcal{D}$, and hence also $\mu (E_Q^i)>0$ for all labels $i$.

The Haar basis consists of all functions $h_Q^i$  where $Q\in \mathcal{D}$ and $i=1,2,\dots , N(Q)-1$. Note that a cube may not subdivide for a while, meaning that it could have just one child, itself, for several generations or forever. In the former case we wait until we subdivide to define the subspace $S_Q^0$, in the later case we let $S_Q^0$ be the trivial subspace.

By construction for each $Q\in \mathcal{D}$ the collection $\{h^i_Q \, : \, i=1,\dots, N(Q)-1\}$ is normalized on $L^2(X,\mu )$, each Haar function has mean zero, and by the nested property of the dyadic cubes it is easy to verify this  is an orthonormal family. No matter what enumeration for ch$(Q)$ we use we will get  each time an orthonormal basis of $S^0_Q$. 
The orthogonal projection onto $S^0_Q$ of a square integrable function $f$ is independent of the  orthonormal basis chosen on $S^0_Q$. Given $x\in Q$ choose an enumeration so that $x\in u_Q(1)=:R\in \mbox{ch}(Q)$ then 
$$\mbox{Proj}_{S_Q} f(x) = \langle f, h^1_Q\rangle_{\mu} \, h^1_Q(x) = \langle f\rangle^{\mu}_R - \langle f\rangle^{\mu}_Q,$$
where $\langle f,g\rangle_{\mu }$ denotes the inner product in $L^2(X,\mu)$ and $\langle f\rangle^{\mu}_Q$ denotes the $\mu$-average of $f$. The first equality holds by support considerations, since $h^i(x)=0$ for all $i>1$ by the choice of the enumeration,  the second equality is now a simple calculation by substitution.

Using a telescoping sum argument one can verify that  completeness of  the Haar basis on $L^2(\mu)$ hinges on the following limits holding in the $L^2(\mu )$ sense:
\begin{eqnarray*} \lim_{j\to\infty} E_j^{\mu} f & {=}&  f,\\
 \lim_{j\to\infty} E_j^{\mu} f& {=}&  0,
 \end{eqnarray*}
 where $E_j f :=  \langle f\rangle^{\mu }_Q$ ,  with $x\in Q\in \mathcal{D}_j$, or $E_jf=\sum_{Q\in\mathcal{D}_j}  \langle f\rangle^{\mu }_Q \, \mathbbm{1}_Q$. That the limits do hold can be justified by martingale theory \cite{Hyt3, Mul}, in fact they do hold in $L^p(X,\mu)$ for $1<p<\infty$. The pointwise convergence a.e. of the averages to $f$ as $j$ goes to infinity is a consequence of the Lebesgue differentiation theorem which holds because the measure is assumed to be Borel regular, see \cite[Section 3.3]{AMi}.

Haar-type bases for $L^2(X,\mu)$ have been constructed in
general metric spaces, and the construction, along the lines described here,  is well known to
experts. 
 Haar-type wavelets associated to nested partitions in
abstract measure spaces were constructed in 1997 by  Girardi and Sweldens \cite{GS}.
 For the case of spaces of homogeneous type there is a lot of work related to Haar bases done in Argentina  this millennium,  specifically  by Hugo Aimar and collaborators Osvaldo Gorosito, Ana Bernardis, Bibiana Iaffei, and Luis Nowak  \cite{AiG,Ai,AiBI,AiBN1,AiBN2}, all descendants of Eleonor Harboure.
Haar functions have been used in geometrically doubling metric spaces  \cite{NRV}.
 For the case of  a geometrically doubling quasi-metric space $(X,\rho)$,
with a positive Borel regular measure~$\mu$, see \cite{KLPW}.

\subsubsection{Random dyadic grids, adjacent dyadic grids, and wavelets on {\rm SHT}} 
The counterparts of the random dyadic grids and the one-third trick have been identified in the general setting  of geometrically doubling quasi-metric spaces by Tuomas Hyt\"onen and his students and collaborators. Using them, {Pascal Auscher and Tuomas  Hyt\"onen } constructed  in 2013 a remarkable {orthonormal basis} of $L^2(X, \mu)$ \cite{AH, AH2}.

 A  notion of random dyadic grids can be introduced  on geometrically doubling quasi-metric spaces  $(X,d)$  by randomizing the order relations in the construction of the Hyt\"onen-Kairema cubes  \cite{HMa,HK}.
 In 2014, Tuomas Hyt\"onen and Olli Tapiola modified the randomization to improve upon Auscher-Hyt\"onen wavelets in metric spaces \cite{HTa}. 
 A different randomization can be found in \cite{NRV}.  

 One can find finitely many  adjacent families of Hyt\"onen-Kairema dyadic cubes, $\mathcal{D}^t$ for $t=1,\dots , T$,  with the same parameters, that play the role of the 1/3-shifted dyadic grids in~$\R$. The main property the adjacent families of dyadic cubes have is that  given any ball $B(x,r) \in X$, with $r\sim \delta^k$, then there is $t\in\{1,2,\dots, T\}$ and a cube in the $t$-grid  and in the $k$th-generation,  $Q\in\mathcal{D}^t_k$,  such that
$B(x,r)\subset Q\subset B(x, Cr)$,  where $C>0$ is a geometric constant only dependent on the quasi-metric and geometric doubling parameters of $X$ \cite{HK}. 
Furthermore, given a $\sigma$-finite measure $\mu$ on $X$, the adjacent dyadic systems can be chosen so that all cubes have small boundaries: $\mu (\partial Q) =0$ for all $Q\in \cup_{t=1}^T \mathcal{D}^t$  \cite{HK}.

 Given nested maximal sets $\mathcal{X}^k$ of $\delta^k$-separated points in $X$ for $k\in\mathbb{Z}$, let $\mathcal{Y}^k:=\mathcal{X}^{k+1}\setminus\mathcal{X}^k$ and  relabel points in $\mathcal{Y}^k$ by $y^k_{\alpha}$. To each point {$y^k_{\alpha}$},  Auscher and Hyt\"onen associate a {wavelet function $\psi^k_{\alpha}$} (a linear spline) of regularity $0< {\eta} <1$ that is morally supported near $y_{\alpha}^k$ at scale $\delta^k$, with mean zero and some smoothness. More precisely, 
these functions are not compactly supported, but have exponential decay away from the base cube $Q^k_{\alpha}$, 
and they have  H\"older regularity exponent $\eta>0 $, where $\eta$ depends only on $\delta$ and on some 
 finite quantities needed for extra labeling of the random dyadic grids used in the construction of the wavelets.
 The number  of indexes $\alpha$ so that $y_{\alpha}^k\in \mathcal{Y}_{k}$ 
for each  $Q_{\alpha}^k$ is exactly  {$N(Q_{\alpha}^k) - 1$}, }  where recall that $N(Q_\alpha^k)$ denotes the number of children of $Q_{\alpha}^k$. This is the right number of wavelets per cube $Q^k_{\alpha}$ if our intuition is to be guided by the constructions of the Haar functions.
The precise nature of these wavelets is   detailed in  \cite[Theorem 7.1]{AH}.

Furthermore, the functions $\{\psi^k_{\alpha}\}_{k\in\Z, \alpha\in \mathcal{Y}^k}$ form an unconditional basis on $L^p(X)$ for all $1<p<\infty$ and the following wavelet expansion is valid in $L^p(X)$,
\[    f(x)
    = \sum_{k\in\mathbb{Z}}\sum_{y_{\alpha}^k \in \mathcal{Y}^k}
        \langle f,\psi_{\alpha}^k \rangle \, \psi_{\alpha}^k(x).\]

Hyt\"onen and Tapiola} were able to  build such wavelets for all $0<\eta<1$ in the context of  metric spaces \cite{HTa}. It is still an open problem to construct smooth wavelets that are compactly supported. These wavelets have been used to study Hardy and  ${\rm BMO}$ spaces on product spaces of homogeneous type, as well as their dyadic counterparts \cite{KLPW}.

\section{Dyadic operators, weighted inequalities, and Hyt\"onen's representation theorem}\label{sec:Dyadic-operators}

In this section we introduce the model dyadic operators: the martingale transform, the dyadic square function, the dyadic paraproduct, Petermichl's Haar shift operator, and  Haar shift operators of arbitrary complexity. All ingredients in Hyt\"onen's proof of the $A_2$ conjecture \cite{Hyt2}. We will state the known quantitative one- and two-weight inequalities for these dyadic operators. We end the section with Hyt\"onen's representation theorem in terms of Haar shift operators of arbitrary complexity, dyadic paraproducts and adjoints of  dyadic paraproducts  over random dyadic grids, valid for all Calder\'on-Zygmund operators and  key to the resolution of the $A_2$ conjecture.

 \subsection{Martingale transform}\label{sec:martingale}
Let $\mathcal{D}$ denote a dyadic grid on $\R$, the Martingale transform is the linear operator  formally defined as
$$
   T_{\sigma} f (x):= \sum_{I\in\mathcal{D}}  \sigma_I \, \langle f, h_I\rangle \, h_I(x), 
\quad\quad  \mbox{where} \quad  \sigma_I=\pm 1. 
 $$
 
 This is  a constant Haar multiplier in analogy to Fourier multipliers, where here the Haar coefficients are modified  multiplying them by uniformly bounded constants, the \emph{Haar symbol} $\{\sigma_I\, :\, I\in\mathcal{D}\}$  (in this case arbitrary changes of sign). The martingale transform is bounded on $L^2(\R )$, in fact it is an isometry  on $L^2(\R)$ by Plancherel's identity, that is
 $\|T_{\sigma}f\|_{L^2}= \|f\|_{L^2}$.
 
The martingale transform is a good toy model for Calder\'on-Zygmund singular integral operators such as the Hilbert transform. Suffices to recall that  on Fourier side the Hilbert transform is a Fourier multiplier with Fourier symbol $m_H(\xi)=-i\, {\rm sgn} (\xi)$.  
Compare the Fourier transform of the Hilbert transform and the "Haar transform" of the martingale transform, namely,
$$\widehat{Hf}(\xi)=-i\,\mbox{sgn}(\xi) \,\widehat{f}(\xi) \quad \mbox{and} \quad \langle T_{\sigma}f,h_I\rangle = \sigma_I\, \langle f,h_I\rangle.$$

   Unconditionality of the Haar basis  on $L^p(\R )$ follows from uniform (on the choice of signs~$\sigma$) boundedness of  the martingale transform $T_{\sigma}$ on
$L^p(\R )$. More precisely  for all $f\in L^p(\R )$    
\[  \sup _{\sigma} \| T_{\sigma} f\|_{L^p}\lesssim_p \|f\|_{L^p}. \]
  This was proven by {Donald Burkholder in 1984}, he also found the optimal constant $C_p$ in work that can be described as the precursor of the (exact) Bellman function method  \cite{Bur}.

Unconditionality  of the Haar basis on ${L^p(w)}$ when  ${w\in A_p}$ follows
 from the uniform boundedness of $T_{\sigma}$ on
 $L^p(w)$, this was proven in 1996 by Sergei Treil and  Sasha Volberg  \cite{TV}.

 \subsubsection{Quantitative weighted inequalities for the martingale transform}\label{sec:martingale-transform}
Quantitative one- and two-weight inequalities are known for the martingale transform. In fact, the $A_2$ conjecture (linear bound) was proven by Janine Wittwer in 2000  and  necessary and sufficient conditions for two-weight  uniform (on the symbol $\sigma$) $L^2$ boundedness were identified by Fedja Nazarov,  Sergei Treil, and  Sasha Volberg in 1999. We present now the precise statements.

Sharp linear  bounds on $L^2(w)$  when $w$ is an $A_2$ weight   are known \cite{W1}. More precisely, for all $\sigma$ there is $C>0$ such that for all $w\in A_2$ and all $f\in L^2(w)$
\[ \|T_{\sigma}f\|_{L^2(w)}\leq C[w]_{A_2}\|f\|_{L^2(w)}.\]

 Sharp extrapolation gives optimal bounds on $L^p(w)$ when $w$ is an $A_p$  weight \cite{DGPPet}. More precisely,  for all $\sigma$ there is a constant $C_p>0$ such that for all $w\in A_p$ and $f\in L^p(w)$
\[ \|T_{\sigma}f\|_{L^p(w)}\leq C_p[w]_{A_p}^{\max\{1,\frac{1}{p-1}\}}\|f\|_{L^p(w)}.\]
 
 Necessary and sufficient conditions on pairs of weights $(u,v)$ are known ensuring two weight boundedness \cite{NTV1}. More precisely,  
 
 \begin{theorem}[Nazarov, Treil, Volberg 1999] \label{thm:NTVmartingale}
 The martingale transforms $T_{\sigma}$ are
uniformly $($on $\sigma)$  bounded  from $L^2(u)$ to $L^2(v)$ if and only if the following conditions hold simultaneously:
  \begin{itemize}
    \item[(i)] $(u,v)$ is in \emph{joint  dyadic $\mathcal{A}_2$}. 
   Namely  $[u,v]_{\mathcal{A}_2}:= \sup_{I\in\mathcal{D}} \langle u^{-1}\rangle_I \langle v\rangle_I <\infty$.
    \item[(ii)] $\{|I|\,|\Delta_I (u^{-1})|^2\langle v\rangle_I\}_{I\in\mathcal{D}}$ is a \emph{$u^{-1}$-Carleson sequence}. 
    \item[(iii)] $\{|I|\,|\Delta_Iv|^2\langle u^{-1}\rangle_I\}_{I\in\mathcal{D}}$ is a \emph{$v$-Carleson sequence}  (dual condition). 
    \item[(iv)] The {positive dyadic operator} $T_0$ is bounded from $L^2(u)$ into $L^2(v)\,.$ Where
\vskip -.1in
          $$ 
            T_0f(x):=\sum_{I\in\mathcal{D}}{\frac{\alpha_I}{|I|}}\, \langle f \rangle_I \, \mathbbm{1}_I(x)\,,
          $$
     with  $\,\alpha_I:=\big (|\Delta_Iv|/\langle v\rangle_I\big )\big (|\Delta_I(u^{-1})|/\langle u^{-1}\rangle_I\big )|I|\,,$ 
     and        $\,\Delta_I v:= \langle v\rangle_{I_+} - \langle v\rangle_{I_-}.$
\end{itemize}    
\end{theorem}
      A sequence $\{\lambda_I\}_{I\in\mathcal{D}}$ is \emph{$v$-Carleson} if and only if  there is constant $B>0$  such that   $\sum_{I\in\mathcal{D}(J)} \lambda_I \leq B v(J)$  for all $J\in \mathcal{D}$.
 The smallest constant $B$ is called the \emph{intensity} of the sequence.     
 When $u=v=w\in A_2$ then  (i)-(iii)  hold,  and by Example~\ref{examples-Carleson-sequences} the sequence $\{\alpha_I\}_{I\in\mathcal{D}}$ is a 1-Carleson sequence implying (iv).  

In 2008, Nazarov, Treil, and Volberg found necessary and sufficient conditions for two-weight boundedness of individual martingale transforms  and other well-localized operators \cite{NTV4}, see also \cite{Vu1}.

\subsection{Dyadic square function}\label{sec:dyadic-square-function}

The dyadic square function is the sublinear operator formally defined as
$$
(S^{\mathcal{D}}f)(x) :=   \bigg (\sum_{I\in\mathcal{D} } \frac{|\langle f, h_I\rangle |^2}{|I|}\, \mathbbm{1}_I(x)\bigg )^{1/2}.
$$
 The dyadic square function is an isometry on $L^2(\R )$, as a calculation quickly reveals, namely $\|S^{\mathcal{D}}f\|_{L^2}=\|f\|_{L^2}$.
 It is also bounded on $L^p(\R )$ for  $1<p<\infty$,  furthermore
    $$ 
          \|S^{\mathcal{D}} f\|_{L^p} \sim \|f\|_{L^p}.
     $$
   This result plays the role of \emph{Plancherel on $L^p$} (Littlewood-Paley theory). 
   It readily implies boundedness of $T_{\sigma}$   on $L^p(\R )$  since $S^{\mathcal{D}}(T_{\sigma}f) = S^{\mathcal{D}} f$, as follows,
   $$   \|T_{\sigma}f\|_{L^p}\sim \|S^{\mathcal{D}}(T_{\sigma}f)\|_{L^p} = \|S^{\mathcal{D}}f\|_{L^p}\sim \|f\|_{L^p}.$$
   
   A somewhat convoluted argument can be done to prove the $L^p$ boundedness of the dyadic square function. First  prove $L^2(w)$ estimates for  $A_2$ weights $w$, second  extrapolate to get  $L^p(w)$  estimates for $A_p$ weights $w$,  and third set $w\equiv 1 \in A_p$. Stephen Buckley has a very nice and elementary argument showing boundedness of the dyadic square function on  $L^2(w)$ when $w$ is an $A_2$ weight  \cite{Bu2} or see \cite[Section 2.5.1]{P1}.  One can track the dependence on the weight an get a 3/2 power on the $A_2$ characteristic of the weight \cite[Section 5]{BeCMP}, far from the optimal linear dependence discussed in Section~\ref{sec:one-weight-square-function}

\subsubsection{One-weight estimates for $S^{\mathcal{D}}$}\label{sec:one-weight-square-function}
Quantitative one-weight inequalities are known for the dyadic square function. The $A_2$ conjecture (linear bound) was proven by Sanja Hukovic, Sergei Treil, and Sasha Volberg in 2000 \cite{HTV} and 
the  reverse estimate was proven by Stefanie Petermichl and Sandra Pott in 2002 \cite{PetPo}.  

We present now the precise statements.
 For all weights $w\in A_2$ and functions $f\in L^2(w)$ 
    $$ 
      {[w]^{-\frac12}_{A_2}}\|f\|_{L^2(w)} \lesssim \|S^{\mathcal{D}}f\|_{L^2(w)} \lesssim {[w]_{A_2}}\|f\|_{L^2(w)} 
    $$
 The direct and reverse estimates  on $L^2(w)$ for the dyadic square function play the role of  {Plancherel on $L^2(w)$}.
  We can use these inequalities to  obtain  $L^2(w)$ bounds for the martingale transform $T_{\sigma}$ of the form ${[w]_{A_2}^{3/2}}$. However the optimal bound is   linear \cite{W1}, as we already mentioned in Section~\ref{sec:martingale-transform}.

 Boundedness on $L^2(w)$ for all weights $w\in A_2$ implies by extrapolation boundedness on $L^p(\R )$ (and on $L^p(w)$ for all $w\in A_p$).  However sharp extrapolation will only yield the optimal power for $1<p\leq 2$, if one starts with the optimal linear bound on $L^2(w)$. Not only   $S^{\mathcal{D}}$ is bounded on $L^p(w)$ if $w\in A_p$, moreover, for $1<p<\infty$ and for all $w\in A_p$ and $f\in L^p(w)$
 $$
     \|S^{\mathcal{D}}f\|_{L^p(w)} \lesssim_p{[w]_{A_p}^{\max\{\frac{1}{2},\frac{1}{p-1}\}} }      \|f\|_{L^p(w)}.
    $$
   The  power $\max\{{1}/{2},{1}/({p-1})\}$ is  optimal. It corresponds to \emph{sharp extrapolation} starting at {$r=3$ } with square root power   \cite{CrMPz2}. More precisely, for all $w\in A_3$ and $f\in L^3(w)$, 
  \[ \|S^{\mathcal{D}} f\|_{L^3(w)}\lesssim [w]_{A_3}^{\frac12}\|f\|_{L^3(w)}.\]
  This estimate is valid more generally for \emph{Wilson's intrinsic square function} \cite{Le2, Wil2}.

Sharp extrapolation from the reverse estimate on  $L^2(w)$   also yields the following  reverse estimate on $L^p(w)$ for all $w\in A_p$ and $f\in L^p(w)$,
\[ \|f\|_{L^p(w)} \lesssim_p [w]_{A_p}^{\frac{1}{2}\max\{1, \frac{1}{p-1}\}} \|S^{\mathcal{D}}f\|_{L^p(w)}.\]
This estimate can be improved, using deep estimates of Chang, Wilson, and Wolff \cite{ChaWilWo}  for all $p>1$ to the following $1/2$ power of the smaller Fujii-Wilson $A_{\infty}$ characteristic,
\[ \|f\|_{L^p(w)} \lesssim_p [w]_{A_{\infty}}^{\frac{1}{2}}\|S^{\mathcal{D}}f\|_{L^p(w)}.\]
This estimate is better in the range $1<p<2$ where the power is $1/2$ instead of $1/2(p-1)$.

For future reference, we can compute precisely the weighted $L^2$ norm of $S^{\mathcal{D}}f$ as follows
 $$\|S^{\mathcal{D}}f\|_{L^2(w)}^2=\sum_{I\in\mathcal{D}} |\langle f,h_I\rangle |^2\langle w\rangle_I.$$

\subsubsection{Two-weight estimates for $S^{\mathcal{D}}$} 
Two-weight inequalities are understood for the dyadic square function. The necessary and sufficient conditions for two-weight   $L^2$ boundedness are known  \cite{NTV1}. Qualitative (mixed) estimates have been found by different authors, these estimates reduce to the linear estimate in the one-weight case.  We present now the precise statements.

\begin{theorem}[Nazarov, Treil, Volberg 1999]\label{thm:NTVdyadicS}
The dyadic square function $S^{\mathcal{D}}$ is bounded from $L^2(u)$ into $L^2(v)$ if and only if the following conditions hold simultaneously:
   \begin{itemize}
      \item[(i)]   $\;(u,v)\in \mathcal{A}_2$ (joint dyadic $\mathcal{A}_2$).
       \item[(ii)] $\; \{|I|\,|\Delta_I u^{-1}|^2\langle v\rangle_I\}_{I\in\mathcal{D}}$ is a {$u^{-1}$-Carleson sequence}        
       with intensity $C_{u,v}$.
    \end{itemize}
 \end{theorem}   
   Notice  (ii) is  a localized "testing condition" on test functions $u^{-1}\mathbbm{1}_J$ . Also note that the necessary and sufficient conditions (i)-(iii) in Theorem~\ref{thm:NTVmartingale} for the martingale transform can be now replaced by
   \begin{itemize}
   \item[(i)] $S^{\mathcal{D}}$ is bounded from $L^2(u)$ into $L^2(v)$.
   \item[(ii)] $S^{\mathcal{D}}$ is bounded from $L^2(v^{-1})$ into $L^2(u^{-1})$.
   \end{itemize}
  This is because $(u,v)\in\mathcal{A}_2$ if and only if $(v^{-1},u^{-1})\in\mathcal{A}_2$. 
  
A quantitative version of the boundedness estimate in terms of the constants appearing in the necessary and sufficient conditions is the following
$$
\|S^{\mathcal{D}}\|_{L^2(u)\to L^2(v)} \lesssim ([u,v]_{\mathcal{A}^d_2} + {C}_{u,v})^{1/2}.
$$
There are similar two-weight $L^p$ estimates for continuous square function 
 \cite{LLi,LLi2}, see also \cite[Theorem 6.2]{BeCMP}.

If the weights $(u,v)\in \mathcal{A}_2$ and $u^{-1}\in A_{\infty}$ then they satisfy the necessary and sufficient conditions in Theorem~\ref{thm:NTVdyadicS} and the following estimate holds \cite{BeCMP}, 
$$
\|S^{\mathcal{D}}\|_{L^2(u)\to L^2(v)} \lesssim ([u,v]_{\mathcal{A}_2} +[u,v]_{\mathcal{A}_2 }[u^{-1}]_{A_{\infty}})^{1/2}.
$$

Setting $u=v=w \in A_2$ this improves the known linear bound to a mixed-type bound
 $$
  \|S^{\mathcal{D}}\|_{L^2(w)}\lesssim ([w]_{A_2}[w^{-1}]_{A_{\infty}})^{1/2}\lesssim [w]_{A_2}.
 $$
 
   Same one-weight estimate have been shown to hold for the dyadic square function and  for  matrix valued weights \cite{HPetV}.
Quantitative weighted estimates from $L^p(u)$ into  $L^q(v)$ in terms of \emph{quadratic testing condition} are known \cite{Vu2}.

\subsection{Petermichl's dyadic shift operator}\label{sec:Petermichl'sSha}
Given parameters $(r,\beta) \in \Omega = [1,2)\times \{0,1\}^{\Z}$,
the {\em Petermichl's  dyadic shift operator} $\Sha^{r,\beta}$ (pronounced ``Sha'')
associated to the random 
dyadic grid ${\mathcal{D}^{r,\beta}}$  is defined for functions $f\in L^2(\R)$ by 
$$ \Sha^{r,\beta} f(x) := \sum_{I\in\mathcal{D}^{r,\beta}}  
      \langle f, h_I\rangle\,  H_I(x) = \sum_{I\in\mathcal{D}^{r,\beta}} 2^{-1/2}\sigma(I) \,\langle f,h_{\widetilde{I}}\rangle \, h_I(x),
$$ 
{where}  $  H_I=  2^{-1/2}(h_{I_r}-h_{I_l})$ 
and $\sigma(I)$ is $\pm 1$ depending whether $I$ is the right or left child of  $I$'s parent $\widetilde{I}$. More precisely,
$\sigma (I)=1$ if $I=(\widetilde{I})_r$ and $\sigma (I)=-1$ if $I=(\widetilde{I})_l$.

When $r=1$ and $\beta_j=0$ for all $j\in\Z$ the corresponding grid is the regular dyadic grid and we denote the associated daydic shift operator simply $\Sha$.

 Petermichl's dyadic shift operators  are isometries on $L^2(\R )$,  that is for all $r,\beta\in \Omega$, $\|\Sha^{r,\beta} f\|_{L^2}=\|f\|_{L^2}$,  and they are bounded  operators on $L^p(\R )$, as can be readily seen using Plancherel's identity and dyadic square function estimates.

Each operator  $\Sha^{r,\beta}$ is a good dyadic model for the Hilbert transform $H$. The images  under $\Sha^{r,\beta}$ of the Haar functions are the normalized differences of the Haar functions on its children, namely 
  $\; \Sha^{r,\beta} h_J(x) = H_J(x)$. The functions  $h_J$ and $H_J$  can be viewed  as localized
sines and cosines, in the sense that if we were to extend them periodically, with period the length of the support, we will see two square waves  shifted by half the length of the period. 
 More evidence comes from the way the family $\{\Sha^{r,\beta}\}_{(r,\beta)\in\Omega}$ interacts with translations, 
dilations and reflections.
Each dyadic shift operator does not have symmetries
that characterize the Hilbert transform\footnote{Recall that any bounded linear operator  on $L^2(\R )$  that commutes with dilations  and translations and anticommutes with reflexions must be a constant multiple of the Hilbert transform.}, but an average over all random dyadic 
grids $\mathcal{D}^{r,\beta}$ does. It is a good exercise to figure out how each individual shift interacts with these rigid motions, they almost commute except that the dyadic grid changes. For example regarding reflections,
it can be seen that if we denote by $R(x)=-x$ then
$  R\Sha^{r,\beta} = - \Sha^{r,-\beta} R$, where $-\beta = \{ 1-\beta_i\}_{i\in\Z}$. The corresponding rules for translations and dilations are slightly more complicated, but what matters is that there is a one-to-one correspondence between the dyadic grids so that when averaging over all dyadic grids the average will have the desired properties, and hence it will be a constant multiple of the Hilbert transform. This is precisely what Stefanie Petermichl proved in 2000, a ground breaking and unexpected new result for the Hilbert transform \cite{Pet1}. More precisely she showed that 
$$
H=-\frac{8}{{\pi}} \, \mathbb{E}_{\Omega} \Sha^{{r,\beta}} = -\frac{8}{\pi}\,\int_{\Omega} \Sha^{r,\beta}  d\mathbb{P}(r,\beta ).
$$

 The result  follows after verifying that the averages have the invariance properties that characterize the Hilbert transform  \cite{Pet1, Hyt1}. 
  Because the shift operators
 $\Sha^{r,\beta}$ are uniformly bounded on $L^p(\R )$ for $1<p<\infty$, this representation will immediately imply that the Hilbert transform   $H$ is bounded on $L^p(\R )$ in the same range, a result first proved by Marcel Riesz in 1928.
Similarly once uniform (on the dyadic grids $\mathcal{D}^{r,\beta}$) weighted inequalities are verified for $\Sha^{r,\beta}$ the inequalities will be inherited by the Hilbert transform. Petermichl proved the linear bounds on $L^2(w)$ for the shift operators using Bellman function methods  and hence she proved the $A_2$ conjecture for the Hilbert transform \cite{Pet2}.

These results added a very precise new dyadic  perspective to   such a classic and well-studied operator  as the Hilbert transform.
 Similar representations hold for the  \emph{Beurling} 
and the \emph{Riesz}  transforms \cite{PetV,Pet3}, these operators have many invariance properties as the Hilbert transform does. For a while it was believed that such invariances were responsible for these representation formulas.  It came as a surprise when  Tuomas Hyt\"onen proved in 2012 that 
there is  a representation formula valid for {\sc all} Calder\'on-Zygmund singular integral operators \cite{Hyt2}.
To state Hyt\"onen's result we need to introduce Haar shift operators of arbitrary complexity and paraproducts.

\subsection{Haar shift operators of arbitrary complexity}\label{sec:Haar-shift-complexity}

The  \emph{Haar shift operators of complexity $(m,n)$} associated to a dyadic grid $\mathcal{D}$ were introduced by Michael Lacey,  Stefanie Petermichl, and Mari Carmen Reguera \cite{LPetR}, they are defined on $L^2(\R )$ as follows
$$
 \Sha_{m,n}f(x) := \sum_{L\in{\mathcal D}} \sum_{I\in{\mathcal D}_m(L), 
J\in {\mathcal D}_n(L)} c_{I,J}^L \, \langle f, h_I\rangle \, h_J(x),
$$
 where the coefficients $|c_{I,J}^L  | \leq \frac{\sqrt{|I|\,|J|}}{|L|}$,
 and ${\mathcal D}_m(L)$ denotes the dyadic subintervals of $L$
 with length $2^{-m}|L|$.

 The cancellation property of the Haar functions and the
 normalization of the coefficients ensures that 
$\|\Sha_{m,n}f\|_{L^2}\leq \|f\|_{L^2}$, square function estimates ensure boundedness on $L^p(\R )$ for all
$1<p<\infty$. The martingale transform, $T_{\sigma}$, is a Haar shift operator of complexity $(0,0)$.
Petermichl's  $\Sha^{r,\beta}$ operators    are Haar shift operators of complexity $(0,1)$.
 The dyadic paraproduct, $\pi_b$,  to be introduced in Section~\ref{sec:paraproduct},  is not one of these and nor is its adjoint $\pi_b^*$.

The following estimates are known for dyadic shift operators of arbitrary complexity.
First, Michael Lacey, Stefanie Petermichl,  and Mari Carmen Reguera proved the $A_2$ 
conjecture for the Haar shift operators of arbitrary complexity with constant depending exponentially
on the complexity \cite{LPetR}. Unlike their predecessors, they 
did not use Bellman functions, instead they used 
stopping time techniques and a 
two-weight theorem for  "well localized operators"
 of \cite{NTV4}.
Second,  David Cruz-Uribe, Chema Martell, and  Carlos P\'erez   \cite{CrMPz2} used a local 
median oscillation technique introduced by Andrei Lerner \cite{Le1, Le2}.
The local median oscillation method was quite  flexible, they  obtained new results such as 
the sharp bounds for the square function for $p>2$, 
 for the dyadic paraproduct,  also for vector-valued maximal
operators, as well as two-weight results, however  for the dyadic shift operators the weighted estimates
 still depended exponentially  on the complexity. 
 Third, Tuomas Hyt\"onen \cite{Hyt2} obtained the linear estimates with polynomial dependence on the complexity, needed to prove  the $A_2$ conjecture for Calder\'on-Zygmund  singular integral operators. 

\subsection{Dyadic paraproduct}\label{sec:paraproduct}

Quoting from  an article on "What is....a Paraproduct?" for the broader public by Arpad B\'enyi, Diego Maldonado, and Virginia Naibo \cite{BMN}:
\begin{quote}{``The term paraproduct is nowadays used rather loosely in the literature to indicate a bilinear operator that, although noncommutative, is somehow better behaved than the usual product of functions. Paraproducts emerged in J.-M. Bony's theory of paradifferential operators [Bo], which stands as a milestone on the road beyond pseudodifferential operators pioneered by R. R. Coifman and Y. Meyer in [CM]. 
 Incidentally, the Greek word $\pi\alpha\rho\alpha$  (para) translates as \emph{beyond} in English 
and \emph{au délà de} in French, just as in the title of [CM]. 
The defining properties of a paraproduct should therefore go beyond the desirable properties of the product.''}
\end{quote}

 The \emph{dyadic paraproduct} associated to a dyadic grid $\mathcal{D}$ and to $b\in {\rm BMO}^{\mathcal{D}}$ is an operator acting on square integrable functions $f$ as follows
\vskip -.1in
   $$
 \pi_b f(x) := \sum_{I\in{\mathcal D}}  \langle f\rangle_I \, \langle b,h_I\rangle \, h_I(x),
$$
where $ \langle f\rangle_I=\frac{1}{|I|}\int_If(x)\,dx = \langle f,\mathbbm{1}_I/|I|\rangle$. A function $b$ is in the \emph{space of dyadic bounded mean oscillation}, ${\rm BMO}^{\mathcal{D}}$ if and only if 
$$ \| b\|_{{\rm BMO}^{\mathcal{D}}}:= \sup_{J\in\mathcal{D}}\bigg (\frac{1}{|J|} \int_J |b(x)-\langle b\rangle_J|^2dx \bigg )^{1/2}<\infty.$$
Notice that we are using an $L^2$ mean oscillation instead of the $L^1$ mean oscillation used in \eqref{BMO} in the definition of ${\rm BMO}$, and of course, we are restricting to dyadic intervals. As it turns out, one could use an $L^p$  mean oscillation for any $1\leq p$ and obtain equivalent norms in ${\rm BMO}$ thanks to the celebrated John-Nirenberg Lemma \cite{JN}.

 Formally, expanding $f$  and $b$ in the Haar basis, multiplying and separating the terms into upper triangular, diagonal, and lower triangular parts, one gets that  
 $$ bf = \pi_bf + \pi^*_bf + {\pi_fb},$$ 
 in doing so is important to note that $\sum_{I\in\mathcal{D}: I\supset J} \langle f, h_I\rangle \, h_I = \langle f\rangle_J$.
 It is well known that multiplication by a function $b$ is a bounded operator on $L^p(\R )$ if and only if the function is essentially bounded, that is, $b\in L^{\infty}(\R )$. However the
 paraproduct is a  bounded operator on $L^p(\R )$ if and only if $b\in {\rm BMO}^{\mathcal{D}}$, which is a space strictly larger than $L^{\infty}(\R )$. The  $L^2$ estimate can be obtained using,  for example,  the {Carleson embedding  lemma}, see Section~\ref{sec:Carleson-sequences}.

 Using a weighted Carleson embedding lemma, one can  check that the paraproduct  is bounded on $L^2(w)$ for all  $w\in A_2$ \cite{P1}.  Furthermore, Beznosova proved the $A_2$ conjecture for paraproducts \cite{Be}, namely
  $$ 
    \|\pi_bf\|_{L^2(w)}\leq C{[w]_{A_2}\|b\|_{{\rm BMO}^{\mathcal{D}}}}\|f\|_{L^2(w)}.     
    $$
By extrapolation one concludes that the paraproduct is bounded on $L^p(w)$ for all $w\in A_p$  and $1<p<\infty$, in particular it is bounded on $L^p(\R )$. In  Section~\ref{sec:A2paraproduct}, we will present  Beznosova's Bellman function argument proving the $A_2$ conjecture for the dyadic paraproduct. This argument was generalized to $\R^d$ in \cite{Ch3} and to spaces of homogeneous type in \cite{We}. It was pointed out to us recently \cite{Wic} that the paraproduct is a well-localized operator (for trivial reasons) in the sense of \cite{NTV4} and therefore it falls under their theory.

To finish this brief introduction to the paraproduct, we would like to mention its intimate connection to the $T(1)$ and $T(b)$ Theorems of Guy David, Jean-Lin Journ\'e, and Stephen Semmes \cite{DaS, DaJS}. These theorems give (necessary and sufficient) conditions to verify boundedness on $L^2(\R )$ for singular integral  operators  $T$ with a Calder\'on-Zygmund  kernel when Fourier analysis, almost-orthogonality (Cotlar's lemma), or other more standard techniques fail. In the $T(1)$ theorem the conditions amount to checking some weak boundedness property which is a necessary condition, and checking that the function~$1$ is "mapped" under the operator and its adjoint, $T(1)$ and $T^*(1)$,  into ${\rm BMO}$. Once this is verified the operator can be decomposed into a "simpler"  operator $S$ with the property that $S(1)=S^*(1)=0$,  a paraproduct, $\pi_{T(1)}$, and the adjoint of a paraproduct, $\pi_{T^*(1)}^*$. The paraproduct terms are bounded on $L^2(\R)$, the operator $S$ can be verified to be bounded on $L^2(\R )$, and as a consequence so will be the operator $T$. %

We have defined all these model operators in the one-dimensional case, there are corresponding Haar shift operators and dyadic paraproducts  defined  on $\R^d$ as well as $T(1)$ and $T(b)$ theorems.

\subsection{Hyt\"onen's representation theorem}\label{sec:Hytonen-representation-theorem}
Let us remind the reader that a bounded operator on $L^2(\R^d)$ is a  Calder\'on-Zygmund singular integral operator with  smoothness parameter  $\alpha>0$ if it has  an integral representation
$$ Tf(x)= \int_{\R^d} K(x,y)f(y)\,dy, \quad\quad x\notin \mbox{supp}\,f,$$
for a kernel $K(x,y)$ defined for all  $(x,y) \in \R^d\times \R^d$ such that $x\neq y$, and verifying the standard size and smoothness estimates, respectively $\displaystyle{|K(x,y)|\leq {C}/{|x-y|^d}}$ and
$$|K(x+h,y)-K(x,y)| +|K(x,y+h)-K(x,y)|\leq {C|h|^{\alpha}}/{|x-y|^{d+\alpha}},$$
for all $|x-y|>2|h|>0$ and some fixed $\alpha\in[0,1]$.

It is worth remembering that such Calder\'on-Zygmund  singular integral operators  are bounded on $L^p(\R^d)$ for all $1<p<\infty$, they are of weak-type $(1,1)$, and they map ${\rm BMO}$ into itself.

We now have all the ingredients to state the celebrated Hyt\"onen's representation theorem \cite{Hyt2} at least in the one-dimensional case.

 \begin{theorem}[Hyt\"onen's 2012]
 Let $T$ be a Calder\'on-Zygmund singular integral operator with smoothness parameter $\alpha>0$, then 
 $$ T(f) = \mathbb{E}_{\Omega} \left (\sum_{(m,n)\in\N^2} e^{-(m+n)\alpha/2} \Sha_{m,n}^{r,\beta}(f) +
\pi^{r,\beta}_{T(1)}(f) + (\pi^{r,\beta}_{T^*(1)})^*(f) \right ).
$$
  \end{theorem}
  Where  for each pair of random parameters $(r,\beta)\in \Omega$, the operator
 $\Sha_{m,n}^{r,\beta}$ is a Haar shift operators of complexity $(m,n)$,  the operator $\pi^{r,\beta}_{T(1)}$ is a dyadic paraproduct,
 and the operator $(\pi^{r,\beta}_{T^*(1)})^*$ is  the adjoint of a dyadic paraproduct,
 all defined on the {random dyadic grid $\mathcal{D}^{r,\beta}$}.  The paraproducts and their adjoints in the decomposition depend on 
 the operator $T$ via $T(1)$ and $T^*(1)$. The Haar shift operators in the decomposition also depend on $T$ although it is not obvious in the notation we used. Indeed, the coefficients $c^L_{I,J}$, in the definition of the Haar shift multiplier of complexity $(m,n)$ (see Section~\ref{sec:Haar-shift-complexity}), will depend on the given operator $T$ for each $L\in \mathcal{D}^{r,\beta}$, $I\in \mathcal{D}^{r,\beta}_m(L)$, and $J\in \mathcal{D}^{r,\beta}_n(L)$. 
 Notice that the exponential nature of the coefficients in the expansion explains why the Haar shift multipliers of arbitrary complexity will need to be bounded with a bound depending at most polynomially on the complexity.
 
  To the author, this is a remarkable result providing a dyadic decomposition theorem for a large class of operators. 
 Once you have such decomposition and $L^2(w)$ estimates for each of the components (Haar shift operators, paraproducts and their adjoints) that are linear on $[w]_{A_2}$ and that are uniform on  the dyadic grids then the $A_2$ conjecture is resolved in the positive, as Tuomas Hyt\"onen did in his celebrated paper \cite{Hyt2}.

\section{$A_2$ theorem for the dyadic paraproduct: A Bellman function proof}\label{sec:A2paraproduct}

As a model example we will present in this section Beznosova's argument proving  the $A_2$~conjecture for the dyadic paraproduct \cite{Be}. The goal is  to show  that  for all weights  $w\in A_2$, functions $b\in {\rm BMO}^{\mathcal{D}}$, and functions $f\in L^2(w)$   the following estimate holds,
$$ \|\pi_bf\|_{L^2(w)}\lesssim [w]_{A_2}\|b\|_{BMO^{\mathcal{D}}}\|f\|_{L^2(w)}.$$
We remind the reader that the \emph{dyadic paraproduct} associated to $b\in {\rm BMO}^{\mathcal{D}}$ is defined by
$
 \pi_b f(x) := \sum_{I\in{\mathcal D}}   \langle f\rangle_I  \, b_I\,h_I(x),
$
{where}  $b_I=\langle b,h_I\rangle$ and $\langle f\rangle_I=(1/|I|)\int_If(y)\,dy$.

To achieve a preliminary estimate, where instead of the linear bound on $[w]_{A_2}$ we get a 3/2 bound, namely $[w]_{A_2}^{3/2}$, we need to introduce a few ingredients: (weighted) Carleson sequences, Beznosova's Little Lemma,  and the Weighted Carleson Lemma. Afterwards, in Section~\ref{sec:A2-conjecture-paraproduct}, we will refine the argument to get the desired linear bound. To achieve the linear bound,  we will need a few additional  ingredients, including  the $\alpha$-Lemma,   introduced by Oleksandra Beznosova in her original proof \cite{Be1}. Both the Little Lemma and the $\alpha$-Lemma are proved using Bellman functions and we sketch their proofs, as well as the proof of the Weighted Carleson Lemma.

\subsection{Weighted Carleson sequences, Weighted Carleson Lemma, and Little Lemma}\label{sec:Carleson-sequences}
In this section we introduce  weighted an unweighted Carleson sequences and the weighted Carleson embedding lemma. We also present Bezonosova's Little Lemma that enables us to ensure that given a weight $w$ and a Carleson sequence $\{\lambda_I\}_{I\in\mathcal{D}}$ we can create a $w$-weighted Carleson sequence by multiplying each term of the given sequence by the reciprocal of $w^{-1}(I)$.

\subsubsection{Weighted Carleson sequences and lemma}
 Given a weight $w$, a  positive sequence $\{\lambda_I\}_{I\in\mathcal{D}}$ is  \emph{$w$-Carleson} if
there is a constant  $A>0$ such that
$$
\sum_{I\in\mathcal{D}(J)}\lambda_I\leq Aw(J)\quad \mbox{for all} \; J\in\mathcal{D},
$$
where  $w(J)=\int_Jw(x)\,dx$. The smallest  constant $A>0$ is called the \emph{intensity} of the sequence.}
When $w=1$ a.e.  we say that the sequence is  Carleson (not 1-Carleson).

\begin{example} \label{eg:b-BMO-Carleson}
If $b\in {\rm BMO}^{\mathcal{D}}$ then the sequence $\{b_I^2\}_{I\in \mathcal{D}}$ is  Carleson with intensity $\|b\|_{{\rm BMO}^{\mathcal{D}}}^2$. Indeed,   for any $J\in\mathcal{D}$, the collection of Haar functions corresponding to dyadic intervals $I\subset J$, $\{h_I\}_{I\in\mathcal{D}(J)}$, forms an orthonormal basis on $L^2_0(J)=\{f\in L^2(J): \int_J f(x)\,dx=0\}$.  The function $(b-\langle b\rangle_J)\big |_{J}$ belongs to $L^2_0(J)$, therefore by Plancherel's inequality,
$$  \sum_{I\in\mathcal{D}(J)} b_I^2= \sum_{I\in\mathcal{D}(J)} |\langle b, h_I\rangle|^2 = \int_J|b(x)-\langle b\rangle_J|^2\, dx  \leq \|b\|_{{\rm BMO}^{\mathcal{D}}}^2 |J|.$$
\end{example}

The following weighted Carleson lemma  that appeared in \cite{NTV1} will be extremely useful in our estimates, you can find a proof in \cite{MP} that we reproduce in Section~\ref{sec:W-C-L}.

\begin{lemma}[Weighted Carleson Lemma] \label{lem:weighted-Carleson} Given a weight $v$,  then  $\{\lambda_I\}_{I\in\mathcal{D}}$ is a $v$-Carleson sequence with intensity $A$ if and only if  for all non-negative $F\in L^1(v)$ we have
$$ \sum_{I\in\mathcal{D}} \lambda_I \inf_{x\in I} F(x)\leq A\int_{\R}F(x) \,v(x) \, dx.$$
\end{lemma}

The following particular instance of the Weighted Carleson Lemma will be useful. 
\begin{example} Let $\{\lambda_I\}_{I\in\mathcal{D}}$ be a $v$-Carleson sequence with intensity $A$, let  $f\in L^2(v)$ and set  $F(x)=(M^{\mathcal{D}}_vf(x))^2$ where $M^{\mathcal{D}}_v$ is the weighted dyadic maximal function, namely 
$$M^{\mathcal{D}}_vf(x) :=\hskip -.1in  \sup_{I\in\mathcal{D}:x\in I} \langle |f|\rangle_I^v \quad \mbox{where} \;\;
\langle |f|\rangle^v_I  :={\langle |f|v\rangle_I}/{\langle v \rangle_I}.$$ 
By definition of the dyadic maximal function,  $\langle |f|\rangle^v_I \leq \inf_{x\in I} M^{\mathcal{D}}_vf(x)$.
Then by the Weighted Carleson Lemma  {\rm (Lemma~\ref{lem:weighted-Carleson})} and the boundedness of $M_v^{\mathcal{D}}$ on $L^2(v)$ with operator bound  independent of the weight we conclude that
$$ \sum_{I\in\mathcal{D}} \lambda_I \,(\langle |f| \rangle_I^v  )^2\leq A \| M^{\mathcal{D}}_vf\|^2_{L^2(v)}\lesssim A \|f\|^2_{L^2(v)}.$$
\end{example}

Specializing even further we get another useful result, that establishes  the boundedness of the dyadic paraproduct on $L^2(\R )$ when $b\in {\rm BMO}^{\mathcal{D}}$.
\begin{example}\label{BMO-CarlesonSeq}
In particular, if $v\equiv 1$ and $b\in {\rm BMO}^{\mathcal{D}}$, then  $\lambda_I:=b_I^2$ for $I\in\mathcal{D}$ defines a  Carleson sequence with intensity $\|b\|_{{\rm BMO}^{\mathcal{D}}}^2$, hence
$${ \|\pi_bf\|_{L^2}^2=\sum_{I\in\mathcal{D}}|\langle \pi_bf,h_I\rangle |^2   \leq \sum_{I\in\mathcal{D}} b_I^2\, \langle |f|\rangle_I^2\lesssim \|b\|_{{\rm BMO}^{\mathcal{D}}}^2 \|f\|^2_{L^2}}.$$
\end{example}

\subsubsection{Beznosova's Little Lemma}

We will need to create $w$-Carleson sequences from  given Carleson sequences. The following lemma  will come in handy \cite{Be}.
 
\begin{lemma}[Little Lemma]\label{thm:Little-Lemma}
Let $w$ be a weight, such that $w^{-1}$ is also  a weight.
Let $\{ \lambda_I \}_{I \in \mathcal{D}}$ be a Carleson sequence 
with intensity $A$,  the sequence $\{{\lambda_I}/{\langle w^{-1}\rangle_I}\}_{I\in\mathcal{D}}$ is  $w$-Carleson  with intensity $4A$. In other words,
for all $J \in \mathcal{D}$
\begin{equation}\label{Little-Lemma}
\sum_{I \in \mathcal{D}(J)}
\frac{\lambda_I}{\langle w^{-1}\rangle_I}\leq 4A \; w(J).
\end{equation}

\end{lemma}

The proof uses a Bellman function argument that we will present in Section~\ref{sec:loose-ends}. Note that the weight $w$ in the Little Lemma is not required to be in the  Muckenhoupt $A_2$ class. It does require that the reciprocal $w^{-1}$ is a weight, of course if $w\in A_2$ then $w^{-1}$ is a weight  in $A_2$.

\begin{example}\label{eg:little-lemma}
Let $b\in {\rm BMO}^{\mathcal{D}}$ and  $w\in A_2$.The sequence  $\{b_I^2/\langle w\rangle_I \}_{I\in\mathcal{D}}$ is
a $w^{-1}$-Carleson, with intensity $4\|b\|_{{\rm BMO}^{\mathcal{D}}}^2$. By {\rm Example~\ref{eg:b-BMO-Carleson}}
the sequence $\{b_I^2\}_{I\in\mathcal{D}}$ is a Carleson sequence with intensity $\|b\|_{{\rm BMO}^{\mathcal{D}}}^2$, and then applying {\rm Lemma~\ref{thm:Little-Lemma}} with the roles of $w$ and $w^{-1}$ interchanged we get the stated result.
\end{example}

\subsection{The $3/2$ bound for the paraproduct on weighted $L^2$}
We now show that the paraproduct $\pi_b$   is bounded on $L^2(w)$ when $w\in A_2$ and $b\in {\rm BMO^{\mathcal{D}}}$, with bound $[w]_{A_2}^{3/2}\|b\|_{{\rm BMO}^{\mathcal{D}}}$, not yet the optimal linear bound.
\begin{proof} By duality suffices to show that for all  $f\in L^2(w)$ and $g\in L^2(w^{-1})$ 
$${ |\langle \pi_b f, g\rangle | \lesssim [w]_{A_2}^{3/2} \|b\|_{{\rm BMO}^{\mathcal{D}}}
\|f\|_{L^2(w)}\|g\|_{L^2(w^{-1})}.}$$
By definition of the dyadic paraproduct and the triangle inequality,
$${ |\langle \pi_b f, g\rangle | \leq \sum_{I\in\mathcal{D}} \langle |f| \rangle_I  \,|b_I| \,|\langle g, h_I\rangle|}.$$
{First, using the  Cauchy-Schwarz inequality,  we can estimate as follows,
$$|\langle \pi_b f, g\rangle | \leq  \left (\sum_{I\in\mathcal{D}} \frac{\langle |f|\rangle^2_I \,b^2_I }{\langle w^{-1}\rangle_I}\right )^{1/2}
 \left (\sum_{I\in\mathcal{D}}|\langle g, h_I\rangle |^2\langle w^{-1}\rangle_I\right )^{1/2}.$$
Second, using  the fact that $\|S^{\mathcal{D}}g\|_{L^2(w^{-1})}^2=\sum_{I\in\mathcal{D}}|\langle g, h_I\rangle |^2\langle w^{-1}\rangle_I$  and the linear bound on $L^2(v)$ for the square function for $v=w^{-1}\in A_2$ with $[w^{-1}]_{A_2}=[w]_{A_2}$, we further estimate by,
 \begin{align*} |\langle \pi_b f, g\rangle | & \leq  \left (\sum_{I\in\mathcal{D}} \left  (\frac{\langle |f|ww^{-1}\rangle_I}{\langle w^{-1}\rangle_I}\right )^2\frac{b^2_I }{\langle w\rangle_I}\, {\langle w\rangle_I\,\langle w^{-1}\rangle_I}\right )^{1/2} \|S^{\mathcal{D}}g\|_{L^2(w^{-1})} \\
 &\lesssim  {[w]_{A_2}^{1/2}}\left (\sum_{I\in\mathcal{D}} \big (\langle |f|w\rangle^{w^{-1}}_I\big )^2{\frac{b^2_I }{\langle w\rangle_I}}\right )^{1/2}  { [w]_{A_2}}\|g\|_{L^2(w^{-1})}.
 \end{align*}
  Third, using the Weighted Carleson Lemma (Lemma~\ref{lem:weighted-Carleson}) for the $w^{-1}$-Carleson sequence $\{b_I^2/\langle w\rangle_I \}_{I\in\mathcal{D}}$ with intensity $4\|b\|_{{\rm BMO}^{\mathcal{D}}}^2$ (see Example~\ref{eg:little-lemma}), together with the fact that $\|f\|_{L^2(w)}=\|fw\|_{L^2(w^{-1})}$, we get that,}
\begin{eqnarray*}
 |\langle \pi_b f, g\rangle |      
     &\lesssim & {[w]_{A_2}^{3/2}}\, {2 \|b\|_{{\rm BMO}^{\mathcal{D}}}\| M^{\mathcal{D}}_{w^{-1}}(fw)\|_{L^2(w^{-1})} }\|g\|_{L^2(w^{-1})}\\ 
     &\lesssim & {  [w]_{A_2}^{3/2}  \|b\|_{{\rm BMO}^{\mathcal{D}}}} \|f\|_{L^2(w)}\|g\|_{L^2(w^{-1})},
    \end{eqnarray*}
where in the last line we used the boundedness of the dyadic weighted maximal function $M^{\mathcal{D}}_{v}$ on $L^2(v)$ with an operator norm independent of the weight $v$. This implies that
\[ \|\pi_b f\|_{L^2(w)} \lesssim [w]_{A_2}^{3/2}\|b\|_{{\rm BMO}^{\mathcal{D}}}\|f\|_{L^2(w)}.\]
This is precisely what we set out to prove. 
\end{proof}

\subsection{Algebra of Carleson sequences, $\alpha$-Lemma, and weighted Haar bases}
To get a linear bound instead of the 3/2 power bound we just obtained, we will need a couple more ingredients, some algebra with Carleson sequences, the $\alpha$-Lemma, and weighted Haar bases.

\subsubsection{Algebra of Carleson sequences} Given weighted Carleson sequences we can create new weighted Carleson sequences by linear operations or by taking geometric means.
\begin{lemma}[Algebra of Carleson sequences]\label{lem:algebra-of-Carleson-seq}
Given  a weight $v$,
let $\{\lambda_I\}_{I\in\mathcal{D}}$ and $\{\gamma_I\}_{I\in\mathcal{D}}$ be two $v$-Carleson sequences with
intensities $A$ and $B$ respectively then for any $c, d >0$ 
\begin{itemize}
\item[(i)] The sequence $ \{c \lambda_I + d\gamma_I\}_{I\in\mathcal{D}}$ is a $v$-Carleson sequence with
intensity at most $cA + dB$.
\item[(ii)] The sequence $ \{\sqrt{\lambda_I \gamma_I}\}_{I\in\mathcal{D}}$ is a $v$-Carleson sequence
with intensity at most $\sqrt{AB}$.
\end{itemize}
\end{lemma}

The proof is a simple exercise which we leave to the interested reader.  We do need some specific Carleson sequences, and we record them in the next example.

\begin{example}\label{examples-Carleson-sequences}
Let $u,v \in A_{\infty}$ and  $\Delta_I v:= \langle v\rangle_{I_+} - \langle v\rangle_{I_-}$.
Then 
\begin{itemize}
\item[(i)] The sequence $\big \{ \left |{\Delta_I v|}/{\langle v\rangle_I}\right |^2|I| \big \}_{I\in\mathcal{D}}$ 
 is a Carleson sequence, with intensity  {$C\log [w]_{A_{\infty}}$}.  
\item[(ii)] Let $\alpha_I=({|\Delta_I v|}/{\langle v\rangle_I})({|\Delta_I u|}/{\langle u\rangle_I})|I|$. The sequence $\{\alpha_I\}_{I\in\mathcal{D}}$ is a Carleson sequence. 
\item[(iii)] When $v\in A_2$, $u=v^{-1}$ $($also in $A_2)$ the sequence $\{\alpha_I\}_{I\in\mathcal{D}}$ defined in item {\rm (ii)} has  intensity at most $ {\log [v]_{A_2}}$.
\end{itemize}
\end{example}
Example~\ref{examples-Carleson-sequences}(i) was discovered by Robert Fefferman, Carlos Kenig\footnote{Carlos Kenig, an Argentinian mathematician,  was elected President of the International Mathematical Union in July 2018 in the International Congress of Mathematicians (ICM) held in Brazil and for the first time in the Southern hemisphere.}, and Jill Pipher\footnote{Jill Pipher is the president-elect of the American Mathematical Society (AMS), and will begin a two-year term in 2019.}.  in 1991, see \cite{FKP}.  The sharp constant $C=8$ was obtained by Vasily Vasyunin using the Bellman function method \cite{Va}. In fact this example provides a characterization of $A_{\infty}$ by summation conditions, for many more such characterizations for other weight classes see \cite{Bu2,BeRe}.
Example~\ref{examples-Carleson-sequences}(ii)-(iii) follow from Example~\ref{examples-Carleson-sequences}(i) and from Lemma~\ref{lem:algebra-of-Carleson-seq}(ii).

\subsubsection{The $\alpha$-Lemma}
The key to dropping  from power 3/2 to linear power in the weighted $L^2$ estimate for the paraproduct is
the following lemma, discovered by Beznosova, like the Little Lemma, in the course of writing her PhD Dissertation \cite{Be1}, see also \cite{Be,MP}. Both lemmas were proved using  Bellman functions and  we will sketch the arguments in Section~\ref{sec:loose-ends}.
\begin{lemma}[$\alpha$-Lemma] \label{lem:alpha-lemma}
If $w\in A_2$  and $\alpha >0 $, then the sequence
\begin{equation*}
  \mu_I:= \langle w\rangle_I^{\alpha}\langle \,w^{-1}\rangle_I^{\alpha} \, |I| \bigg( \frac{|\Delta_Iw|^2}{\langle w\rangle_I^2} +\frac{|\Delta_Iw^{-1}|^2}{\langle w^{-1}\rangle_I^2}  \bigg) \quad I  \in \mathcal{D}
  \end{equation*}
is a Carleson sequence with  intensity at most {$C_{\alpha}[w]_{A_2}^{\alpha}$}, and $C_{\alpha} =\max\{72/(\alpha-2\alpha^2), 576\}$. 
\end{lemma}
Notice that  the algebra of Carleson sequences encoded in Lemma~\ref{lem:algebra-of-Carleson-seq} together with  the R.~Fefferman-Kenig-Pipher Example~\ref{examples-Carleson-sequences}(iii) give, for $\mu_I$, an intensity of $[w]_{A_2}^{\alpha}\log [w]_{A_2}$, which is larger by a logarithmic factor than the one claimed in the $\alpha$-Lemma. This lesser estimate will improve the 3/2 estimate to a linear times logarithmic estimate \cite{Be1}, the stronger $\alpha$-Lemma will yield the desired linear estimate.
\begin{example}\label{eg:nuI-Carleson}
Let $w\in A_2$ and $b\in {\rm BMO}^{\mathcal{D}}$.
By the $\alpha$-Lemma and  the algebra of Carleson sequences   we conclude that
\begin{itemize}
\item[(i)]  $\{\nu_I := |\Delta_I w|^2\langle w^{-1}\rangle_I^2 |I| \}_{I\in\mathcal{D}}$ is  Carleson  with intensity $C_{1/4}[w]^2_{A_2}$, $C_{1/4}=576$.
\item[(ii)] $\{b_I\sqrt {\nu_I} \}_{I\in\mathcal{D}}$  is Carleson with intensity $24[w]_{A_2}\|b\|_{{\rm BMO}^{\mathcal{D}}}$.
\end{itemize}
\end{example}

\subsubsection{Weighted Haar basis}

The last ingredient before we present the proof of the $A_2$ conjecture for the dyadic paraproduct
is the weighted Haar basis.

Given a doubling weight $w$  and  an interval $I$,  the  \emph{weighted Haar 
function} $h_I^w$ is given by
$$  h^w_I(x):= \sqrt{w(I_-)}/\sqrt{w(I)w(I_+)}\,\mathbbm{1}_{I_+}(x)
             -\sqrt{w(I_+)}/\sqrt{w(I)w(I_-)}\,\mathbbm{1}_{I_-}(x).
$$
The collection $\{h_I^w\}_{I\in \mathcal{D}}$, of weighted Haar functions indexed on $\mathcal{D}$  --a system of dyadic intervals--,  is an orthonormal system of $L^2(w)$. In fact, the weighted Haar functions are the Haar functions corresponding to the space of homogeneous type $X=\R$ with the Euclidean metric,  the doubling measure $d\mu=w\,dx$, and  the dyadic structure $\mathcal{D}$, defined in Section~\ref{sec:Haar-SHT}.

There is a very simple formula relating the weighted Haar function and the regular Haar function. More precisely, 
given $I\in\mathcal{D}$ there exist numbers $\alpha_I^w$,
$\beta^w_I$ such that
$$ h_I(x) = \alpha^w_I h^w_I(x) + \beta_I^w {\mathbbm{1}_I(x)}/{\sqrt{|I|}}.
$$
The coefficients can be calculated precisely,  and they have the following upper bounds:
   \begin{itemize}
       \item[(i)] $\;\;|\alpha^w_I | \leq \sqrt{\langle w\rangle_I}$,
       $\quad\quad\quad$ (ii) $\;\;|\beta^w_I| \leq {|\Delta_I w|}/{\langle w\rangle_I}$ where $\;\Delta_I w:= \langle w\rangle_{I_+} -\langle w\rangle_{I_-}$.
  \end{itemize}

\subsection{$A_2$ conjecture for the dyadic paraproduct}\label{sec:A2-conjecture-paraproduct}
We present a proof of Beznosova's theorem,  namely, for all  $b\in {\rm BMO}^{\mathcal{D}}$, $w\in A_2$, and $f\in L^2(w)$
$$ \| \pi_bf\|_{L^2(w)}\lesssim \|b\|_{{\rm BMO}^{\mathcal{D}}} [w]_{A_2}\|f\|_{L^2(w)}.$$

The proof uses the same ingredients introduced by Oleksandra Beznosova \cite{Be}, and 
 a beautiful argument by Fedja Nazarov, Sasha Reznikov,  and Sasha Volberg that yields polynomial in the complexity bounds for
Haar shift operators on geometric doubling metric spaces \cite{NRV}. An extension of  their result 
to \emph{paraproducts with  arbitrary complexity} can be found in joint work with Jean Moraes \cite{MP}.

\begin{proof}
Suffices by duality to prove that
$$ |\langle \pi_b f, g\rangle| \leq C \|b\|_{{\rm BMO}^{\mathcal{D}}} [w]_{A_2} \|f\|_{L^2(w)}\|g\|_{L^2(w^{-1})}.
$$

We  introduce weighted Haar functions to obtain two terms to be estimated separately,
$$ | \langle \pi_b f, g \rangle | \leq  
\sum_{I \in \mathcal{D}} |b_I| \, \langle |f|w w^{-1}\rangle_I \, | \langle gw^{-1}w, h_I \rangle | \leq \Sigma_1 + \Sigma_2.
$$
Explicitly,  the sums $\Sigma_1$ and $\Sigma_2$ are obtained replacing $h_I = \alpha^{w}_I h^{w}_I + \beta^{w}_I {\mathbbm{1}_I}/{\sqrt{|I|}}$, and using the estimates on the coefficients $\alpha_I$, $\beta_I$, to  get
\begin{eqnarray*}
\Sigma_1 & :=  & \sum_{I \in \mathcal{D}} |b_I| \,\langle |f|ww^{-1}\rangle_I \,| \langle gw^{-1}w,h_I^{w}  \rangle | \,\sqrt{\langle w\rangle_I},\\
\Sigma_2 & :=  & \sum_{I \in \mathcal{D}} |b_I| \,\langle |f|ww^{-1}\rangle_I \, \langle |g| w^{-1}w\rangle_I \, \frac{|\Delta_Iw|} {\langle w\rangle_I}\,{\sqrt{|I|}}.
\end{eqnarray*}

\subsubsection*{First Sum $\Sigma_1$}  Denote the $L^2(w)$ pairing  $\langle h, k\rangle_{L^2(w)}:=\langle hw,k\rangle$.
To estimate the first sum we observe that the weighted average (with respect to the weight $w^{-1}$) of the function $|f|w$ over a dyadic interval  is bounded by the corresponding dyadic weighted maximal function evaluated at any point on the interval, hence by the infimum over the interval, more precisely, 
$ \langle |f|ww^{-1}\rangle_I\ /\langle w^{-1}\rangle_I \leq \inf_{x \in I} M^{\mathcal{D}}_{w^{-1}}(fw) (x)  $. Then  using the definition of $A_2$ and the Cauchy-Schwarz inequality,   we get
\begin{align*}
\Sigma_1 & =   \sum_{I \in \mathcal{D}} \frac{|b_I|}{\sqrt{\langle w\rangle_I }} \frac{\langle |f|ww^{-1}\rangle_I}{\langle w^{-1}\rangle_I} | \langle gw^{-1},h_I^{w}  \rangle_{L^2(w)} |\, \langle w\rangle_I \langle w^{-1}\rangle_I\\
&  \leq  [w]_{A_2} \sum_{I \in \mathcal{D}} \frac{|b_I|}{\sqrt{\langle w\rangle_I}}\inf_{x \in I} M^{\mathcal{D}}_{w^{-1}}(fw) (x) \, | \langle gw^{-1},h_I^{w}  \rangle_{L^2(w)} | \\
    & \leq  [w]_{A_2} \Bigg (\sum_{I \in \mathcal{D}} \frac{|b_I|^2}{\langle w\rangle_I} \inf_{x \in I} |M^{\mathcal{D}}_{w^{-1}}(fw)(x)|^2\,  \Bigg)^{\frac{1}{2}}\Bigg
(\sum_{I \in \mathcal{D}} \big |\langle gw^{-1}, h_I^{w} \rangle_{L^2(w)}\big |^2 \Bigg)^{\frac{1}{2}}.
\end{align*}
Using the {Weighted Carleson Lemma}  (Lemma~\ref{lem:weighted-Carleson}) with $F(x)=|M^{\mathcal{D}}_{w^{-1}} (fw)(x)|^2$, with  weight $v=w^{-1}\in A_2$ recalling  $[w]_{A_2}=[w^{-1}]_{A_2}$, and with  $w^{-1}$-Carleson sequence $\{b_I^2/ \langle w\rangle_I\}_{I\in\mathcal{D}}$ with intensity $4\|b\|_{{\rm BMO}^{\mathcal{D}}}^2$ by the {Little Lemma} (Lemma~\ref{Little-Lemma}), we get that,
\begin{align*}
\Sigma_1&\leq 2 [w]_{A_2} \|b\|_{{\rm BMO}^{\mathcal{D}}} \left ( \int_{\mathbb{R}} |M^{\mathcal{D}}_{w^{-1}}(fw)(x)|^2 \, w^{-1}(x) \, dx \right )^{\frac{1}{2}} \|gw^{-1}\|_{L^2(w)} \\
&\leq 4 {[w]_{A_2} \|b\|_{{\rm BMO}^{\mathcal{D}}}} \|f\|_{L^2(w)} \|g\|_{L^2(w^{-1})}.
\end{align*}
Where  we used in the last inequality the estimate  \eqref{Mv-on-Lp(v)} for the weighted dyadic maximal function,
and noting that $h\in L^2(w)$ if and only if $hw\in L^2(w^{-1})$,  moreover $\|hw\|_{L^2(w^{-1})} =\|h\|_{L^2(w)}$ (we used this twice, for $h=f$ and for $h=gw^{-1}$).

\subsubsection*{Second Sum $\Sigma_2$}
Using similar arguments to those
used for $\Sigma_1$ we get
\begin{align*}
\Sigma_2  &\leq \sum_{I \in \mathcal{D}} |b_I| \,\frac{\langle |f|ww^{-1}\rangle)}{\langle w^{-1}\rangle_I}\, \frac{\langle |g|w^{-1}w\rangle_I}{\langle w\rangle_I}\; 
\sqrt{{|\Delta_Iw|^2}{\langle w^{-1}\rangle_I^2} \|I| }\\
& \leq  \sum_{I \in \mathcal{D}} |b_I|\,
\sqrt{\nu_I} \; \inf_{x \in I} M^{\mathcal{D}}_{w^{-1}}(fw) (x) \, M^{\mathcal{D}}_{w}(gw^{-1})(x),
\end{align*}
where $|b_I|^2$ and $\nu_I:=|\Delta_Iw|^2{\langle w^{-1}\rangle_I^2}|I|$ 
are Carleson sequences with intensities
$\|b\|^2_{{\rm BMO}^{\mathcal{D}}}$ and $[w]^{2}_{A_2}$  respectively, by Example~\ref{eg:b-BMO-Carleson} and Example~\ref{eg:nuI-Carleson}(i).
Then by the algebra of Carleson sequences
 the sequence $|b_I| \sqrt{\nu_I}$ is a Carleson
sequence with intensity $\|b\|_{{\rm BMO}^{\mathcal{D}}}[w]_{A_2}$.
Using the Weighted Carleson Lemma (Lemma~\ref{lem:weighted-Carleson}) with $F(x)=M^{\mathcal{D}}_{w^{-1}}(fw) (x) \, M^{\mathcal{D}}_{w}(gw^{-1})(x)$ and with  $v=1$,  we conclude that 
$$
\Sigma_2 \leq [w]_{A_2} \|b\|_{{\rm BMO}^{\mathcal{D}}} \int_{\R} M^{\mathcal{D}}_{w^{-1}}(fw)(x) \, M^{\mathcal{D}}_{w}(gw^{-1})(x) \, dx.
$$
To finish we use the Cauchy-Schwarz inequality,  the fact that $w^{\frac12}(x)\,w^{-\frac12}(x)=1$, and estimate~\eqref{Mv-on-Lp(v)} for the weighted dyadic maximal functions,  to get that,
\begin{align*}
\Sigma_2 
& \leq [w]_{A_2}\|b\|_{{\rm BMO}^{\mathcal{D}}}\left [ \int_{\mathbb{R}} \big (M^{\mathcal{D}}_{w^{-1}}(fw) (x)\big )^2w^{-1}(x)\, dx \right ]^{\frac{1}{2}}\left [ \int_{\mathbb{R}} \big (M^{\mathcal{D}}_{w}(gw^{-1}) (x)\big )^2w(x)\, dx \right ]^{\frac{1}{2}}\\
& = [w]_{A_2} \|b\|_{{\rm BMO}^{\mathcal{D}}} \|M^{\mathcal{D}}_{w^{-1}} (fw) \|_{L^2(w^{-1})}\|M^{\mathcal{D}}_{w}(gw^{-1}) \|_{L^2(w)} \\
&\leq 4{ [w]_{A_2} \|b\|_{{\rm BMO}^{\mathcal{D}}}} \|f \|_{L^2(w)}\|g \|_{L^2(w^{-1})}.
\end{align*}
All together this implies that  $\|\pi_bf\|_{L^2(w)}\leq 8[w]_{A_2}\|b\|_{{\rm BMO}^{\mathcal{D}}}\|f\|_{L^2(w)}$, proving the $A_2$ conjecture for the dyadic paraproduct.
\end{proof}

\subsection{Auxiliary lemmas}\label{sec:loose-ends}
We now present the Bellman function proofs (or at least the main ideas)  for the Little Lemma (Lemma~\ref{thm:Little-Lemma}) and  the $\alpha$-Lemma (Lemma~\ref{lem:alpha-lemma}), to illustrate the method in a very simple setting. For completeness we also present the proof of  the Weighted Carleson Lemma (Lemma~\ref{lem:weighted-Carleson}).  The lemmas in this section hold on $\R^d$ and also on geometrically doubling metric spaces \cite{Ch3, NRV}.

\subsubsection{Proof of Beznosova's Little Lemma}
We wish to prove Lemma~\ref{thm:Little-Lemma}.
The proof uses a Bellman function argument, which we now describe. As usual, the argument proceeds in two steps. First, 
 Lemma~\ref{lem:nduction-on-scales}, encodes what now is called an \emph{induction on scales argument}.
If we can find a Bellman function with certain properties, then we will solve our problem
by induction on scales. This type of arguments shows that if we can find a function with certain size, domain, and dyadic convexity properties tailored to the inequality of interest, we will be able to induct on scales and obtain the desired inequality. Second, Lemma~\ref{lem:existence-of-Bellman-function}
 will show that such Bellman function exists.
 
\begin{lemma}[Beznosova 2008]\label{lem:nduction-on-scales}
\label{prop1 lemma1} Suppose there exists a real valued function of
$3$ variables $B(x) = B(u,v,l)$, whose domain $\mathfrak{D}$ contains points $x=(u,v,l)$
$$ \mathfrak{D} := \{(u,v,l)\in \mathbb{R}^3:
u,v >0, \;\;\; uv \geq 1 \;\;\; and \;\;\; 0\leq l \leq 1\},
$$
whose range is given by
$\;
0 \leq B(x)  \leq u,
$
and such that the following convexity property holds,
\begin{equation}\label{dyadic-convexity}
B(x) - ({B(x_+) + B(x_-)})/{2} \;
\geq \;  \alpha /4v, \;\;\mbox{for all  $x, x_\pm \in \mathfrak{D}$ with $x -
\frac{x_+ + x_-}{2} = (0,0,\alpha)$}.
\end{equation}
Then {\rm Lemma~{\rm \ref{thm:Little-Lemma}}}. will be proven, more precisely \eqref{Little-Lemma}  holds.
\end{lemma}

\begin{proof} Without loss of generality we may  assume that the intensity $A$ of the Carleson sequence $\{\lambda_I\}_{I\in\mathcal{D}}$ in Lemma~\ref{thm:Little-Lemma} is one, $A=1$.

Fix a dyadic interval $J$. Let $u_J :=
\langle w\rangle_J$, $v_J := \langle w^{-1}\rangle_J$ and $\ell_J:=\frac{1}{|J|}\sum_{I\in
\mathcal{D}(J)}{\lambda_I}$, then $x_J := (u_J, v_J, \ell_J)\in \mathfrak{D}$.  Recall that $\mathcal{D}(J)$ denotes the intervals $I\in\mathcal{D}$ such that $I\subset J$.

Let $x_\pm := x_{J^\pm} \in \mathfrak{D}$, then 
$$
x_J - \frac{x_{J^+} + x_{J^-}}{2} \; = \; (0,0,\alpha_J), \;\; \mbox{where} \;\; \alpha_J := \frac{\lambda_J}{|J|}.
$$
Hence, by the size and convexity
property  \eqref{dyadic-convexity},  and  $|J^+| = |J^-| = |J|/2$,
$$
|J| \; \langle w\rangle_J \; \geq \; |J| \; B(x_J) \geq |J^+| B(x_{J^+}) + |J^-|B(x_{J^-}) + {\lambda_J}/{4 \langle w^{-1}\rangle_J}.
$$
Repeat the argument this time for $|J^+| B(x_{J^+})$
and $|J^-|B(x_{J^-})$, use that $B\geq 0$ on $\mathfrak{D}$ and keep repeating to get,  after dividing by $|J|$, that
$$\langle w\rangle_J \geq \frac{1}{4|J|}\sum_{I\in D(J)}{\frac{\lambda_I}{\langle w^{-1}\rangle_I}}$$ which implies \eqref{Little-Lemma} after multiplying through by $4|J|$.  
The lemma is proved.
\end{proof}

The previous induction on scales argument is conditioned on the existence of a function with certain properties, a Bellman function. We now establish the existence of such function, Both lemmas appeared in \cite{Be}.

\begin{lemma}[{Beznosova 2008}]\label{lem:existence-of-Bellman-function}
The function
$
{B(u,v,l) := u - \frac{1}{v(1+l)}}
$
is {\rm (i)}  defined on the domain $\mathfrak{D}$ introduced in {\rm Lemma~\ref{lem:nduction-on-scales}},  {\rm (ii)} $0 \leq B(x) \leq u$ for all
$x=(u,v,l) \in \mathfrak{D}$, and  {\rm (iii)} obeys the following differential estimates on $\mathfrak{D}$:
$$ ({\partial B}/{\partial l})(u,v,l) \;
\geq \; {1}/({4v})  \quad\mbox{and} \quad 
 - \left( du, dv, dl \right) d^2B(u,v,l)
\left( du, dv, dl
 \right )^t \; \geq \; 0,
$$
where $d^2B(u,v,l)$ denotes the Hessian matrix of the function $B$
evaluated at $(u,v,l)$. Moreover, these  imply the dyadic convexity condition
$B(x) - (B(x_+) + B(x_-))/{2} \;
\geq \; \alpha /(4v).
$
\end{lemma}

\begin{proof}
Differential conditions can be checked by a direct calculation that we leave as an exercise for the reader.
By the Mean Value Theorem and some calculus,
$$B(x) - \frac{B(x_+) + B(x_-)}{2} =\frac{\partial B}{\partial l}(u,v,l^\prime)\alpha - \frac{1}{2}
\int_{-1}^1{(1-|t|)b^{\prime\prime}(t)dt}\geq \frac{\alpha}{4v},
$$
where  $b(t) := B(x(t))$ and  $x(t) :=\frac{1+t}{2}\, x_+ + \frac{1-t}{2}\,x_-$ for $-1\leq t \leq 1$.

 Note that $x(t) \in \mathfrak{D}$ whenever $x_+$ and $x_-$ are in the domain, since $\mathfrak{D}$ is a
convex domain and $x(t)$ is a point on the line segment between
$x_+$ and $x_-$, and $l^\prime$ is a point between $l$ and
$\frac{l_+ + l_-}{2}$.  This proves the lemma.
\end{proof}

These two lemmas prove Beznosova's Little Lemma (Lemma~\ref{thm:Little-Lemma}).

\subsubsection{ $\alpha$-Lemma} 
We present a very brief sketch of the argument leading to the proof of the $\alpha$-Lemma  (Lemma~\ref{lem:alpha-lemma}), see  \cite{Be} for $0<\alpha <1/2$, and \cite{MP} for $\alpha\geq 1/2$. Recall that we wish to show that if $w\in A_2$  and $0< \alpha $, then the sequence
\begin{equation*}
  \mu_I:= \langle w\rangle_I^{\alpha}\langle w^{-1}\rangle_I^{\alpha}|I| \bigg( \frac{|\Delta_Iw|^2}{\langle w\rangle_I^2} +\frac{|\Delta_Iw^{-1}|^2}{\langle w^{-1}\rangle_I^2}  \bigg) \quad \mbox{for} \;\; I  \in \mathcal{D}
  \end{equation*}
is a Carleson sequence with  intensity at most {$C_{\alpha}[w]_{A_2}^{\alpha}$}, and $C_{\alpha} =\max\{{72}/(\alpha-2\alpha^2), 576\}$. 

\begin{proof}[Sketch of the Proof]
Use the Bellman function method.
 Figure out the domain, range and dyadic convexity conditions needed to run an induction on
scale argument that will yield the inequality.
 Verify that the Bellman function {$B(u,v) = (uv)^{\alpha }$} satisfies those conditions
(or at least a differential version, that  can then be seen implies the dyadic convexity) for $0<\alpha <1/2$. For $\alpha\geq 1/2$ just observe that one can factor out
$ \langle w\rangle_I^{\alpha-1/4}\langle w^{-1}\rangle_I^{\alpha-1/4}\leq [w]_{A_2}^{\alpha -1/4}$  and then use the already proven lemma when $\alpha =1/4 <1/2$.
\end{proof}

\subsubsection{Weighted Carleson Lemma}\label{sec:W-C-L}
Finally we present a proof of the Weighted Carleson Lemma (Lemma~\ref{lem:weighted-Carleson}), that states that if
   $v$ is a  weight, 
 $\{\alpha_L\}_{L\in\mathcal{D}}$  a $v$-Carleson sequence with intensity~$A$,
and $F$ a positive measurable function on $\R$,  then
$$
\sum_{L\in \mathcal{D}} \alpha_L \inf_{x\in L} F(x) \leq A \int_{\R } F(x)\,v(x)\,dx.
$$
The Weighted Carleson Lemma we present here is a variation in the spirit of other weighted Carleson embedding theorems that appeared before in the literature \cite{NTV1}.  The converse is immediately true by choosing $F(x)=\mathbbm{1}_J(x)$.
\begin{proof}
Assume that $F \in  L^1(v)$ otherwise the first statement
is automatically true. Setting $\displaystyle{\gamma_L =
\inf_{x \in L}F(x)}$, we can write
\begin{equation}\label{eq:CarlesonLemma}
 \sum_{L \in \mathcal{D}} \alpha_L \gamma_L = \sum_{L \in \mathcal{D}} \alpha_L  \int^{\infty}_{0} \chi(L,t)\,dt
= \int_0^{\infty}\Big (\sum_{L\in\mathcal{D}} \chi (L,t)\,\alpha_L \Big )dt,
\end{equation}
where $\chi(L,t)=1$ for $ t < \gamma_L$ and zero otherwise, and where we used the monotone convergence theorem in the last equality.
Define the level set $E_t= \{ x \in \mathbb{R} \; :\; F(x)>t\}$.
 Since $F\in L^1(v)$  then $E_t$ is a $v$-measurable set for
every $t$ and we have, by
Chebychev's inequality, that the $v$-measure of $E_t$ is finite for all
 $t>0$.  Moreover, there is a collection of  maximal disjoint
dyadic intervals $\mathcal{P}_t$ that will cover $E_t$ except for at
most a set of $v$-measure zero. Finally observe that 
 $L \subset E_t$  if and only if $\;\chi(L,t)=1$.
All together we can rewrite the integrand in the right-hand-side of \eqref{eq:CarlesonLemma} as
$$
 \sum_{L \in \mathcal{D}} \chi(L,t) \, \alpha_L = \sum_{L \subset E_t} \alpha_L \leq \sum_{L \in \mathcal{P}_t} \sum_{I \in \mathcal{D}(L)} \alpha_I \leq  A \sum_{L \in \mathcal{P}_t}v(L) = A\,v(E_t),
$$
where we used in the second inequality the fact that $\{\alpha _J \}_{I \in \mathcal{D}}$ is a $v$-Carleson sequence with intensity $A$.
Thus we can estimate
\begin{equation*}
  \sum_{L\in \mathcal{D}} \alpha_L \inf_{x\in L} F(x)= \sum_{L \in \mathcal{D}}  \alpha_L\gamma_L \leq A \int^{\infty}_{0} v(E_t )\, dt = A \int_{\mathbb{R}} F(x)\,v(x)\,dx,
\end{equation*}
where the last equality follows from the layer
cake representation.
\end{proof}

\section{Case study: Commutator of Hilbert transform and function in ${\rm BMO}$}\label{sec:Commutator}

In this section, we  summarize chronologically the weighted norm inequalities known for the commutator
$[b,T]$ where $T$ is a linear operator and $b$ a function in $ {\rm BMO}$. In  particular we will consider  $T=H$ the Hilbert transform.
We   sketch a dyadic proof of the first quantitative weighted estimate for the commutator $[b,H]$ due to  Daewon Chung \cite{Ch2}, yielding the optimal quadratic dependence on the $A_2$ characteristic of the weight. We  discuss a very useful transference theorem of Chung, P\'erez and the author \cite{ChPPz}, and present its proof based on the celebrated Coifman--Rochberg--Weiss argument. The transference theorem allows to deduce quantitative weighted $L^p$ estimates for the commutator of a linear operator  with a ${\rm BMO}$ function, from given quantitative weighted $L^p$ estimates for the operator.

\subsection{$L^p$ theory for $[b,H]$}
Recall that the commutator of  a function $b\in {\rm BMO}$ and  $H$ the Hilbert Transform is defined to be
$$ [b,H](f) := b\,(Hf) -H(bf).$$
 The commutator $[b,H]$  is {bounded on ${L^p(\R )}$}  for ${1<p<\infty}$ if and only if $b\in {\rm BMO}$ \cite{CRW}.  Moreover, the following estimate is known to hold for all $b\in {\rm BMO}$ and $f\in L^p(\R )$
$${ \| [b,H](f) \|_{L^p} \lesssim_p\|b\|_{{\rm BMO}} \|f\|_{L^p}}.$$
In fact the operator norm $\|[b,H]\|_{L^2\to L^2} \sim \|b\|_{{\rm BMO}}$. Observe that  {$bH$} and {$Hb$} are NOT necessarily bounded on  ${L^p}(\R )$ when ${b\in {\rm BMO}}$. {The commutator introduces
some key cancellation}. This is very much connected to the celebrated {$H^1$-${\rm BMO}$} duality theorem by
Feffferman and   Stein  \cite{FS}, where the Hardy space $H^1$ can be defined as those functions $f$  in $L^1(\R )$ such that their maximal function $Mf$ is also in $L^1(\R )$.

The commutator $[b,H]$  is more singular than $H$, as evidenced by
the fact that, unlike the Hilbert transform, the commutator is  {\sc not} of weak-type {$(1,1)$  \cite{Pz1}. In particular the commutator is not a Calder\'on-Zygmund operator, if it were it would be of weak-type $(1,1)$, and  is not.

\subsection{Weighted Inequalities}

The first two-weight results for the commutator that we present are of a qualitative nature. The first one is a two-weight result due to Steven Bloom for the commutator of the Hilbert transform with a function in weighted ${\rm BMO}$  when both weights are in $A_p$ \cite{Bl}.

\begin{theorem}[Bloom 1985]
If {$u,v\in A_p$} then $[b,H]: L^p(u)\to L^p(v)$ is bounded if and only if {$b\in {\rm BMO}_{\mu}$} where {$\mu=u^{-1/p}v^{1/p}$}. Where $b\in {\rm BMO}_{\mu}$ if and only if 
\begin{equation}\label{BloomBMO}
 \|b\|_{{\rm BMO}_{\mu}} := \sup_{I\in \R} \frac{1}{\mu (I)}\int_I |b(x)-\langle b\rangle_{I} |\, dx <\infty.
 \end{equation}
\end{theorem}
What is important in this setting is that the hypothesis $u,v\in A_p$ imply that $\mu\in A_2$ as can be seen by a direct calculation using H\"older's inequality.
The weighted ${\rm BMO}$ space defined by \eqref{BloomBMO} was first introduced by Eric Sawyer and Richard Wheeden \cite{SW}, and it has been called in the literature, somewhat misleadingly,  Bloom ${\rm BMO}$. For a "modern" dyadic proof of Bloom's result see \cite{HoLW1,HoLW2}.

The second result is a one-weight result for very general linear operators $T$  obtained by Josefina \'Alvarez, Richard Bagby, Doug Kurtz, and Carlos P\'erez \cite{ABKPz}, they also prove two-weight estimates.

\begin{theorem}[\'Alvarez, Bagby, Kurtz, P\'erez 1993]  
Let $T$  be a linear operator on the set of real-valued Lebesgue measurable functions defined on $\R^d$, with a domain of definition which contains every compactly supported function in a fixed $L^p$ space.
 If {$w\in A_p$} and $b\in {\rm BMO}$ then there is a constant $C_p(w)>0$ such that for all $f\in L^p(w)$ the following inequality holds,
$$\| [b,T](f) \|_{L^p(w)} \leq C_p(w)\|b\|_{{\rm BMO}} \|f\|_{L^p(w)}.$$
\end{theorem}
 The proof  uses a classical  argument by Raphy Coifman\footnote{As I am writing  these notes, it has been announced that Coifman won the 2018 Schock Prize in Mathematics for his "fundamental contributions to pure and applied harmonic analysis".}, Richard Rochberg, and Guido Weiss \cite{CRW}. In Section~\ref{sec:proof-transference} we will present a quantitative version of this argument \cite{ChPPz}. For a proof of Bloom's result using this type of argument, see \cite{Hyt5}. 

The next result is a quantitative weighted inequality obtained by Daewon Chung in his PhD Dissertation \cite{Ch1, Ch2}.

\begin{theorem}[Chung 2010] For all  $b\in {\rm BMO}$, $w\in A_2$  and $f\in L^2(w)$ the following holds
$$ \| [b,H](f) \|_{L^2(w)} \lesssim \|b\|_{{\rm BMO}} {[w]_{A_2}^2 }\|f\|_{L^2(w)}.$$
\end{theorem}
The quadratic power on the $A_2$ characteristic and the linear bound on the ${\rm BMO}$ norm are both optimal powers. The quadratic dependence on the $A_2$ characteristic is another indication that this operator is more singular than the Calder\'on-Zygmund   singular integral operators for whom the dependence is linear \cite{Hyt2}, as we have emphasized throughout these lectures.

\subsection{Dyadic proof of Chung's Theorem}
We now  sketch Chung's dyadic proof of the quadratic estimate  for the commutator \cite{Ch2}.
\begin{proof}[Sketch of proof] Chung's "dyadic" proof is based on using Petermichl's dyadic shift operators $\Sha^{r,\beta}$ instead of $H$ \cite{Pet1} and proving uniform (on the dyadic grids $\mathcal{D}^{r,\beta}$) quadratic estimates for  the corresponding commutators {$[\Sha^{r,\beta} ,b]$}. To ease notation we drop the superscripts $r,\beta$ and simply write $\Sha$ for $\Sha^{r,\beta}$, the estimates will be independent of the parameters $r$ and $\beta$.

To achieve this we first 
recall the decomposition of a  product ${ bf}$  in terms of paraproducts and their adjoints,
$$
bf = {\pi_bf }+ {\pi_b^*f}+{ \pi_fb},
$$ 
notice that 
the first two terms are bounded on ${L^p(w)}$ when ${b\in {\rm BMO}}$ and ${w\in A_p}$,
the enemy is the third term.
 Decomposing the commutator accordingly we get, \begin{equation}\label{commutatorSha-b}
[b,\Sha](f) =  [\pi_b,\Sha](f) +  [\pi_b^*,\Sha](f) + \big ( \pi_{\Sha f}(b)-\Sha (\pi_fb)\big ).
\end{equation}
 Known linear bounds on $L^2(w)$ for the dyadic paraproduct {$\pi_b$}, its adjoint {$\pi_b^*$}, and  for Petermichl's dyadic shift operator  $\Sha$, see \cite{Be, Pet2}, immediately
give by iteration, {quadratic} bounds for the  first two terms on the right-hand-side of \eqref{commutatorSha-b}. Surprisingly, the third term is better,
it obeys a {linear} bound, and so do halves of the  first two commutators, as shown in \cite{Ch2} using \emph{Bellman function} techniques, namely
$$
{\|\pi_{\Sha f}(b)-\Sha (\pi_fb)\|_{L^2(w)}}+ {\|  \Sha \pi_b(f)\|_{L^2(w)}}+{\|\pi_b^*\Sha (f)\|_{L^2(w)}}  \leq  C \|b\|_{{\rm BMO}} {[w]_{A_2}}\|f\|_{L^2(w)}.
$$
All together   providing uniform (on the random dyadic grids $\mathcal{D}^{r,\beta}$)  quadratic bounds for the commutators ${[b,\Sha^{r,\beta}]}$, and hence, averaging over the random grids we get the desired quadratic estimate for $[b,H]$. 
\end{proof}

The quadratic estimate  and the corresponding extrapolated estimates, namely for all $b\in {\rm BMO}$, $w\in A_p$, and $f\in L^p(w)$ 
\begin{equation}\label{optimal-Lp-bound-commutator}
\|[b,H](f)\|_{L^p(w)}\lesssim_p {[w]_{A_p}^{2\max\{1,\frac{1}{p-1}\}}}\|b\|_{{\rm BMO}}\|f\|_{L^p(w)},
\end{equation}
are optimal for all $1<p<\infty$, as can be seen considering appropriate power functions and  power weights \cite{ChPPz}.

The "bad guys" are the  non-local terms {$\pi_b\Sha$}, {$\Sha\pi_b^*$}.   A posteriori one realizes the  pieces that obey linear bounds are generalized Haar Shift operators and hence
their linear bounds can be deduced from general results for those operators.

 As a byproduct of Chung's dyadic proof we get that 
the  extrapolated bounds for the dyadic paraproduct  are optimal \cite{P2}, namely for all $b\in {\rm BMO}$, $w\in A_p$, and $f\in L^p(w)$ 

$$
\|\pi_bf\|_{L^p(w)}\lesssim_p {[w]_{A_p}^{\max\{1,\frac{1}{p-1}\}}}\|b\|_{{\rm BMO}}\|f\|_{L^p(w)}.
$$

\begin{proof} By  contradiction, if not for some $p$ then $[b,H]$ will have a better bound
in $L^p(w)$ than the known optimal  bound given by \eqref{optimal-Lp-bound-commutator} for that $p$.
\end{proof}

\subsection{A quantitative transference theorem} 
The following theorem provides a mechanism for transferring known quantitative weighted estimates for linear operators to their commutators with ${\rm BMO}$ functions \cite{ChPPz,P2}.

\begin{theorem}[Chung, Pereyra, P\'erez 2012] 
Given  linear operator $T$ and $1<r<\infty$,  such that {for all $w\in A_r$} and  $f\in L^r(w)$ the following estimate holds
\vskip -.1in
$$
\|Tf\|_{L^r(w)}\lesssim_{T,d} {[w]^{\alpha}_{A_r}} \|f\|_{L^r(w)},
$$
then the  commutator of $T$ with $b\in {\rm BMO}$ is such that  for all $w\in A_r$ and $f\in L^r(w)$
$$
  \|{[b,T]}(f)\|_{L^r(w)} \lesssim_{r,T,d} {[w]_{A_r}^{\alpha +\max\{1,\frac{1}{r-1}\}} \|b\|_{{\rm BMO}}} \|f\|_{L^r(w)}.
  $$
\end{theorem}

The proof follows the classical Coifman--Rochberg--Weiss  argument using (i) the Cauchy integral formula; (ii) the following quantitative Coifman-Fefferman result: $w\in A_r$ implies $w\in RH_q$ with 
$q=1+c_d/[w]_{A_r}$ and  $[w]_{RH_q}\leq 2$; (iii) a quantitative version of the estimate:  $b\in {\rm BMO}$ implies $e^{\alpha b}\in A_r$ for $\alpha$ small enough with control on $[e^{\alpha b}]_{A_r}$. 
We will present the whole argument in  the case $r=2$ in  Section~\ref{sec:proof-transference}. Here the \emph{Reverse H\"older-$q$ weight class} ($RH_q$) for $1<q<\infty$ is defined to be all those weights $w$ such that
\[ [w]_{RH_q} := \sup_{Q} \, \langle w^q\rangle_Q^{1/q}\langle \,w\rangle_Q^{-1} <\infty,\]
where the supremum is taken over all cubes in $\R^d$ with sides parallel to the axes.

A variation on the argument yields corresponding estimates for the \emph{higher order commutators }$T^k_b:=[b,T^{k-1}_b]$ for $k\geq 1$ and $T^0_b:=T$. More precisely, given the initial estimate $\|T^0_bf\|_{L^r(w)} \lesssim [w]_{A_r}^{\alpha}$, valid for all $w\in A_r$,  then the following estimate holds for all $k\geq 1$, $b\in {\rm BMO}$, $w\in A_r$, and $f\in L^r(w)$
$$
  \|{T^k_b}f\|_{L^r(w)} \lesssim_{r,T,d} {[w]_{A_r}^{\alpha +k\max\{1,\frac{1}{r-1}\}} \|b\|^k_{{\rm BMO}}} \|f\|_{L^r(w)}.
  $$

Transference theorems for commutators are useless unless there are operators known to obey an  initial  $L^r(w)$ bound valid for all $w\in A_r$.
We  have already mentioned  that the class of Calder\'on-Zygmund singular integral operators obey linear bounds on $L^2(w)$ thanks to Hyt\"onen's $A_2$ theorem \cite{Hyt2}. We conclude that for all Calder\'on-Zygmund singular integral operators $T$ their commutators obey a quadratic bound on $L^2(w)$, more precisely, 
$$
\|[b,T]f\|_{L^2(w)}\lesssim_{T,d} {[w]^2_{A_2}  \|b\|_{{\rm BMO}}} \, \|f\|_{L^2(w)}.
$$
With a slight modification of the argument one can see \cite{ChPPz} that the correct estimate for the iterated commutators of Calder\'on-Zygmund singular integral operators and function $b\in {\rm BMO}$ is
$$
\|[T_b^kf\|_{L^2(w)}\lesssim_{T,d} {[w]^{1+k}_{A_2} \|b\|^k_{{\rm BMO}} } \, \|f\|_{L^2(w)}.
$$
There are operators (for example the Hilbert, Riesz, and Beurling transforms) for whom these estimates are optimal in terms of the powers both for the $A_2$ characteristic and the $BMO$ norm. This can be seen testing power functions and weights \cite{ChPPz}.

\subsubsection{Some generalizations}

There are extensions to commutators with fractional integral operators, 
 two-weight problem  and more \cite{CrMoe,Cr}.  
There are mixed $A_2$-$A_{\infty}$ estimates, where  recall that $A_{\infty}=\cup_{p>1}A_p$ and   $[w]_{A_{\infty}}\leq [w]_{A_2}$  \cite{HPz, OrPzRe}, more precisely estimates of the form,
$$
\| [b,T]\|_{L^2(w)} \lesssim {[w]^{\frac{1}{2}}_{A_2}\big ( [w]_{A_{\infty}} + [w^{-1}]_{A_{\infty}}\big )^{\frac{3}{2}} \|b\|_{{\rm BMO}}}.
$$
 
There are generalizations to commutators of matrix valued operators  and ${\rm BMO}$ \cite{IKP} as well as 
  to the two-weight setting (both weights in $A_p$, \`a la Bloom) \cite{HoLW1,HoLW2},  and
 also for  biparameter Journ\'e operators  \cite{HoPetW}. See also the comprehensive paper \cite{BMMST} where a systematic use of the Coifman--Rochberg--Weiss trick recovers all known results and some new ones such as boundedness of the commutator of the bilinear Hilbert transform and a function in ${\rm BMO}$.
  Pointwise control by sparse operators adapted to the commutator, improving weak-type, Orlicz bounds, and quantitative two weight Bloom bounds was recently obtained  \cite{LeOR1,LeOR2}.} We will say more about this generalization in Section~\ref{sec:Sparse}.

\subsubsection{Proof of the transference theorem}\label{sec:proof-transference}
We now present the proof of the quantitative transference theorem when $r=2$, following the lines of the Coifman--Rochberg--Weiss argument \cite{CRW} with a few quantitative ingredients. For $r\neq 2$ see \cite{P2}.
\begin{proof}[Proof in \cite{ChPPz}]
"Conjugate" the operator as follows: for any  $z\in \C$  define
$$T_z(f)=e^{zb}\, T\, (e^{-zb}f).$$
A computation together with the  Cauchy integral theorem give (for "nice" functions),
$$
[b,T](f)=\frac{d}{dz}T_z(f)|_{z=0}=\frac{1}{2\pi i}\int_{|z|=\epsilon}
\frac{T_z(f)}{z^2}\,dz , \quad \epsilon>0.$$
Now, by Minkowski's integral inequality
\[\|[b,T](f)\|_{L^2(w)}\leq \frac{1}{2\pi\,\epsilon^2} \,\int_{|z|=\epsilon}
\|T_z(f)\|_{L^2(w)}|dz|, \quad \epsilon>0.\]
The key point is to find an {appropriate radius $\epsilon>0$}. 
To that effect, we look at  the inner norm and try to find  bounds depending on $z$.  More precisely,
$$
\|T_z(f)\|_{L^2(w)}=\|T(e^{-zb}f)\|_{L^2 ({w\, e^{2{\rm Re}z\,b} })}.
$$
We use the  main hypothesis, namely that $T$ is bounded on $L^2(v)$ if $v\in A_2$ with
$\|T\|_{L^{2}(v)} \le C[v]_{A_2}$,  for  $v=w\, e^{2{\rm Re} z\,b }$. 
We must check that if $w\in A_2$ then $v\in A_2$  for $|z|$ small enough. Indeed,
$$
[v]_{A_2}= \sup_Q \left (\frac1{|Q|}\int_Q w(x)\, e^{2{\rm Re}
z\,b(x) }\,dx \right )\left (\frac1{|Q|}\int_Q w^{-1}(x) \,e^{-2{\rm Re} z\,b(x)
}\,dx\right ).
$$
It is well known that if   $w\in A_2$ then $w\in RH_q$ for some $q>1$  \cite{CoFe}. There is a quantitative
version of this result \cite{Pz2}, namely if
{$q=1+\frac{1}{2^{d+5}[w]_{A_2}}\;$} then
$$
\left ( \frac{1}{|Q|}\int_Q w^{q}(x) \, dx\right )^{\frac{1}{q}} \leq
\frac{2}{|Q|}\int_Q w(x) \,d x , 
$$
similarly for  $w^{-1}\in A_2$ and for the same $q$, since  $[w]_{A_2}=[w^{-1}]_{A_2}$, we have that
$$
\left ( \frac{1}{|Q|}\int_Qw^{-q}(x)\, dx\right )^{\frac{1}{q}} \leq
\frac{2}{|Q|}\int_Qw^{-1} (x)\, dx\, .
$$

In what follows $q=1+{1}/{(2^{d+5}[w]_{A_2})}$.
Using these estimates and Holder's inequality we have for an arbitrary  cube $Q$
\begin{align*}
&\hskip -.1cm \left (\frac1{|Q|}\int_Q w(x)e^{2{\rm Re} z\,b(x)}\,dx \right )\left
(\frac1{|Q|}\int_Q w(x)^{-1}e^{-2{\rm Re} z\,b(x) }\,dx \right )\\
&  \leq \left ( \frac{1}{|Q|}\int_Q w^{q}(x) \, dx \right )^{\frac{1}{q}} \hskip -.05in
\left ( \frac1{|Q|} \int_Q e^{2{\rm Re} zq'b(x) } \,dx\right )^{\frac{1}{q'}}  \hskip -.06in
\left ( \frac{1}{|Q|} \int_Q w^{-q}(x) \,dx\right )^{\frac{1}{q}}   \hskip -.05in
\left ( \frac1{|Q|} \int_Q e^{-2{\rm Re} zq'b(x)}\,dx\right )^{\frac{1}{q'}}\\
& \leq 4\,\left (\frac{1}{|Q|} \int_Q w(x) \, dx \right ) \hskip -.05in 
\left (\frac{1}{|Q|} \int_Q w^{-1}(x) \, dx \right ) \hskip -.05in 
\left (\frac1{|Q|} \int_Q e^{2{\rm Re} z\,q'\,b(x) } \,dx \right )^{\frac{1}{q'}}  \hskip -.05in 
\left ( \frac1{|Q|} \int_Q e^{-2{\rm Re} z\,q'\,b(x) } \, dx \right )^{\frac{1}{q'}}\\
&\hskip .8cm \leq 4\,{[w]_{A_2}}\,[e^{2{\rm Re} z\,q'\,b}\,  ]_{A_2}^{\frac{1}{q'}}.
\end{align*}
Taking the supremum over all cubes we conclude that
$$ [v]_{A_2}=[w\, e^{2{\rm Re} z\,b}]_{A_2} \leq 4\,{[w]_{A_2}}\,[e^{2{\rm Re} z\,q'\,b}]_{A_2}^{\frac{1}{q'}}.$$
Now, since $b\in {\rm BMO}$  there are $0<\alpha_d <1$ and $\beta_d >1$ such that
if $\;  |2{\rm Re} z\,q'| \leq {\alpha_d}/\|b\|_{{\rm BMO}}$
 then $ [e^{2{\rm Re} z\,q'\,b}]_{A_2}\leq \beta_d$, see  \cite[Lemma 2.2]{ChPPz}.
Hence for these $z$, 
$$
[v]_{A_2}\leq 4\,[w]_{A_2}\, \beta_d^{\frac{1}{q'}}\leq
4\, {[w]_{A_2}}\, \beta_d .
$$
We have shown that if  $|z| \leq \alpha_d / (2q'\|b\|_{{\rm BMO}})$ then $[v]_{A_2}\leq 4 [w]_{A_2}\,\beta_d$ and
$$
\|T_z(f)\|_{L^2(w)}=\|T(e^{-zb}f)\|_{L^2 (v)}\lesssim
[v]_{A_2}\|f\|_{L^2(w)} \leq 4{[w]_{A_2}}\,
\beta_d\,
 \|f\|_{L^2(w)}. 
$$
Where the first inequality holds since  $\|e^{-zb}f\|_{L^2(v)}=\|e^{-zb}f\|_{L^2(we^{2{\rm Re} z \,b})}=\|f\|_{L^2(w)}$.

Thus choose the radius
$\displaystyle{\;\epsilon:= {\alpha_d}/({2q'\|b\|_{{\rm BMO}}}) ,}$ and get
\begin{align*}
\|[b,T](f)\|_{L^2(w)} & \leq \frac{1}{2\pi\,\epsilon^2}
\,\int_{|z|=\epsilon} \|T_z(f)\|_{L^2(w)}|dz|\\
&\leq \frac{1}{2\pi\,\epsilon^2} \,\int_{|z|=\epsilon} 4{[w]_{A_2}}\,
\beta_d\,  \|f\|_{L^2(w)} |dz|
= \frac{1}{\epsilon}\,4{[w]_{A_2}}\, \beta_d\,\|f\|_{L^2(w)}.
\end{align*}
Note that  $\epsilon^{-1} \approx {[w]_{A_2}}\|b\|_{{\rm BMO}}$, because $q'=1+2^{d+5}[w]_{A_2}\approx 2^d[w]_{A_2}$, 
$$\|[b,T](f)\|_{L^2(w)}\leq  C_d\,{[w]^2_{A_2}}\,\|b\|_{{\rm BMO}}.
$$
Which is exactly what we wanted to prove. \end{proof}

\section{Sparse operators and sparse  families of dyadic cubes}\label{sec:Sparse}

In this section, we discuss the sparse domination by finitely many positive dyadic operators paradigm that has 
recently emerged as a byproduct of the study of weighted inequalities. This sparse domination paradigm has proven to be very powerful with applications in  areas other than  weighted norm inequalities. In this section, we introduce the sparse operators and the sparse families of cubes.
We discuss a characterization of sparse families of cubes via Carleson families of dyadic cubes due to Andrei Lerner and Fedja Nazarov, however this was well known 20 years earlier by Igor Verbitsky \cite[Corollary 2, p.23]{Ve}, see also \cite{Ha}. We  present the beautiful proof of the $A_2$ conjecture for sparse operators due to David Cruz-Uribe, Chema Martell, and  Carlos P\'erez. We record the sparse domination results for the operators discussed in these notes. We present   how to  dominate  pointwise the martingale transform by a sparse operator following Michael Lacey's argument, illustrating the technique in a toy model. Finally we briefly discuss a sparse domination theorem for commutators valid for (rough) Calderón-Zygmund singular integral operators due to Andrei Lerner, Sheldy Ombrosi, and Israel Rivera-Ríos that yields  new quantitative two weight estimates of Bloom type, and recovers all known  weighted results for the commutators.

\subsection{Sparse operators}\label{sec:sparse-operators}

David Cruz-Uribe, Chema Martell, and Carlos P\'erez  showed in \cite{CrMPz2} the $A_2$ conjecture in a few lines for
\emph{sparse operators} $\mathcal{A}_{\mathcal{S}}$ defined as follows,
$$
 \mathcal{A}_{\mathcal{S}} f (x)= \sum_{Q\in \mathcal{S}} {\langle f\rangle_Q}\,\mathbbm{1}_Q(x).
 $$
Here  $\mathcal{S}$ is a sparse collection of dyadic cubes.  A collection of dyadic cubes $\mathcal{S}$ in $\R^d$ is \emph{$\eta$-sparse}, $0<\eta<1$ if there are pairwise disjoint measurable sets $E_Q$ for each $Q\in \mathcal{S}$ such that
 $$
 E_Q\subset Q \;\;\mbox{with}\;\;  |E_Q|\geq \eta |Q| \quad \mbox{for all} \;\; Q\in\mathcal{S}.
 $$
  A primary example for us  are the Calder\'on-Zygmund singular integral operators, they and the "rough" Calder\'on-Zygmund operators  have been shown to be   pointwise dominated by a finite number of sparse operators  \cite{Le4,CR,LeN,L4}. A  quantitative form of these estimates can be found in  \cite{Le5,HRoTa}.
 More recently see sparse domination principles for rough Calder\'on-Zygmund singular integral operators  \cite{CuDiPOu1,CCuDiPOu,HRoTa, DiPHLi}.

\subsection{Sparse vs Carleson families of dyadic cubes}
We have seen in Section~\ref{sec:Dyadic-operators} how  Carleson sequences and Carleson embedding lemmas come handy when proving weighted inequalities. There is an intimate connection between Carleson families of cubes and sparse families of cubes.
 A family of dyadic cubes $\mathcal{S}$ in $\R^d$ is called \emph{$\Lambda$-Carleson} for  $\Lambda >1$ if
 $$
 \sum_{P\in\mathcal{S}, P\subset Q}|P|\leq \Lambda |Q| \quad \forall Q\in\mathcal{D}.
 $$
 
 Notice that a family of cubes being  $\Lambda$-Carleson is equivalent to  the sequence $\{|P|\mathbbm{1}_{\mathcal{S}}(P)\}_{P\in\mathcal{D}}$ being  Carleson with intensity $\Lambda$. Furthermore the notion is equivalent to the family of cubes being $1/\Lambda$-sparse.  These type of conditions are also called \emph{Carleson packing conditions}. 
 
 \begin{lemma}[Verbitsky 1996, Lerner, Nazarov 2014] Let $\Lambda >1$.
 The family of dyadic cubes  $\mathcal{S}$ in $\R^d$  is $\Lambda$-Carleson if and only if  $\mathcal{S}$ is $1/\Lambda$-sparse.
  \end{lemma}
 
 \begin{proof} We sketch the beautiful argument in \cite{LeN}.\\
 
  \noindent{\bf ($\Leftarrow$)} The family of cubes $\mathcal{S}$  being $1/\Lambda$-sparse means that for all cubes $P\in\mathcal{S}$
 there are  pairwise disjoint  subsets $E_P\subset P$ that have a considerable portion of the total mass of the cube, more precisely   $\Lambda |E_P| \geq  |P| $.
 Hence, 
 $$  \sum_{P\in\mathcal{S}, P\subset Q}|P|\leq  \Lambda \sum_{P\in\mathcal{S}, P\subset Q}|E_P| \leq \Lambda |Q|.$$
Where the last inequality holds because the sets $E_P\subset Q$ and are pairwise disjoint.  Therefore the family of cubes $\mathcal{S}$ is $\Lambda$-Carleson.\\

\noindent {\bf ($\Rightarrow$)} 
Assume now that $\mathcal{S}$ is a $\Lambda$-Carleson family. We say that a  \emph{family $\mathcal{S}$ has a bottom layer} $\mathcal{D}_K$ if for all $Q\in\mathcal{S}$ we have  $Q\in \mathcal{D}_k$ for some $k\leq K$. 
Assume {\sc  $\mathcal{S}$ has a bottom layer $\mathcal{D}_K$}. Then consider all cubes in the bottom layer, $Q\in \mathcal{S}\cap \mathcal{D}_K$,  and {choose any} sets $E_Q\subset Q$ with $|E_Q| = \frac{1}{\Lambda}|Q|$. This choice is always possible, because of the nature of the Lebesgue measure, and the sets will automatically be pairwise disjoint because the cubes in a fixed generation $\mathcal{D}_K$ are pairwise disjoint.  
Then go up layer by layer,  meaning we have already selected sets $E_R\subset R$ for all $R\in \mathcal{S}\cap\mathcal{D}_j$ and $k<j\leq K$ with the property that $|E_R| = \frac{1}{\Lambda}|R|$, then for each $Q\in \mathcal{D}_k$, $k< K$, {choose any} $E_Q\subset Q\setminus\cup_{R\in\mathcal{S},R\subsetneq Q}E_R$
 with $|E_Q|=\frac{1}{\Lambda}|Q|$.  Such choice is always possible because for every $Q\in\mathcal{S}$ we have
 $$ \Big | \cup_{R\in\mathcal{S},R\subsetneq Q}E_R\Big | \leq \frac{1}{\Lambda}\sum_{R\in\mathcal{S},R{\subsetneq} Q} |R|\; {\leq}\; \frac{\Lambda -1}{\Lambda} |Q|=\Big (1-\frac{1}{\Lambda}\Big )|Q|,$$
where we used in the inequality  the hypothesis that $\mathcal{S}$ is a $\Lambda$-Carleson family.
Therefore  
$$|Q\setminus\cup_{R\in\mathcal{S},R\subsetneq Q}E_R|\geq  \frac{1}{\Lambda}|Q|,$$
hence there is enough mass left in $Q$, after removing the sets $E_R$  corresponding to $R$ in $\mathcal{S}$ and proper subcubes of $Q$, to select a subset $E_Q$ of $Q$ with the aformentioned property. Moreover  by construction the
sets $E_Q$ are pairwise disjoint, and we are done. 

{{\sc But, what if there is no bottom layer?}} The idea is to run the construction for each $K\geq 0$ and pass to the limit! One has  to be a bit careful! As Lerner and Nazarov put it:
 \emph{``All we have to do is replace "free choice" with "canonical choice".''}  The diligent reader can find the details of the argument, including a very illuminating picture, in  \cite[Lemma 6.3 and Figure 8]{LeN}.
\end{proof}

\subsection{$A_2$ theorem for sparse operators}
We now present David Cruz-Uribe, Chema Martell, and  Carlos P\'erez's  beautiful proof of the $A_2$ conjecture for sparse operators \cite{CrMPz2}.
\begin{theorem}[Cruz-Uribe, Martell, P\'erez 2012] Let $\mathcal{S}$ be an $\eta$-sparse family of cubes then forall~$w\in A_2$ and  $f\in L^2(w)$ the following inequality holds
\begin{equation}\label{A2-conjecture-sparse}
\| \mathcal{A}_{\mathcal{S}}f\|_{L^2(w)}\lesssim {[w]_{A_2}} \|f\|_{L^2(w)}.
\end{equation}
\end{theorem}

\begin{proof}
For $w\in A_2$,  $\mathcal{S}$  and $\eta$-sparse family with $\eta\in (0,1)$,  showing \eqref{A2-conjecture-sparse}
 is equivalent by duality to showing that for all $f\in L^2(w)$, $g\in L^2(w^{-1})$
 $$|\langle \mathcal{A}_{\mathcal{S}} f, g\rangle | \lesssim  [w]_{A_2} \|f\|_{L^2(w)}\|g\|_{L^2(w^{-1})}.$$
By the Cauchy-Schwarz inequality $|E_Q| = \int_{E_Q} w^{\frac12}w^{-\frac12} \leq (w(E_Q))^{\frac12} (w^{-1}(E_Q))^{\frac12}$. Using the definition of the sparse operator, some algebra  and the definition of an $\eta$-sparse family of cubes, namely $|Q|\leq (1/\eta) |E_Q|$ we get that
\begin{eqnarray*}
|\langle \mathcal{A}_{\mathcal{S}} f, g\rangle | & \leq &
 \sum_{Q\in \mathcal{S}} \langle |f| \rangle_Q\, \langle |g| \rangle_Q \, |Q|  \\
 &\leq &  \frac{1}{\eta} \sum_{Q\in \mathcal{S}} \frac{\langle |f| w w^{-1} \rangle_Q}{\langle w^{-1}\rangle_Q} \, \frac{ \langle |g| w^{-1}w \rangle_Q}{\langle w\rangle_Q} \,  {\langle w\rangle_Q} \, {\langle w^{-1}\rangle_Q}\, |E_Q|
\\ 
 &\leq & \frac{{[w]_{A_2}}}{\eta} \sum_{Q\in \mathcal{S}} 
 \frac{\langle |f| w w^{-1} \rangle_Q}{\langle w^{-1}\rangle_Q} \, (w^{-1}(E_Q))^{\frac12} \, 
   \frac{ \langle |g| w^{-1}w \rangle_Q}{\langle w\rangle_Q} \, (w(E_Q))^{\frac12}.
   \end{eqnarray*}
Using once more the Cauchy-Schwarz inequality and the fact that for all $x\in E_Q\subset Q$ it holds that
$\langle |h| v\rangle_Q /\langle v \rangle_Q\leq M^{\mathcal{D}}_vh(x)$  therefore $|\langle |h| v\rangle_Q /\langle v \rangle_Q |^2v(E_Q)\leq \int_{E_Q} |M^{\mathcal{D}}_vh(x)|^2 \,v(x)\,dx$, we conclude that 
\begin{eqnarray*}
|\langle \mathcal{A}_{\mathcal{S}} f, g\rangle |&   \leq  & \frac{[w]_{A_2}}{\eta}   \bigg [\sum_{Q\in \mathcal{S}} 
 \frac{\langle |f| w w^{-1} \rangle_Q^2}{\langle w^{-1}\rangle_Q^2}\, w^{-1}(E_Q)\bigg ]^{\frac12}
  \bigg [\sum_{Q\in \mathcal{S}}   \frac{ \langle |g| w^{-1}w \rangle_Q^2}{\langle w\rangle_Q^2} \, w(E_Q)\bigg ]^{\frac12} \\
&  & \hskip -.5in \leq \;\; \frac{[w]_{A_2}}{\eta} 
\bigg [\sum_{Q\in \mathcal{S}} \int_{E_Q} |M^{\mathcal{D}}_{{w^{-1}}}(fw)(x)|^2{w^{-1}(x)} \, dx\bigg ]^{\frac12}
\bigg [\sum_{Q\in \mathcal{S}} \int_{E_Q} |M^{\mathcal{D}}_{{w}}(gw^{-1})(x) |^2{w}(x) \,dx\bigg ]^{\frac12} \\
 &  &  \hskip -.5in  \leq \;\;  \frac{[w]_{A_2}}{\eta} \|M^{\mathcal{D}}_{w^{-1}}(fw)\|_{L^2(w^{-1})} \, \|M^{\mathcal{D}}_{w}(gw^{-1})\|_{L^2(w)}\\
 &  &  \hskip -.5in  \lesssim \;\;  {[w]_{A_2}} \|fw\|_{L^2(w^{-1})} \, \|gw^{-1}\|_{L^2(w)} 
 \; = \; {{[w]_{A_2}}} \|f\|_{L^2(w)} \, \|g\|_{L^2(w^{-1})}.  
  \end{eqnarray*}
  Where, in the line before the last, we used the fact that the sets $E_Q$ for $Q\in\mathcal{S}$ are pairwise disjoint and, in the
   last line, we used estimate \eqref{Mv-on-Lp(v)} for the weighted dyadic maximal functions.
  \end{proof}
Similar argument yields linear bounds on $L^p(w)$ for $p>2$ and by duality (sparse operators are self-adjoint) we get  bounds like  $[w]_{A_p}^{{1}/{(p-1)}}=[w^{{-1}/{(p-1)}}]_{A_{p'}}$ when $1<p<2$, see  \cite{Moe}. 
In other words, we can get directly the same $L^p(w)$  bounds that sharp extrapolation will give if we were to extrapolate from the linear $L^2(w)$ bounds, namely, for all $w\in A_p$ and $f\in L^p(w)$
\begin{equation}\label{Lpw-bounds-sparse-operators}
 \| \mathcal{A}_{\mathcal{S}}f\|_{L^p(w)} \lesssim  [w]_{A_p}^{\max\{1,\frac{1}{p-1}\}} \|f\|_{L^p(w)}.
 \end{equation}

\subsection{Domination by Sparse Operators}

Many operators can be  dominated by finitely many sparse operators,  pointwise, in norm, or by forms. 
The collections $\mathcal{S}$, $\mathcal{S}_i$ are sparse families  tailored to the operator and the particular function $f$ the operator is acting on. Identifying these sparse families is where most of the work lies, usually  done using some sort of weak-$(1,1)$ inequality that is available a priori, or a specific stopping time designed for the problem at hand. We will illustrate this process for the martingale transform in Section~\ref{sec:sparse-martingale-Lacey}. Here is the status, in terms of sparse domination, of the operators we have been discussing in these lecture notes. In particular quantitative weighted estimates for corresponding sparse operators, such as \eqref{Lpw-bounds-sparse-operators}, immediately transfer to the dominated operators, providing new and streamlined proofs of the quantitative weighted inequalities we have been focusing on previous sections.

The martingale transforms and the dyadic paraproduct  are locally pointwise dominated by sparse operators \cite{L4}. More precisely,  given a cube $Q_0$ and $f\in L^1(\R )$ there are  sparse families $\mathcal{S}, \mathcal{S}'$ such that 
$$|\mathbbm{1}_{Q_0}T_{\sigma}(f\mathbbm{1}_{Q_0}) | \lesssim\mathcal{A}_\mathcal{S}|f|, \quad |\mathbbm{1}_{Q_0} \pi_b(f\mathbbm{1}_{Q_0}) |\lesssim\mathcal{A}_{\mathcal{S}'}|f|.$$ 
We will say more about the martinagale transform in Section~\ref{sec:sparse-martingale-Lacey}.

 Calder\'on-Zygmund operators are pointwise dominated by finitely many sparse operators \cite{CR,Le5,LeN}. More precisely,  given $T$ and $f$ there are finitely many sparse families $\mathcal{S}_i$, for $i=1,\dots, N_d$, such that 
$$|Tf|\leq \sum_{i=1}^{N_d}\mathcal{A}_{\mathcal{S}_i}f.$$ 

 The dyadic square function is pointwise dominated by finitely many sparse-like operators \cite{LLi2}. More precisely,   given $f$ there are finitely many sparse families $\mathcal{S}_i$, for  $i=1,\dots, N_d$, such that 
$$|S^{\mathcal{D}}f|^2 \leq \sum_{i=1}^{N_d}\sum_{Q\in\mathcal{S}_i} \langle |f|\rangle_Q^2\mathbbm{1}_Q.$$
Notice that the sparse-like operators have been adapted to the square function.

Commutator $[b,T]$  for  $T$ an $\omega$-Calder\'on-Zygmund operator with $\omega$ satisfying a Dini condition, $b\in L^1_{{\rm loc}}(\R )$ can be pointwise dominated by finitely many sparse-like operators and their adjoints \cite{LeOR1,LeOR2}. We will say more about this in Section~\ref{sec:Sparse-vs-commutators}.

The finitely many sparse families come from the analogue of the one-third trick for the dyadic grids, usually $N_d=3^d$ will suffice.

\subsection{Domination of martingale transform d'apr\`es {Lacey}}\label{sec:sparse-martingale-Lacey}
We would like to illustrate how to achieve domination by sparse operators for a toy model operator, the martingale transform $T_{\sigma}$ on $L^2(\R )$. Following  an argument of Michael Lacey  \cite[Section 3]{L4}.

Given interval $I_0\in\mathcal{D}$ and function $f\in L^1(\R )$ supported on $I_0$,  we need to find a $1/2$-sparse family $\mathcal{S}\subset\mathcal{D}$, such that for all  choices of signs $\sigma$, there is a constant $C>0$ such that 
$$|\mathbbm{1}_{I_0}T_{\sigma}f| \leq C \mathcal{A}_{\mathcal{S}}|f|.$$

\begin{proof}  Without loss of generality we can assume that $f\in L^1(\R )$ is not only supported on $I_0$ but also $\int_{I_0}|f(x)| dx >0$. We will need  the following well-known  weak-type estimates. 

First, the  {sharp truncation $T_{\sigma}^{\sharp}$ is of weak-type $(1,1)$} \cite{Bur1}, with a constant independent of the choice of signs $\sigma$, thus
$$ \sup_{\lambda>0} \lambda \big |\{x\in \R: T^{\sharp}_{\sigma}f(x)>\lambda\}\big |\leq  C\|f\|_{L^1(\R)},$$
where 
${T^{\sharp}_{\sigma}f=\sup_{I'\in\mathcal{D}}\big |\sum_{I\in\mathcal{D}, I\supset I'} \sigma_I\langle f,h_I\rangle h_I  }\big |$.

Second, the maximal function {$M$ is also of weak-type $(1,1)$}, therefore
$$ \sup_{\lambda>0} \lambda \big |\{x\in \R: Mf(x)>\lambda\}\big |\leq  C\|f\|_{L^1(\R)},$$

As a consequence there exists a constant $C_0>0$ such that the subset of $I_0$ defined by
$$ F_{I_0}:=\{x\in I_0: \max\{Mf(x), T^{\sharp}_{\sigma}f(x)\}>\mbox{$\frac{1}{2}$}C_0\langle |f|\rangle_{I_0}\}$$
has no more than half the mass of $I_0$,  that is, $|F_{I_0}|\leq \frac{1}{2} |I_0|$.  In fact, suppose no such constant would exist, then for all $C_0>0$ it would hold that 
$$|F_{I_0}| =  \Big |\{x\in I_0: \max\{Mf(x), T^{\sharp}_{\sigma}f(x)\}>\mbox{$\frac{1}{2}$}C_0\langle |f|\rangle_{I_0}\} \Big | > \frac{1}{2} |I_0|,$$ 
therefore for each $C_0>0$ it must be that either
$ |\{x\in I_0: Mf(x)> \frac{1}{2}C_0\langle |f|\rangle_{I_0}\}| > \frac{1}{4}|I_0|$  or
$ |\{x\in I_0: T_{\sigma}^{\sharp}f(x)> \frac{1}{2}C_0\langle |f|\rangle_{I_0}\}| > \frac{1}{4}|I_0|$.
But either of these sets has measure bounded above by $2C\|f\|_{L^1(\R )}/ (C_0\langle|f|\rangle_{I_0})$, choosing $C_0$ large enough, so that $2C\|f\|_{L^1(\R )}/ (C_0\int_{I_0} |f(y)|dy ) <1/4$, a contradiction will be reached. It sems as if the constant $C_0$ depends on the interval $I_0$, however once we recall that the function $f$ is supported on $I_0$ then all is required is that $2C/C_0 < 1/4$.

 Let  $\mathcal{E}_{I_0}$ be the collection of maximal dyadic intervals $I\in\mathcal{D}$ contained in the set $F_{I_0}$, then we claim that
   \begin{equation}\label{Tsigma}
|T_{\sigma}f(x)|\,\mathbbm{1}_{I_0}(x)\leq C_0\langle |f|\rangle_{I_0}+\sum_{I\in\mathcal{E}_{I_0}}|T_{\sigma}^I f(x)|
\end{equation}
where ${T_{\sigma}^If:={\sigma_{\widetilde{I}}}\langle f\rangle_I\mathbbm{1}_I+\sum_{J: J\subset I}\sigma_J\langle f, h_J\rangle h_J}$,  and  $\widetilde{I}$ is the parent of $I$.

 Repeat for each $I\in\mathcal{E}_{I_0}$ and the function $T^I_{\sigma}f$ which is supported on $I$, then repeat  for each $I'\in\mathcal{E}_I$, etc. 
Let  $\mathcal{S}_0=\{I_0\}$, and 
$\mathcal{S}_j := \cup_{I\in\mathcal{S}_{j-1}}  \mathcal{E}_{I}$. Finally let $\mathcal{S}:=\cup_{j=0}^{\infty}\mathcal{S}_j$. For each $I\in\mathcal{S}$, let $E_I=I\setminus F_I$, by construction  the sets $E_I\subset I$ are pairwise disjoint and $ |E_I|\geq \frac{1}{2} |I|$, therefore  $\mathcal{S}$ is a $\frac{1}{2}$-sparse family.  Moreover
$$|\mathbbm{1}_{I_0}T_{\sigma}f| \leq C_0 \mathcal{A}_{\mathcal{S}}|f|,$$
which is what we set out to prove. We are done modulo verifying the claimed inequality \eqref{Tsigma}, which we now prove. 
Note that $|T_{\sigma}f(x)| \leq 2T_{\sigma}^{\sharp}f(x)$.  Thus, if {$x\in I_0\setminus F_{I_0}$} then
$|T_{\sigma}f(x)| \leq  C_0\langle |f|\rangle_{I_0}$, and \eqref{Tsigma} is satisfied.

 If $x\in F_{I_0}$ then there is unique $I\in \mathcal{S}_1=\mathcal{E}_{I_0}$ with {$x\in I$}, and  recalling that 
 $\langle f, h_{\widetilde{I}} \rangle h_{\widetilde{I}}(x)=\langle f\rangle_I -\langle f\rangle_{\widetilde I}$, we conclude that
\begin{eqnarray*} T_\sigma f(x) & =  &
\sum_{J\supsetneq \widetilde{I}} \sigma_J \langle f, h_J\rangle h_J(x) 
+\sum_{ J\subset \widetilde{I}} \sigma_J \langle f, h_J\rangle h_J(x) \\
& = & \sum_{J\supsetneq \widetilde{I}} \sigma_J \langle f, h_J\rangle h_J(x) - \sigma_{\widetilde I}  \langle f\rangle_{\widetilde{I}} + T_\sigma^If(x).
\end{eqnarray*}
Therefore we find that  when $x\in F_{I_0}$ and for all $y\in \widetilde{I}$ the following inequality holds
\begin{equation}\label{eqn:Lacey}
 |T_\sigma f(x)| \leq   
T_{\sigma}^{\sharp}f(y) + Mf(y) +\sum_{I\in\mathcal{E}_{I_0}} T^I_{\sigma}f(x).
\end{equation}
In particular,   because $I$ is a maximal dyadic interval in $F_{I_0}$,  there must be $y_0\in \widetilde{I}\setminus I$ such that $y_0\notin F_{I_0}$ and therefore $T_{\sigma}^{\sharp}f(y_0) + Mf(y_0) \leq  \frac{1}{2} C_0 \langle |f|\rangle_{I_0}$.  Substituting $y=y_0$ in \eqref{eqn:Lacey}, and using this estimate proves the claimed inequality \eqref{Tsigma}, and therefore the pointwise localized domination by sparse operators for the martingale transform is proven.
\end{proof}

\subsection{Case study: Sparse operators vs  commutators}\label{sec:Sparse-vs-commutators}

Carlos P\'erez and Israel Rivera-R\'ios proposed the following  $L\log L$-sparse operator  as a candidate for sparse domination of the commutator. $$ B_{\mathcal{S}} f(x) = \sum_{Q\in \mathcal{S}} {\|f\|_{L\log L, Q}} \mathbbm{1}_Q(x).$$
The reason for this choice is that $M^2 \sim M_{L\log L}$ is the correct maximal function for the commutator.
However they showed that these operators {cannot bound pointwise} the commutator $[b,T]$ in \cite{PzR}.

Andrei  Lerner, Sheldy Ombrosi, and  Israel Rivera-R\'ios   proposed the following  sparse-like operator and its adjoint adapted to the commutator with locally integrable function $b$,
 \begin{eqnarray*}
\mathcal{T}_{\mathcal{S},b}f(x)& := &\sum_{Q\in\mathcal{S}}{  |b(x)-\langle b\rangle_Q|\,\langle|f|\rangle_Q }\,\mathbbm{1}_Q(x), \\
\mathcal{T}^*_{\mathcal{S},b}f(x)& := &\sum_{Q\in\mathcal{S}}{ \langle |b-\langle b\rangle_Q|\,|f|\rangle_Q }\,\mathbbm{1}_Q(x).
\end{eqnarray*}
They showed,  in \cite{LeOR1},
 that finitely many of these operators will provide pointwise domination for the commutator, $[b,T]$, where $T$ is a rough Calder\'on-Zygmund operator and $b$ a locally integrable function.

\begin{theorem}[Lerner, Ombrosi, Rivera-R\'ios 2017]
Let $T$ be an $\omega$-Calder\'on-Zygmund singular integral  operator with $\omega$ satisfying a Dini condition, $b\in L^1_{{\rm loc}}(\R^d )$. For every compactly supported $f\in L^{\infty}(\R^d )$, there are $3^n$ dyadic lattices $\mathcal{D}^{(k)}$ and 
$\frac{1}{2\cdot 9^n}$-sparse families $\mathcal{S}_k\subset \mathcal{D}^{(k)}$ such that for a.e. $x\in\R^d$
\[ |[b,T](f)(x)|\lesssim_{d,T}\sum_{k=1}^{3^n}\big (\mathcal{T}_{\mathcal{S}_k,b}|f|(x) + \mathcal{T}^*_{\mathcal{S}_k,b}|f|(x)\big ).\]
\end{theorem}

 Quadratic bounds on $L^2(w)$ for the commutator $[b,T]$  will follow from quadratic bounds for these adapted sparse operators \cite{LeOR1}.
 The following quadratic bounds on $L^2(w)$ for $\mathcal{T}_{\mathcal{S},b}$ , $\mathcal{T}^*_{\mathcal{S},b}$ hold,
$$ \|\mathcal{T}_{\mathcal{S},b}f\|_{L^2(w)}+\|\mathcal{T}^*_{\mathcal{S},b}f\|_{L^2(w)} \lesssim [w]_{A_2}^2  \|b\|_{{\rm BMO}} \|f\|_{L^2(w)}.$$
These quadratic bounds, the  corresponding extrapolated bounds on $L^p(w)$
$$ \|\mathcal{T}_{\mathcal{S},b}f\|_{L^p(w)}+\|\mathcal{T}^*_{\mathcal{S},b}f\|_{L^p(w)} \lesssim_p [w]_{A_p}^{2\max\{1,\frac{1}{p-1}\}}  \|b\|_{{\rm BMO}} \|f\|_{L^p(w)},$$
and  much more follow from a key lemma that we now state.

\begin{lemma}[Lerner, Ombrosi, Rivera-R\'ios 2017]\label{lem:LOR-R}
Given $\mathcal{S}$ an $\eta$-sparse family in $\mathcal{D}$ , $b\in L^1_{{\rm loc}}(\R^d)$ then there is a larger collection $\mathcal{\widetilde{S}}\in\mathcal{D}$ which is an $\frac{\eta}{2(1+\eta)}$-sparse family, $\mathcal{S}\subset\mathcal{\widetilde{S}}$, such that  for all $Q\in \mathcal{\widetilde{S}}$,   the following estimate holds
$$ |b(x)-\langle b\rangle_Q|\leq 2^{d+2}\sum_{R\in\mathcal{\widetilde{S}}, R\subset Q}\Omega(b; R)\mathbbm{1}_R(x), \quad \mbox{a.e.}\; x\in Q,$$
where  $\Omega(b; R):= \frac{1}{|R|}\int_R|b(x)-\langle b\rangle_R|\,dx$, the mean oscillation of $b$ on the dyadic cube $R$.
\end{lemma}

From this lemma we immediately deduce quantitative Bloom bounds for the sparse-like adjoint operator associated to the commutator \cite{LeOR1}. A similar result holds for $\mathcal{T}_{\mathcal{S},b}$.         

\begin{corollary}[Quantitative Bloom]
Let $u,v\in A_p$, $\mu = u^{1/p}v^{-1/p}$ and $b\in {\rm BMO}_{\mu}$ then there is a constant $c_{d,p}>0$ such that for all $f\in L^p(u)$ the following inequality holds,
$$\|\mathcal{T}_{\mathcal{S},b}^*|f|\|_{L^p(v)}\leq 
{c_{d,p} \|b\|_{{\rm BMO}_{\mu}} \big ( [v]_{A_p}[u]_{A_p}\big )^{\max\{1,\frac{1}{p-1}\}} \|f\|_{L^p(u)}}.$$
Similarly for $\mathcal{T}_{\mathcal{S},b}$.
\end{corollary}

\begin{proof} First notice that since $\|b\|_{{\rm BMO}_{\mu}} =\sup_Q |Q|\, \Omega(b;Q)/\mu (Q) $, 
$$\mathcal{T}_{\mathcal{\widetilde{S}},b}^*|f|(x)\leq  c_d\|b\|_{{\rm BMO}_{\mu}} \mathcal{A}_{\mathcal{\widetilde{S}}}\big ( \mathcal{A}_{\mathcal{\widetilde{S}}} (|f|)\mu\big )(x),$$ 
where $\mathcal{\widetilde{S}}$ is the larger sparse family given by Lemma~\ref{lem:LOR-R}.

Taking $L^p(v)$ norm on both sides, and unfolding we conclude that
\begin{align*}
\|\mathcal{T}_{\mathcal{\widetilde{S}},b}^*|f|\|_{L^p(v)} & \leq  \;c_{d,p} \|b\|_{{\rm BMO}_{\mu}} 
\|\mathcal{A}_{\mathcal{\widetilde{S}}}\|_{L^p(v)}\|\mathcal{A}_{\mathcal{\widetilde{S}}}\|_{L^p(u)} \|f\|_{L^p(u)}\\
&\leq c_{d,p} \|b\|_{{\rm BMO}_{\mu}} \big ( [v]_{A_p}[u]_{A_p}\big )^{\max\{1,\frac{1}{p-1}\}} \|f\|_{L^p(u)},
\end{align*}
where in the last line we used the one-weight estimates on both $L^p(u)$ and $L^p(v)$ for the sparse operator
 $\mathcal{A}_{\mathcal{\widetilde{S}}}$ given that $u$ and $v$ are $A_p$ weights by assumption.  Observing that $\mathcal{T}_{\mathcal{S},b}^*|f|(x)\leq \mathcal{T}_{\mathcal{\widetilde{S}},b}^*|f|(x)$ we get the desired estimate. 
\end{proof}

Setting  $u=v=w\in A_p$, then $\mu\equiv 1$, $b\in {\rm BMO}$, and  we recover the expected one-weight  
quantitative  $L^p$ estimates for the sparse-like  operators dominating the commutator, and hence for the commutator itself, without  using extrapolation, 
$$\|\mathcal{T}_{\mathcal{S},b}|f|\|_{L^p(w)} +\|\mathcal{T}_{\mathcal{S},b}^*|f|\|_{L^p(w)}   \leq 
c_{n,p}\|b\|_{{\rm BMO}}  [w]_{A_p}^{2\max\{1,\frac{1}{p-1}\}}\|f\|_{L^p(w)}.$$

\section{Summary and recent progress}\label{sec:Recent-progress}

In these lecture notes we  have studied  weighted norm  inequalities  through the dyadic harmonic analysis lens.
We focused on classical operators such as the Hilbert transform and the maximal function, and dyadic operators such as the dyadic maximal function, the martingale transform, the dyadic square function, Haar shift multipliers,  the dyadic paraproduct, and the latest "kid in the block" the dyadic sparse operator.  To carry on our program, we  discussed  dyadic tools such as dyadic cubes (regular, random, adjacent) and Haar functions on $\R$, $\R^d$, and more generally on spaces of homogenenous type.

In this millennium  the interest shifted from qualitative weighted norm inequalities to quantitative weighted norm inequalities. New techniques were developed to obtain quantitative estimates, including Bellman function and median oscillation techniques,  quantitative extrapolation and transference theorems, corona decompositions and stopping times, representation of operators as averages of dyadic operators, and, most recently, domination by dyadic sparse operators.  One important landmark in this quest was the proof of the $A_2$ conjecture.
Some of  these techniques  are amenable to  generalizations to other settings that support dyadic structures such as spaces of homogeneous type.  

We  tried to illustrate the power of the dyadic methods studying in detail the maximal function  and  the commutator of the Hilbert transform  with a function in ${\rm BMO}$  via their dyadic counterparts,  in both cases obtaining the optimal estimates on weighted Lebesgue spaces. We  presented a self-contained Bellman function  proof of the $A_2$ conjecture for the dyadic paraproduct,  in order to illustrate these technique. We  showed how to pointwise dominate the martingale transform by sparse operators, and we presented the beautiful and simple proof of the $A_2$ conjecture for sparse operators. We illustrated the power of pointwise domination techniques by sparse-like operators through a case study:  the commutator of Calder\'on-Zygmund singular integral operators and  locally integrable functions, recovering all the quantitative weighted norm inequalities discussed in the notes, and some new ones.

The methods developed in this millennium, initially to study quantitative weighted inequalities for operators defined on $\R^d$,  have proven to be quite flexible and far reaching. 
There are extensions to {metric spaces with geometrically doubling condition}, {spaces of homogeneous type}, and beyond doubling even in a non-commutative setting of operator-valued dyadic harmonic analysis
\cite{Hyt4, NRV, KLPW, LoMaPa, DGKLWY, ThTV, CLo}. There are 
off-diagonal sharp two-weight estimates for sparse operators \cite{FaHyt}.
There are  generalizations to {matrix valued operators} \cite{IKP}, so far the best weighted $L^2$ estimates in this setting are 3/2 powers for the matrix-valued paraproducts, shift operators, and Calder\'on-Zygmund operators satisfying a Dini condition  \cite{NPetTV},  and  linear for the square function \cite{HPetV}. The validity of the $A_2$ conjecture in the matrix setting is unknown.  Two-weight estimates have been obtained for well localized operators  with matrix weights \cite{BiCuTW} and a weighted Carleson embedding theorem  with matrix weights  is known and proved using a "Bellman function with a parameter" \cite{CuT} . Researchers are busy working towards increasing our knowledge on this setting, see for example \cite{CuPetPo} where  a bilinear Carleson embedding theorem with matrix weight and scalar measure is proved using Bellman function techniques.

More importantly, out of these investigations a domination paradigm by {sparse positive dyadic operators} has emerged and proven to be very powerful with applications in many areas not only weighted inequalities. The following is a partial  and ever-growing  list of such applications to: (maximal) rough singular integrals  \cite{CuDiPOu1, CCuDiPOu, HRoTa, DiPHLi}; singular non-integral operators 
\cite{BFPet}; multilinear maximal and singular integral operators 
\cite{CuDiPOu2, LeN, BMu1, Z};  non-homogeneous spaces and  operator-valued  singular integral operators 
 \cite{CPa,VZ}; uncentered variational operators  \cite{deFZ};  variational Carleson operators  
\cite{DiPDoU};  Walsh-Fourier multipliers \cite{CuDiPLOu}; Bochner-Riesz multipliers 
\cite{BBLu, LMR, KL}; maximally truncated oscillatory  singular integral operators  
\cite{KrL1,KrL3}; spherical maximal function   \cite{L5}; Radon transform  
 \cite{Ob}; Hilbert transform along curves   \cite{ClOu}; pseudodifferential operators 
\cite{BCl}; the lattice Hardy-Littlewood maximal operator \cite{HaLo}; fractional operator with $L^{\alpha,r'}$-H\"ormander conditions \cite{IbRiVi}; Rubio de Francia's Littlewood--Paley square function \cite{GRS}. Sparse $T(1)$ theorems \cite{LM} and applications in the  discrete setting  
\cite{KrL2, KMe, CuKL} have been found as well as logarithmic bounds for  maximal sparse operators \cite{KaL}.

We are starting to understand why in certain settings this philosophy does not work. For example
very recently  it was shown that  dominating the dyadic strong maximal function by (1,1)-type sparse forms based on rectangles with sides parallel to the axes  is impossible \cite{BaCOuR}, this is in the realm of multiparameter analysis were many questions  still need to be answered. Perhaps a new type of sparse domination in this setting will have to be dreamed.

Not only the methodology is tried on each author's favorite operator, far reaching extensions and broader understanding is being gained. For example, the convex body domination paradigm  
\cite{NPetTV} shows  that if a scalar operator can be dominated by a sparse operator, then its vector version can be dominated by a convex body valued sparse operator, a transference theorem.
Similarly, 
multiple vector-valued extensions of operators  and more can be explained through the very general  helicoidal method \cite{BMu2}, yet another far-reaching transference methodology.

This is a very 
active area of research and we hope this lecture notes have helped to  impress on the reader its vitality.


\end{document}